\documentclass[12pt,a4paper,reqno]{amsart}
\title[$\Ainfty$-structures in monoidal DG cats \& strong
homotopy unitality]{$\Ainfty$-structures in monoidal DG categories and \\  strong homotopy unitality}
\author{Rina Anno}
\email{ranno@math.ksu.edu}
\address{Department of Mathematics \\
Kansas State University \\
138 Cardwell Hall \\
Manhattan, KS 66506\\
USA}
\author{Sergey Arkhipov}
\email{hippie@math.au.dk}
\address{Matematisk Institut, Aarhus Universitet, Ny Munkegade, DK-8000,
Aarhus C, Denmark}
\author{Timothy Logvinenko} 
\email{LogvinenkoT@cardiff.ac.uk} 
\address{School of Mathematics\\ 
Cardiff University\\
Senghennydd Road,\\
Cardiff, CF24 4AG\\
UK}
\usepackage{amsmath,amsfonts,amssymb,amsthm,epsfig,amscd,latexsym,comment}
\usepackage{caption}
\usepackage{MnSymbol}
\usepackage{tikz}
\usetikzlibrary{cd}
\usepackage{graphicx}
\usepackage{array}
\usepackage{subfigure}
\usepackage{leftidx}
\usepackage{xparse}% http://ctan.org/pkg/xparse

\usepackage[colorlinks=true, pdfpagemode=none, pdfmenubar=false, linkcolor=blue, citecolor=blue, urlcolor=blue]{hyperref}

\let\amsamp=&

\begingroup
\catcode`\&=13
\gdef\smallampmatrix{%
  \begingroup
  \let&=\amsamp
  \begin{smallmatrix}%
}
\gdef\endsmallampmatrix{\end{smallmatrix}\endgroup}
\endgroup

\DeclareMathOperator{\obj}{Ob}
\DeclareMathOperator{\mor}{Mor}
\DeclareMathOperator{\homm}{Hom}

\DeclareMathOperator{\eend}{End}

\DeclareMathOperator{\picr}{Pic}

\DeclareMathOperator{\cl}{Cl}

\DeclareMathOperator{\ext}{Ext}

\DeclareMathOperator{\eval}{ev}

\DeclareMathOperator{\composition}{cmps}
\DeclareMathOperator{\action}{act}

\DeclareMathOperator{\modd}{\bf Mod}

\DeclareMathOperator{\lder}{\bf L}

\DeclareMathOperator{\ldertimes}{\overset{\lder}{\otimes}}

\DeclareMathOperator{\id}{Id}

\DeclareMathOperator{\vectspaces}{\bf Vect}

\DeclareMathOperator{\opp}{{opp}}
\DeclareMathOperator{\fg}{{\it fg}}
\DeclareMathOperator{\qrep}{\it \mathcal{Q}r}
\DeclareMathOperator{\hproj}{\mathcal{P}}

\DeclareMathOperator{\acyc}{\it \mathcal{A}c}

\DeclareMathOperator{\semifree}{\mathcal{S}\mathcal{F}}
\DeclareMathOperator{\sffg}{\mathcal{S}\mathcal{F}_{\fg}}
\DeclareMathOperator{\perf}{{\it \mathcal{P}erf}}
\DeclareMathOperator{\hperf}{{\it h\mathcal{P}erf}}
\DeclareMathOperator{\hmtpy}{{Ho}}

\DeclareMathOperator{\tria}{{Tria}}
\DeclareMathOperator{\cx}{{Cx}}
\DeclareMathOperator{\twcx}{{Tw}}
\DeclareMathOperator{\twbicx}{{Twbi}}
\DeclareMathOperator{\pretriag}{{Pre\text{-}Tr}}
\DeclareMathOperator{\DGFun}{{DGFun}}
\DeclareMathOperator{\DGFuntwocat}{{\bf DGFun}}

\DeclareMathOperator{\TPair}{{\bf TPair}}
\DeclareMathOperator{\alg}{{\bf Alg}}

\DeclareMathOperator{\conv}{{Conv}}

\DeclareMathOperator{\unit}{{unit}}
\DeclareMathOperator{\counit}{{counit}}

\DeclareMathOperator{\forget}{{\bf Forg}}
\DeclareMathOperator{\free}{{\bf Free}}

\DeclareMathOperator{\DGModtwocat}{{\bf{DGMod}}}
\DeclareMathOperator{\BarModtwocat}{\bf{DG\overline{Mod}}}
\DeclareMathOperator{\yoneda}{Yoneda}
\DeclareMathOperator{\eilmoor}{{\mathcal{E}}}
\DeclareMathOperator{\coeilmoor}{{co\mathcal{E}}}
\DeclareMathOperator{\kleisli}{{\mathcal{K}\mathcal{S}}}
\DeclareMathOperator{\cokleisli}{{co\mathcal{K}\mathcal{S}}}
\DeclareMathOperator{\cxrow}{{\mathcal{C}xrow}}
\DeclareMathOperator{\cxcol}{{\mathcal{C}xcol}}

\begin{document}

\def\bv{\mathbf{v}}
\def\kgc_{K^*_G(\mathbb{C}^n)}
\def\kgchi_{K^*_\chi(\mathbb{C}^n)}
\def\kgcf_{K_G(\mathbb{C}^n)}
\def\kgchif_{K_\chi(\mathbb{C}^n)}
\def\gpic_{G\text{-}\picr}
\def\gcl_{G\text{-}\cl}
\def\trch_{{\chi_{0}}}
\def\regring{{R}}
\def\regrep{{V_{\text{reg}}}}
\def\givrep{{V_{\text{giv}}}}
\def\lbar{{(\mathbb{Z}^n)^\vee}}
\def\genpx_{{p_X}}
\def\genpy_{{p_Y}}
\def\genpcn_{p_{\mathbb{C}^n}}
\def\gnat{gnat}
\def\twalg{{\regring \rtimes G}}
\def\L{{\mathcal{L}}}
\def\gcd{\mbox{gcd}}
\def\lcm{\mbox{lcm}}
\def\tf{{\tilde{f}}}
\def\tD{{\tilde{D}}}
\def\A{{\mathcal{A}}}
\def\B{{\mathcal{B}}}
\def\C{{\mathcal{C}}}
\def\D{{\mathcal{D}}}
\def\E{{\mathcal{E}}}
\def\F{{\mathcal{F}}}
\def\H{{\mathcal{H}}}
\def\L{{\mathcal{L}}}
\def\M{{\mathcal{M}}}
\def\N{{\mathcal{N}}}
\def\R{{\mathcal{R}}}
\def\T{{\mathcal{T}}}
\def\RF{{\mathcal{R}\mathcal{F}}}
\def\barA{{\bar{\mathcal{A}}}}
\def\barAi{{\bar{\mathcal{A}}_1}}
\def\barAj{{\bar{\mathcal{A}}_2}}
\def\barB{{\bar{\mathcal{B}}}}
\def\barC{{\bar{\mathcal{C}}}}
\def\barD{{\bar{\mathcal{D}}}}
\def\barT{{\bar{\mathcal{T}}}}
\def\barM{{\bar{\mathcal{M}}}}
\def\Aopp{{\A^{\opp}}}
\def\Bopp{{\B^{\opp}}}
\def\Copp{{\C^{\opp}}}
\def\aA{\leftidx{_{a}}{\A}}
\def\bA{\leftidx{_{b}}{\A}}
\def\Aa{{\A_a}}
\def\Ea{E_a}
\def\aE{\leftidx{_{a}}{E}{}}
\def\Eb{E_b}
\def\bE{\leftidx{_{b}}{E}{}}
\def\Fa{F_a}
\def\aF{\leftidx{_{a}}{F}{}}
\def\Fb{F_b}
\def\bF{\leftidx{_{b}}{F}{}}
\def\aM{\leftidx{_{a}}{M}{}}
\def\aN{\leftidx{_{a}}{N}{}}
\def\aMb{\leftidx{_{a}}{M}{_{b}}}
\def\aNb{\leftidx{_{a}}{N}{_{b}}}
\def\rfRFa{\leftidx{_{\RF}}{\RF}{_{\A}}}
\def\aRFrf{\leftidx{_{\A}}{\RF}{_{\RF}}}
\def\biAMA{\leftidx{_{\A}}{M}{_{\A}}}
\def\biAMC{\leftidx{_{\A}}{M}{_{\C}}}
\def\biCMA{\leftidx{_{\C}}{M}{_{\A}}}
\def\biCMC{\leftidx{_{\C}}{M}{_{\C}}}
\def\biALA{\leftidx{_{\A}}{L}{_{\A}}}
\def\biALC{\leftidx{_{\A}}{L}{_{\C}}}
\def\biCLA{\leftidx{_{\C}}{L}{_{\A}}}
\def\biCLC{\leftidx{_{\C}}{L}{_{\C}}}
\def\Na{{N_a}}
\def\vectk{{\vectspaces\text{-}k}}
\def\vectkfg{{\vectspaces_{\text{fg}}\text{-}k}}
\def\modk{{\modd\text{-}k}}
\def\Amod{{\A\text{-}\modd}}
\def\modA{{\modd\text{-}\A}}
\def\modbar{{\overline{\modd}}}
\def\modbarA{{\overline{\modd}\text{-}\A}}
\def\modbarAopp{{\overline{\modd}\text{-}\Aopp}}
\def\modB{{\modd\text{-}\B}}
\def\modC{{\modd\text{-}\C}}
\def\modD{{\modd\text{-}\D}}
\def\modbarB{{\overline{\modd}\text{-}\B}}
\def\modbarC{{\overline{\modd}\text{-}\C}}
\def\modbarD{{\overline{\modd}\text{-}\D}}
\def\modbarBopp{{\overline{\modd}\text{-}\Bopp}}
\def\kmodk{{k\text-\modd\text{-}k}}
\def\bikmodk{{k_\bullet\text-\modd\text{-}k_\bullet}}
\def\AmodA{{\A\text{-}\modd\text{-}\A}}
\def\AmodM{{\A\text{-}\modd\text{-}\M}}
\def\AmodB{{\A\text{-}\modd\text{-}\B}}
\def\AmodT{{\A\text{-}\modd\text{-}\T}}
\def\BmodB{{\B\text{-}\modd\text{-}\B}}
\def\BmodA{{\B\text{-}\modd\text{-}\A}}
\def\DmodD{{\D\text{-}\modd\text{-}\D}}
\def\MmodA{{\M\text{-}\modd\text{-}\A}}
\def\MmodM{{\M\text{-}\modd\text{-}\M}}
\def\TmodA{{\T\text{-}\modd\text{-}\A}}
\def\TmodT{{\T\text{-}\modd\text{-}\T}}
\def\AmodbarA{\A\text{-}{\overline{\modd}\text{-}\A}}
\def\AmodbarB{\A\text{-}{\overline{\modd}\text{-}\B}}
\def\AmodbarC{\A\text{-}{\overline{\modd}\text{-}\C}}
\def\AmodbarD{\A\text{-}{\overline{\modd}\text{-}\D}}
\def\AmodbarM{\A\text{-}{\overline{\modd}\text{-}\M}}
\def\AmodbarT{\A\text{-}{\overline{\modd}\text{-}\T}}
\def\BmodbarA{\B\text{-}{\overline{\modd}\text{-}\A}}
\def\BmodbarB{\B\text{-}{\overline{\modd}\text{-}\B}}
\def\BmodbarC{\B\text{-}{\overline{\modd}\text{-}\C}}
\def\BmodbarD{\B\text{-}{\overline{\modd}\text{-}\D}}
\def\CmodbarA{\C\text{-}{\overline{\modd}\text{-}\A}}
\def\CmodbarB{\C\text{-}{\overline{\modd}\text{-}\B}}
\def\CmodbarC{\C\text{-}{\overline{\modd}\text{-}\C}}
\def\CmodbarD{\C\text{-}{\overline{\modd}\text{-}\D}}
\def\DmodbarA{\D\text{-}{\overline{\modd}\text{-}\A}}
\def\DmodbarB{\D\text{-}{\overline{\modd}\text{-}\B}}
\def\DmodbarC{\D\text{-}{\overline{\modd}\text{-}\C}}
\def\DmodbarD{\D\text{-}{\overline{\modd}\text{-}\D}}
\def\TmodbarA{\T\text{-}{\overline{\modd}\text{-}\A}}
\def\MmodbarA{\M\text{-}{\overline{\modd}\text{-}\A}}
\def\MmodbarM{\M\text{-}{\overline{\modd}\text{-}\M}}
\def\modbarT{{\overline{\modd}\text{-}T}}
\def\freeA{{\free\text{-}A}}
\def\freeAS{{\free_S\text{-}A}}
\def\freepfA{{\free_{pf}\text{-}A}}
\def\Afree{{A\text{-}\free}}
\def\AfreeA{{A\text{-}\free\text{-}A}}
\def\AfreeB{{A\text{-}\free\text{-}B}}
\def\sfA{{\semifree(\A)}}
\def\sfB{{\semifree(\B)}}
\def\sffgA{{\sffg(\A)}}
\def\sffgB{{\sffg(\B)}}
\def\hprojA{{\hproj(\A)}}
\def\hprojB{{\hproj(\B)}}
\def\qrepA{{\qrep(\A)}}
\def\qrepB{{\qrep(\B)}}
\def\opp{{\text{opp}}}
\def\Aperf{{\A^{\text{pf}}}}
\def\hperfA{{\hperf(\A)}}
\def\hperfB{{\hperf(\B)}}
\def\barperf{{\it\mathcal{P}\overline{er}f}}
\def\barperfA{{\barperf\text{-}\A}}
\def\barperfB{{\barperf\text{-}\B}}
\def\barperfC{{\barperf\text{-}\C}}
\def\qrhpr{{\hproj^{qr}}}
\def\qrhprA{{\qrhpr(\A)}}
\def\qrhprB{{\qrhpr(\B)}}
\def\qrsf{{\semifree^{qr}}}
\def\qrsf{{\semifree^{qr}}}
\def\qrsfA{{\qrsf(\A)}}
\def\qrsfB{{\qrsf(\B)}}
\def\Aperfsf{{\semifree^{\A\text{-}\perf}(\AbimB)}}
\def\Bperfsf{{\semifree^{\B\text{-}\perf}(\AbimB)}}
\def\Aprfhpr{{\hproj^{\A\text{-}\perf}(\AbimB)}}
\def\Bprfhpr{{\hproj^{\B\text{-}\perf}(\AbimB)}}
\def\Aqrhpr{{\hproj^{\A\text{-}qr}(\AbimB)}}
\def\Bqrhpr{{\hproj^{\B\text{-}qr}(\AbimB)}}
\def\Aqrsf{{\semifree^{\A\text{-}qr}(\AbimB)}}
\def\Bqrsf{{\semifree^{\B\text{-}qr}(\AbimB)}}
\def\modAopp{{\modd\text{-}\Aopp}}
\def\modBopp{{\modd\text{-}\Bopp}}
\def\AmodA{{\A\text{-}\modd\text{-}\A}}
\def\AmodB{{\A\text{-}\modd\text{-}\B}}
\def\AmodC{{\A\text{-}\modd\text{-}\C}}
\def\BmodA{{\B\text{-}\modd\text{-}\A}}
\def\BmodB{{\B\text{-}\modd\text{-}\B}}
\def\BmodC{{\B\text{-}\modd\text{-}\C}}
\def\CmodA{{\C\text{-}\modd\text{-}\A}}
\def\CmodB{{\C\text{-}\modd\text{-}\B}}
\def\CmodC{{\C\text{-}\modd\text{-}\C}}
\def\CmodD{{\C\text{-}\modd\text{-}\D}}
\def\AbimA{{\A\text{-}\A}}
\def\AbimC{{\A\text{-}\C}}
\def\AbimM{{\A\text{-}\M}}
\def\BbimA{{\B\text{-}\A}}
\def\BbimB{{\B\text{-}\B}}
\def\BbimC{{\B\text{-}\C}}
\def\BbimD{{\B\text{-}\D}}
\def\CbimA{{\C\text{-}\A}}
\def\CbimB{{\C\text{-}\B}}
\def\CbimC{{\C\text{-}\C}}
\def\DbimA{{\D\text{-}\A}}
\def\DbimB{{\D\text{-}\B}}
\def\DbimC{{\D\text{-}\C}}
\def\DbimD{{\D\text{-}\D}}
\def\MbimA{{\M\text{-}\A}}
\def\AhprA{{\hproj\left(\AbimA\right)}}
\def\BhprB{{\hproj\left(\BbimB\right)}}
\def\AhprB{{\hproj\left(\AbimB\right)}}
\def\BhprA{{\hproj\left(\BbimA\right)}}
\def\AbarA{{\overline{\A\text{-}\A}}}
\def\AbarB{{\overline{\A\text{-}\B}}}
\def\BbarA{{\overline{\B\text{-}\A}}}
\def\BbarB{{\overline{\B\text{-}\B}}}
\def\QAbimB{{Q\A\text{-}\B}}
\def\AbimB{{\A\text{-}\B}}
\def\AbimC{{\A\text{-}\C}}
\def\AonebimB{{\A_1\text{-}\B}}
\def\AtwobimB{{\A_2\text{-}\B}}
\def\BbimA{{\B\text{-}\A}}
\def\MddA{{M^{\tilde{\A}}}}
\def\MddB{{M^{\tilde{\B}}}}
\def\MhdA{{M^{h\A}}}
\def\MhdB{{M^{h\B}}}
\def\NhdB{{N^{h\B}}}
\def\Cat{{\it \mathcal{C}at}}
\def\twoCat{{\it 2\text{-}\;\mathcal{C}at}}
\def\DGCat{{DG\text{-}Cat}}
\def\HoDGCat{{\hmtpy(\DGCat)}}
\def\HoDGCatV{{\hmtpy(\DGCat_\mathbb{V})}}
\def\MoDGCat{{Mo(\DGCat)}}
\def\tr{{tr}}
\def\pretr{{pretr}}
\def\kctr{{kctr}}
\def\PreTrCat{{\DGCat^\pretr}}
\def\KcTrCat{{\DGCat^\kctr}}
\def\HoPretrCat{{\hmtpy(\PreTrCat)}}
\def\HoKcTrCat{{\hmtpy(\KcTrCat)}}
\def\Aquasirep{{\A\text{-}qr}}
\def\QAquasirep{{Q\A\text{-}qr}}
\def\Bquasirep{{\B\text{-}qr}} 
\def\lderA{{\tilde{\A}}} 
\def\lderB{{\tilde{\B}}} 
\def\adjunit{{\text{adj.unit}}}
\def\adjcounit{{\text{adj.counit}}}
\def\degzero{{\text{deg.0}}}
\def\degone{{\text{deg.1}}}
\def\degminusone{{\text{deg.-$1$}}}
\def\barzeta{{\overline{\zeta}}}
\def\Ract{{R {\action}}}
\def\barRact{{\overline{\Ract}}}
\def\actL{{{\action} L}}
\def\baractL{{\overline{\actL}}}
\def\Ainfty{{A_{\infty}}}
\def\moddinf{{{\bf Mod}_{\infty}}}
\def\nodd{{{\bf Nod}}}
\def\noddinf{{{\bf Nod}_{\infty}}}
\def\perfinf{{{\bf Perf}_{\infty}}}
\def\conodd{{{\bf coNod}}}
\def\comodd{{{\bf coMod}}}
\def\conoddinf{{{\bf coNod}_{\infty}}}
\def\conodddgstrict{{{\bf coNod}_{dg}^{strict}}}
\def\conodddg{{{\bf coNod}_{dg}}}
\def\conodddghu{{{\bf coNod}_{dg}^{hu}}}
\def\conoddinfshu{{{\bf coNod}_{\infty}^{shu}}}
\def\noddinfstr{{{\bf Nod}^{\text{strict}}_{\infty}}}
\def\noddinfA{{\noddinf\A}}
\def\noddinfB{{\noddinf\B}}
\def\perfinfA{{\perfinf\A}}
\def\perfinfB{{\perfinf\B}}
\def\nodA{{\noddinf\text{-}A}}
\def\nodstrA{{\nodd\text{-}A}}
\def\nodstrhuA{{\nodd^{hu}\text{-}A}}
\def\nodhuA{{\noddinfhu\text{-}A}}
\def\nodhupfA{{\noddinfhupf\text{-}A}}
\def\nodB{{\noddinf\text{-}B}}
\def\nodAbic{{\left(\nodA\right)^{\text{bicat}}}}
\def\Anodbic{{\left(\Anod\right)^{\text{bicat}}}}
\def\nodBbic{{\left(\nodB\right)^{\text{bicat}}}}
\def\Bnodbic{{\left(\Bnod\right)^{\text{bicat}}}}
\def\repA{{A^{\modA}}}
\def\nodrepA{{\noddinf\text{-}\repA}}
\def\pretrA{{A^{\pretriag\A}}}
\def\nodpretrA{{\noddinf\text{-}\pretrA}}
\def\twcxA{{A^{\twcxub\A}}}
\def\nodtwcxA{{\noddinf\text{-}\twcxA}}
\def\twcxrepA{{A^{\twcxub\modA}}}
\def\nodtwcxrepA{{\noddinf\text{-}\twcxrepA}}
\def\Anod{{A\text{-}\noddinf}}
\def\Bnod{{B\text{-}\noddinf}}
\def\AnodA{{A\text{-}\noddinf\text{-}A}}
\def\AnodB{{A\text{-}\noddinf\text{-}B}}
\def\noddinfAB{{\noddinf\AbimB}}
\def\noddinfBA{{\noddinf\BbimA}}
\def\noddinfu{{({\bf Nod}_{\infty})_u}}
\def\noddinfuA{{(\noddinfA)_u}}
\def\noddinfhu{{{\bf Nod}_{\infty}^{hu}}}
\def\noddinfhupf{{{\bf Nod}_{\infty}^{hupf}}}
\def\noddinfhuA{{(\noddinfA)_{hu}}}
\def\noddinfhuB{{(\noddinfB)_{hu}}}
\def\noddinfdg{{({\bf Nod}_{\infty})_{dg}}}
\def\noddinfdgA{{(\noddinfA)_{dg}}}
\def\noddinfdgAA{{(\noddinf\AbimA)_{dg}}}
\def\noddinfdgAB{{(\noddinf\AbimB)_{dg}}}
\def\noddinfdgB{{(\noddinfB)_{dg}}}
\def\moddinf{{\modd_{\infty}}}
\def\moddinfA{{\modd_{\infty}\A}}
\def\naug{{\text{na}}}
\def\infbar{B_\infty}
\def\infbarl{{B^{l}_\infty}}
\def\infbarr{{B^{r}_\infty}}
\def\infbarbi{{B^{bi}_\infty}}
\def\infbarnaug{{B^{\naug}_\infty}}
\def\infcobar{{CB_\infty}}
\def\infbarres{{\bar{B}_\infty}}
\def\infcobarres{{C\bar{B}_\infty}}
\def\infbarA{{B^A_\infty}}
\def\infbarB{{B^B_\infty}}
\def\inftimes{{\overset{\infty}{\otimes}}}
\def\infhom{{\overset{\infty}{\homm}}}
\def\barhom{{\rm H\overline{om}}}
\def\barend{{\overline{\eend}}}
\def\bartimes{{\;\overline{\otimes}}}
\def\bartimesA{{\;\overline{\otimes}_\A\;}}
\def\bartimesB{{\;\overline{\otimes}_\B\;}}
\def\bartimesC{{\;\overline{\otimes}_\C\;}}
\def\triaA{{\tria \A}}
\def\TPairdg{{\TPair^{dg}}}
\def\algA{{\alg(\A)}}
\def\Ainfty{{A_{\infty}}}
\def\gpmu{{\boldsymbol{\mu}}}
\def\odd{{\text{odd}}}
\def\even{{\text{even}}}
\def\twcxub{\twcx^{\pm}}
\def\twcxmns{\twcx^{-}}
\def\twcxpls{\twcx^{+}}
\def\twbicxub{\twbicx^{\pm}}
\def\twbicxmns{\twbicx^{-}}
\def\twbicxpls{\twbicx^{+}}
\def\twbios{\twbicx_{os}}
\def\twbiosub{{\twbicx_{os}^{\pm}}}
\def\twbiosmns{{\twbicx_{os}^{-}}}
\def\twbiospls{{\twbicx_{os}^{+}}}
\def\pretriagub{{\pretriag^{\pm}}}
\def\pretriagmns{{\pretriag^{-}}}
\def\pretriagpls{{\pretriag^{+}}}
\def\eilmoordg{\eilmoor^{\text{dg}}}
\def\hprojemdg{\hproj^{\text{dg}}}
\def\hperfemdg{\hperf^{\text{dg}}}
\def\coeilmoordg{\coeilmoor^{\text{dg}}}
\def\kleisliA{{\kleisli(A)}}
\def\kleisliB{{\kleisli(B)}}
\def\kleisliAS{{\kleisli_S(A)}}
\def\kleisliwk{\kleisli^{\text{wk}}}
\def\eilmoorwk{\eilmoor^{\text{wk}}}
\def\eilmoorwkpf{\eilmoor^{\text{wk,perf}}}
\def\eilmoorwmor{\eilmoor^{\text{dg,wkmor}}}
\def\coeilmoorwk{\coeilmoor^{\text{wk}}}
\def\coeilmoorwmor{\coeilmoor^{\text{dg,wkmor}}}
\def\dnat{{d_{\text{nat}}}}
\def\bareta{{\tilde{\eta}}}
\def\barepsilon{{\tilde{\epsilon}}}
\def\infalg{{\alg_\infty}}
\def\enhcatkc{{\bf EnhCat_{kc}}}
\def\enhcatkcdg{{\bf EnhCat_{kc}^{dg}}}

\def\conodA{{\conoddinf\text{-}A}}
\def\conodB{{\conoddinf\text{-}B}}
\def\conodC{{\conoddinf\text{-}C}}
\def\conodstrC{{\conodd\text{-}C}}
\def\Aconod{{A\text{-}\conoddinf}}
\def\Cconod{{C\text{-}\conoddinf}}
\def\CconodD{{C\text{-}\conoddinf\text{-}D}}
\def\conodhuA{{{\bf coNod}_{\infty}^{hu}\text{-}A}}
\def\conodshuA{{{\bf coNod}_{\infty}^{shu}\text{-}A}}
\def\conodhuB{{{\bf coNod}_{\infty}^{hu}\text{-}B}}
\def\conodshuB{{{\bf coNod}_{\infty}^{shu}\text{-}B}}
\def\conodhuC{{{\bf coNod}_{\infty}^{hu}\text{-}C}}
\def\conodshuC{{{\bf coNod}_{\infty}^{shu}\text{-}C}}
\def\cokleisliC{{\cokleisli(C)}}
\def\cokleisliD{{\cokleisli(D)}}
\def\cokleisliCS{{\cokleisli_S(C)}}

\theoremstyle{definition}
\newtheorem{defn}{Definition}[section]
\newtheorem*{defn*}{Definition}
\newtheorem{exmpl}[defn]{Example}
\newtheorem*{exmpl*}{Example}
\newtheorem{exrc}[defn]{Exercise}
\newtheorem*{exrc*}{Exercise}
\newtheorem*{chk*}{Check}
\newtheorem{remark}[defn]{Remark}
\newtheorem*{remark*}{Remark}
\theoremstyle{plain}
\newtheorem{theorem}{Theorem}[section]
\newtheorem*{theorem*}{Theorem}
\newtheorem{conj}[defn]{Conjecture}
\newtheorem*{conj*}{Conjecture}
\newtheorem{question}[defn]{Question}
\newtheorem*{question*}{Question}
\newtheorem{prps}[defn]{Proposition}
\newtheorem*{prps*}{Proposition}
\newtheorem{cor}[defn]{Corollary}
\newtheorem*{cor*}{Corollary}
\newtheorem{lemma}[defn]{Lemma}
\newtheorem*{claim*}{Claim}
\newtheorem{Specialthm}{Theorem}
\renewcommand\theSpecialthm{\Alph{Specialthm}}
\numberwithin{equation}{section}
\renewcommand{\textfraction}{0.001}
\renewcommand{\topfraction}{0.999}
\renewcommand{\bottomfraction}{0.999}
\renewcommand{\floatpagefraction}{0.9}
\setlength{\textfloatsep}{5pt}
\setlength{\floatsep}{0pt}
\setlength{\abovecaptionskip}{2pt}
\setlength{\belowcaptionskip}{2pt}

\begin{abstract}
We define $\Ainfty$-structures -- algebras, coalgebras, modules, and
comodules -- in an arbitrary monoidal DG category or bicategory by
rewriting their definitions in terms of unbounded twisted complexes.
We develop new notions of strong homotopy unitality and bimodule
homotopy unitality  to work at this level of generality. For a strong
homotopy unital $\Ainfty$-algebra we construct Free-Forgetful homotopy
adjunction, its Kleisli category, and its derived category of modules.
Analogous constructions for $\Ainfty$-coalgebras require  bicomodule
homotopy counitality.  We define  homotopy adjunction for
$\Ainfty$-algebra and $\Ainfty$-coalgebra and show such pair to be
derived module-comodule equivalent.  As an application,  we obtain the
notions of an $\Ainfty$-monad and of an enhanced exact monad. We also
show that for any adjoint triple $(L,F,R)$ of functors between
enhanced triangulated categories the adjunction monad $RF$ and the
adjunction comonad $LF$ are derived module-comodule equivalent. 
\end{abstract}

\maketitle

\section{Introduction}
\label{section-introduction}

The notion of an $\Ainfty$-algebra originated in the work of Stasheff
\cite{Stasheff-HomotopyAssociativityOfHSpacesIandII} on higher
homotopy associativity. It comes from wanting to relax the associativity 
condition on the multiplication of a
differentially graded (DG) algebra $A$ to only hold ``up to homotopy'', 
i.e.~to a boundary of the differential.
An $\Ainfty$-algebra is a graded vector space $A$ equipped 
with operations $m_i\colon A^i \rightarrow A$ where $m_1$ is the differential, 
$m_2$ the multiplication, $m_3$ the homotopy up to which associativity
holds, and $m_{i \geq 4}$ are the higher homotopies describing the interaction 
of $m_1, \dots, m_{i-1}$. These and their generalisations 
were actively studied since, especially in the context of homological 
mirror symmetry \cite{Kadeishvili-OnTheTheoryOfHomologyOfFiberSpaces}
\cite{Fukaya-MorseHomotopyAinftyCategoryAndFloerHomologies}
\cite{Kontsevich-HomologicalAlgebraOfMirrorSymmetry}
\cite{Keller-IntroductionToAInfinityAlgebrasAndModules}
\cite{Fukaya-FloerHomologyAndMirrorSymmetryII} 
\cite{Lyubashenko-CategoryOfAinftyCategories}
\cite{Lefevre-SurLesAInftyCategories}\cite{KontsevichSoibelman-NotesOnAInftyAlgebrasAInftyCategoriesAndNoncommutativeGeometry}. 

Crucially, this standard definition 
takes place in the monoidal category of graded vector spaces.
Existing generalisations are via slight modifications 
of this monoidal category: an extra grading by objects of
a set yields $\Ainfty$-categories. 

In this paper, we rewrite the definitions of $\Ainfty$-algebras and
their (bi)modules, and $\Ainfty$-coalgebras and (bi)comodules, 
in the language of twisted complexes
\cite{BondalKapranov-EnhancedTriangulatedCategories}. This decouples
them from the differential $m_1$ so that they work in any
monoidal DG category $\A$. New definitions make implicit use of 
the internal differential provided in $\A$ by Yoneda embedding.  
We can now talk about $\Ainfty$-algebras in the monoidal DG
categories of DG functors and of enhanced exact functors, yielding 
the new notions of an \em $\Ainfty$-monad \rm and \em enhanced exact
monad\rm.  

We study modules over such $\Ainfty$-algebras and 
construct their derived categories. First, we develop a new 
notion of \em strong homotopy unitality. \rm There are several existing 
notions of homotopy unitality for $\Ainfty$-algebras 
\cite{Fukaya-MorseHomotopyAinftyCategoryAndFloerHomologies}\cite{Lyubashenko-HomotopyUnitalAinftyAlgebras}\cite{Lefevre-SurLesAInftyCategories}\cite{KontsevichSoibelman-NotesOnAInftyAlgebrasAInftyCategoriesAndNoncommutativeGeometry}. 
In the classical setting our notion is a small fraction of the data 
fixed in \cite{Fukaya-MorseHomotopyAinftyCategoryAndFloerHomologies}\cite{KontsevichSoibelman-NotesOnAInftyAlgebrasAInftyCategoriesAndNoncommutativeGeometry}.
It is still enough to define the derived category $D(A)$ of
an $\Ainfty$-algebra $A$ as 
the triangulated cocomplete hull of its homotopy unital $\Ainfty$-modules. 
We get the \em Free-Forgetful homotopy adjunction \rm 
for such modules. We define the \em Kleisli category \rm of $A$, a usual 
$\Ainfty$-category
\cite{Fukaya-MorseHomotopyAinftyCategoryAndFloerHomologies}\cite{Lefevre-SurLesAInftyCategories},
and prove it is derived equivalent 
to $A$ for strongly homotopy unital $A$. All this 
also works for $\Ainfty$-coalgebras, only with 
stronger notions of \em bicomodule homotopy counitality \rm for coalgebras
and \em strong homotopy counitality \rm for comodules. 
The derived category of an $\Ainfty$-coalgebra 
is only well-behaved if the free comodules generated by perfect objects 
of $\A$ are perfect. 

Finally, we define $\Ainfty$-algebra $A$ and
$\Ainfty$-coalgebra $C$ being homotopy adjoint. Such 
$C$ and $A$ are homotopy adjoint as objects of $\A$ and have 
a coherent system of higher homotopies describing how the adjunction 
interacts with their operations. For a strong homotopy
unital $A$ and bicomodule homotopy counital $C$ which are homotopy
adjoint, we prove \em derived module-comodule equivalence
\rm $D_c(A) \simeq D_c(C)$ by relating their Kleisli and co-Kleisli
categories. 
We show that for any homotopy adjoint triple $(L,F,R)$ 
in $\A$ the algebra $RF$ and the coalgebra $LF$ are homotopy adjoint. This
yields derived module-comodule equivalence $D_c(RF)
\simeq D_c(LF)$ for any adjoint triple $(L,F,R)$ of 
functors between enhanced triangulated categories.  

This paper was motivated by the following perspective. In non-derived algebra 
we can view algebras and coalgebras as monoids and comonoids in 
the abelian category of modules over the ground field or ring $k$. 
The hierarchy is completed by defining modules and comodules over 
the algebras and coalgebras. It turns out that this whole three level tower of constructions works just as well in any monoidal abelian category. For example, in the category of endofunctors of an abelian category. One arrives thus at
the notion of monads and comonads, and of modules and comodules over them. One important class of comonads arises when gluing quasicoherent sheaves over an affine covering of a scheme, that is -- it is built into the procedure of Zariski 
or flat or {\'e}tale descent. 

Difficulties apppear on several levels when one tries to repeat these
constructions in the setting of derived categories and of their DG enhancements. The substitutions for algebras, coalgebras, modules and comodules in general monoidal DG categories
(especially, in the functor categories) do not allow naive definitions mimicking the non-derived setting. This led some to do derived algebra and, in particular, the generalisations of functor adjunctions and of
Barr-Beck theorem using stable $\infty$-categories \cite{Lurie-HigherToposTheory}\cite{Lurie-DAG1StableInfinityCategories}. One example is the desired notion of gluing data up to homotopy for complexes of quasicoherent sheaves
on an affine/{\'e}tale/flat covering of a scheme. While the datum of gluing with higher homotopies that provide corrections is easy to write down, a comparison of glued categories with the derived category of quasicoherent sheaves on the scheme itself lacked both a setting and precise statements of theorems. As a remedy one referred to 
Barr-Beck-Lurie theorem \cite[\S3.4]{Lurie-DAG2NoncommutativeAlgebra} and homotopy descent for stable $\infty$-categories --- a topological machinery of a great level of abstraction that glues $\infty$-categories as a whole but makes it difficult to talk about models for glued objects and for glued
morphisms. 

The present article is devoted to creation of the setting that would fill the existing gap: we complete the hierarchy of a monoidal
category, algebras/coalgebras, modules/comodules in the seting of DG and $\Ainfty$-enhanced triangulated categories, thus allowing for explicit computations. 
We will apply it to proving a version of Barr-Beck
theorem for DG categories in a future work.

We now describe our results in more detail. We first note
that though they are formulated for an arbitrary monoidal DG
category $\A$, they all work equally well for $\A$ being any DG
bicategory. However, the language of bicategories is less intuitive, 
so we chose to write this paper in the language of monoidal
categories.  

\subsection{$\Ainfty$-structures in arbitrary DG monoidal categories}

Fix a ground field or
commutative ring $k$ and recall the usual definition of $\Ainfty$-algebra:
\begin{defn}[\cite{Lefevre-SurLesAInftyCategories}, D{\'e}finition 1.2.1.1]
\label{defn-intro-usual-Ainfty-algebra}
An \em $\Ainfty$-algebra \rm $(A,m_\bullet)$ is a graded $k$-module
$A$ and a collection $\left\{ m_i \right\}_{i \geq 1}$ of degree 
$2-i$ graded maps $A^i \rightarrow A$ satisfying
\begin{equation}
\label{eqn-intro-Ainfty-algebra-definition-equalities}
\forall\;\; i \geq 1
\quad \quad 
\sum_{j+k+l = i} (-1)^{jk + l} m_{j+1+l} \circ (\id^j \otimes m_k
\otimes \id^l) = 0.
\end{equation}
\end{defn}
The operations $m_i$ define a $k$-linear map 
$\bigoplus_{i \geq 0} (A[1])^i \rightarrow A$ from the free tensor 
coalgebra $\bigoplus_{i \geq 0} A^i[i]$. By the universal property of
free coalgebras this induces a unique coderivation of  
$\bigoplus_{i \geq 0} A^i[i]$. The condition
\eqref{eqn-intro-Ainfty-algebra-definition-equalities} is equivalent
to this coderivation squaring to zero, and thus defining a
differential on $\bigoplus_{i \geq 0} A^i[i]$. The resulting 
DG coalgebra is the \em bar-construction \rm of $(A,m_\bullet)$. 

Let $\modk$ be the DG category of DG $k$-modules. Tensor product
$\otimes_k$ makes it into a monoidal DG category. 
Definition \ref{defn-intro-usual-Ainfty-algebra} then reads:
an $\Ainfty$-algebra is an object $A$ and a collection $\left\{ m_i
\right\}_{i \geq 2}$ of morphisms $m_i\colon A^i \rightarrow A$ in
$\modk$ such that if we set $m_1$ to be the internal differential of $A$
the conditions \eqref{eqn-intro-Ainfty-algebra-definition-equalities} are 
satisfied. This doesn't generalise naively to an arbitrary monoidal
DG category $\A$ because its objects have no internal differentials. 
Even if we use Yoneda embedding $\A \rightarrow \modA$ to obtain
internal differentials, these are not compatible with the monoidal 
operation of $\A$: we can no longer express the internal differentials of 
$A^i$ in terms of internal differential $m_1$ of $A$. We can 
not hope to restate Definition \ref{defn-intro-usual-Ainfty-algebra}
with equations involving only $m_1$ and $\left\{ m_i
\right\}_{i \geq 2}$. To make matters worse, internal differentials are not morphisms in $\modA$ as they do not commute with
$\A$-action. 

Think instead of Definition \ref{defn-intro-usual-Ainfty-algebra} as saying
that the coderivation of $\bigoplus_{i \geq 0} A^i[i]$ induced by $\left\{ m_i \right\}_{i \geq 2}$ modifies its 
natural differential in $\modA$ to a new one. Now, the language for modifying 
the existing differential of the sum of shifted representable $\A$-modules 
exists - it is the language of twisted complexes over $\A$
\cite{BondalKapranov-EnhancedTriangulatedCategories}. One of its  
advantages is that it is phrased entirely in terms of the category $\A$, 
hiding the Yoneda embedding and the internal differentials into the
woodwork. We can work with twisted complexes over non-small
$\A$ where $\modA$ is not well-defined. However, 
the usual twisted complexes are bounded, while
due to the infinite number of terms in $\bigoplus_{i \geq 0} A^i[i]$ we 
need unbounded ones. We develop their theory  
in the companion paper \cite{AnnoLogvinenko-UnboundedTwistedComplexes},
and thus can define: 
\begin{defn}
\label{defn-intro-ainfty-algebra-a-monoidal-category}
Let $\A$ be a monoidal DG category. An \em $\Ainfty$-algebra
$(A,m_\bullet)$ in $\A$ \rm is an object $A \in \A$ and 
a collection $\left\{m_i\right\}_{i \geq 2}$ of degree $2-i$
morphisms $A^i \rightarrow A$ whose 
(non-augmented) bar-construction $\infbarnaug(A)$ is 
a twisted complex over $\A$:
\begin{tiny}
\begin{equation}
\label{eqn-intro-nonaugmented-bar-construction-of-A-m_i}
\begin{tikzcd}[column sep = 2.5cm]
\dots
\ar{r}[']{\begin{smallmatrix}A^3 m_2  - A^2m_2A + \\ + Am_2A^2 - m_2 A^3 \end{smallmatrix}}
\ar[bend left=20]{rr}[description]{A^2m_3 + Am_3A +  m_3 A^2}
\ar[bend left=25]{rrr}[description]{Am_4 - m_4A}
\ar[bend left=30]{rrrr}[description]{m_5}
&
A^4
\ar{r}[']{A^2m_2 - A m_2 A + m_2 A^2}
\ar[bend left=20]{rr}[description]{-Am_3 - m_3A}
\ar[bend left=25]{rrr}[description]{m_4}
& 
A^3
\ar{r}[']{Am_2 - m_2A}
\ar[bend left=20]{rr}[description]{m_3}
&
A^2
\ar{r}[']{m_2}
&
\underset{\degzero}{A}.
\end{tikzcd}
\end{equation}
\end{tiny}
\end{defn}

We give this definition in
\S\ref{section-ainfty-structures-in-monoidal-dg-categories}-\ref{section-free-modules-and-bimodules-over-Ainfty-algebra} along with
similar definitions of the DG categories $\nodA$ and $\Anod$ of 
right and left $\Ainfty$-$A$-modules, and the DG category $\AnodB$ of
$\Ainfty$-$A$-$B$-bimodules.
In \S\ref{section-Yoneda-embedding-of-A-modules-into-repA-modules}-\ref{section-twisted-complexes-of-ainfty-modules} we construct an analogue of 
Yoneda embedding for $\nodA$ which acts as the target 
for convolutions of twisted complexes. In 
\S\ref{section-the-homotopy-lemma} we prove the Homotopy Lemma:
a morphism $f_\bullet$ in
$\nodA$ is a homotopy equivalence if and only if $f_1$ 
is a homotopy equivalence in $\A$. For $\A = \modk$ 
for a field $k$, this is the well-known result that 
all $\Ainfty$-quasi-isomorphisms of $\Ainfty$-modules are 
homotopy equivalences, 
cf.~\cite[Prop.~2.4.1.1]{Lefevre-SurLesAInftyCategories}. 

In \S\ref{section-bar-construction-as-a-complex-of-ainfty-A-modules}
we show that the bar-construction $\infbar(E, p_\bullet)$ of an
$\Ainfty$-$A$-module $(E,p_\bullet)$ admits a natural lift
from a twisted complex over $\A$ to one over $\nodA$. 
The Homotopy Lemma then shows that if $(E,p_\bullet)$ is \em
$H$-unital\rm, i.e.~if $\infbar(E, p_\bullet)$ is null-homotopic in $\A$, then 
it is also null-homotopic in $\nodA$. We thus get a functorial 
\em bar-resolution \rm $\infbarres$ resolving any $H$-unital 
module by free modules, cf. \S\ref{section-bar-resolution}. 

\subsection{Strong homotopy unitality for $\Ainfty$-algebras} 
Several notions of homotopy unitality exist for usual $\Ainfty$-algebras. 
\em $H$-unitality \rm asks for 
the bar-construction to be null-homotopic and 
\em (weak) homotopy unitality \rm asks for a unital structure 
to exist in the homotopy category \cite[\S4]{Lefevre-SurLesAInftyCategories}
\cite{Lyubashenko-CategoryOfAinftyCategories}. 
A notion in
\cite{FukayaOhOhtaOno-LagrangianIntersectionFloerTheoryAnomalyAndObstruction}\cite{KontsevichSoibelman-NotesOnAInftyAlgebrasAInftyCategoriesAndNoncommutativeGeometry} fixes a large coherent system of higher homotopies up to which the
unitality holds. These are all equivalent for 
$k$ a field \cite{Lefevre-SurLesAInftyCategories} or any  
commutative ring \cite{LyubashenkoManzyuk-UnitalAinftyCategories}.  

The unitality notion of
\cite{FukayaOhOhtaOno-LagrangianIntersectionFloerTheoryAnomalyAndObstruction}
is easily defined in our generality, but we no longer know it
to be equivalent to the simpler notions above. 
In \S\ref{section-strong-homotopy-unitality} we introduce two new notions: 
\em strong homotopy unitality \rm 
and \em bimodule homotopy unitality\rm. Both
are small truncations of the structure in
\cite{FukayaOhOhtaOno-LagrangianIntersectionFloerTheoryAnomalyAndObstruction}
and are easier to construct, while still being 
enough to make our theory work. We do not know
whether in our generality there are analogues of the results in
\cite{LyubashenkoManzyuk-UnitalAinftyCategories} showing that
these simpler notions imply the existence of the full structure in
\cite{FukayaOhOhtaOno-LagrangianIntersectionFloerTheoryAnomalyAndObstruction}. 

The motivation of our definitions is as follows. 
An $\Ainfty$-algebra $(A,m_\bullet)$ is weakly homotopy unital 
if there exists a morphism $\eta\colon \id \rightarrow A$ in $\A$ 
such that 
\begin{equation}
\label{eqn-intro-weak-homotopy-unitality-conditions}
m_2 \circ \eta A  = \id_A + d h^r 
\quad \text{ and } \quad 
m_2 \circ A \eta  = \id_A + d h^l, 
\end{equation}
for some degree $-1$ endomorphisms $h^l, h^r$ of $A$ in $\A$. 
Now, for a strict algebra $(A,m_2)$ the maps
$m_2 \circ \eta A$ and $m_2 \circ A \eta$ are strict morphisms of 
right and left $A$-modules, respectively. It is natural to ask for
the homotopies between these and $\id_A$ to also be morphisms of left 
and right $A$-modules:
\begin{defn}[Defn.~\ref{defn-strong-homotopy-unitality-for-ainfty-algebras}]
$\Ainfty$-algebra $(A,m_i)$ is \em strongly homotopy unital \rm if there 
exists a unit morphism $\eta\colon \id \rightarrow A$ in $\A$ and
degree $-1$ endomorphisms $h^r_{\bullet}$ and $h^l_{\bullet}$ of $A$
in $\nodA$ and $\Anod$, respectively, such that 
\begin{equation}
\label{eqn-intro-strong-homotopy-unitality-conditions}
\mu_2 \circ {\eta}A = \id_A + d h^r_{\bullet}  \text{ in } \nodA
\quad \text{ and } \quad 
\mu_2 \circ A{\eta} = \id_A + d h^l_{\bullet} \text{ in } \Anod. 
\end{equation}
\end{defn}
The maps $\mu_2$ are natural lifts of
the operation $m_2\colon A^2 \rightarrow A$ from $\A$ into $\nodA$ and 
$\Anod$, see Defn.~\ref{defn-morphisms-pi_i-and-mu_i}. 
Thus our definition simply lifts the weak homotopy unitality conditions 
\eqref{eqn-intro-weak-homotopy-unitality-conditions} from $\A$ to $\nodA$
and $\Anod$. Instead of just the homotopies $h^l$
and $h^r$, we ask for two systems of homotopies $h^l_\bullet$
and $h^r_\bullet$ to exist. 

\em Bimodule homotopy unitality \rm asks 
for a degree $0$ morphism $\bareta_{\bullet\bullet}\colon \id_\A
\rightarrow A$ with a prescribed differential
to exist in the category $\AnodA$ of $\Ainfty$-$A$-$A$-bimodules, 
see Defn.~\ref{def-bimodule-homotopy-unitality}. However, the
following theorem explains it better than its definition:
\begin{theorem}[Theorem
\ref{theorem-conditions-for-bimodule-homotopy-unitality}]
\label{theorem-intro-conditions-for-bimodule-homotopy-unitality}
$(A,m_\bullet)$ is bimodule homotopy unital iff there exist 
\begin{itemize}
\item a closed degree $0$ morphism $\eta\colon \id_{\A} \rightarrow A$ in
$\A$,  
\item a degree $-1$ morphism $h^l_\bullet\colon A \rightarrow A$
in $\Anod$,
\item a degree $-1$ morphism $h^r_\bullet\colon A \rightarrow A$ 
in $\nodA$,
\item a degree $-2$ morphism $\kappa_{\bullet\bullet} \colon A^2
\rightarrow A$ in $\AnodA$, 
\end{itemize}
such that
$dh^r_\bullet = \id - \mu_2 \circ \eta{A}$,
$dh^l_\bullet = \id - \mu_2 \circ A\eta$,
and $d\kappa_{\bullet\bullet} = \mu_2 \circ (h^l_\bullet{A} -
{A}h^r_\bullet) - \mu_3 \circ A \eta A$.
\end{theorem}
In other words, the data of $\bareta_{\bullet\bullet}$ comprises
the homotopy unit $\eta$, strong homotopy unitality systems 
$h^l_\bullet$ and $h^r_\bullet$, and the system
$\kappa_{\bullet\bullet}$ of higher homotopies describing the
interaction of $\eta$, $h^l_\bullet$ and $h^r_\bullet$
with the operations $m_i$ of $A$. 
 
\subsection{Strong homotopy unitality for $\Ainfty$-modules} 

Existing notions of homotopy unitality for $\Ainfty$-modules 
\cite[\S4]{Lefevre-SurLesAInftyCategories} are: a (right) $\Ainfty$-$A$-module 
$(E,p_\bullet)$ is \em $H$-unital \rm if
its bar-construction is null-homotopic and \em (weakly) homotopy
unital \rm if  
\begin{equation}
\label{eqn-intro-weak-homotopy-unitality-for-modules}
E \xrightarrow{E\eta} EA \xrightarrow{p_2} E = \id_E + dh \quad \quad \text{ in } \A 
\end{equation}
for some degree $-1$ endomorphism $h$ of $E$ in $\A$.
Here $\eta$ is the homotopy unit of $A$. 

To define \em strong homotopy unitality \rm for modules 
we lift \eqref{eqn-intro-weak-homotopy-unitality-for-modules}
from $\A$ to $\nodA$. Turns out $EA$ must lift to the 
whole of the bar-resolution $\infbarres(E,p_\bullet)$. This requires
$A$ to be bimodule homotopy unital. We use
$\bareta_{\bullet\bullet}$ to cook up 
$$ \chi\colon (E,p_\bullet) \rightarrow \infbarres(E,p_\bullet) 
$$
which lifts $E\eta$ in \eqref{eqn-intro-weak-homotopy-unitality-for-modules}
to $\nodA$, see \S\ref{section-unitality-conditions-for-A-modules}. 
The lift of $p_2$ is the bar-resolution 
$$ \rho\colon \infbarres(E,p_\bullet) \rightarrow (E,p_\bullet). $$
\begin{defn}
\label{defn-intro-strong-homotopy-unitality}
A module $(E,p_\bullet) \in \nodA$ is \em strongly homotopy 
unital \rm if there exists a degree $-1$ endomorphism $h_\bullet$ 
of $(E,p_\bullet)$ in $\nodA$  such that 
$$ (E,p_\bullet) \xrightarrow{\chi} \infbarres(E,p_\bullet)
\xrightarrow{\rho} (E,p_\bullet) \quad = \quad \id + dh_\bullet 
\quad \quad \text{ in } \nodA. $$
\end{defn}

We prove all these notions 
of homotopy unitality to be equivalent: 
\begin{theorem}[Theorem
\ref{theorem-tfae-unitality-conditions-for-A-modules}]
\label{theorem-intro-tfae-unitality-conditions-for-A-modules}
Let $A$ be strongly homotopy unital and let $(E,p_\bullet)$ be an 
$\Ainfty$-$A$-module. The following are equivalent:
\begin{enumerate}
\item 
$(E,p_\bullet)$ is homotopy unital, 
\item 
$(E,p_\bullet)$ is $H$-unital. 
\end{enumerate}
If $A$ is bimodule homotopy unital, these are further equivalent to:
\begin{enumerate}
\setcounter{enumi}{2} 
\item 
$(E,p_\bullet)$ is strongly homotopy unital. 
\end{enumerate}
\end{theorem}

\subsection{Free-Forgetful homotopy adjunction and Kleisli category} 

For any strict and strictly unital algebra $A$ in $\A$ we have 
Free-Forgetful adjunction. For any $E \in \A$ and $(F,q) \in
\modd\text{-}A$ 
we have a natural isomorphism
$$ \homm_{\A}(E,F) \xrightarrow{\sim} \homm_{\modd\text{-}A}(EA, (F,q)). $$
For any strongly homotopy unital $\Ainfty$-algebra in $\A$
we construct in \S\ref{section-free-forgetful-homotopy-adjunction} 
a Free-Forgetful homotopy adjunction. For any $E \in \A$
and $(F,q_\bullet) \in \nodhuA$, the category of homotopy unital 
$\Ainfty$-$A$-modules, we have a natural homotopy equivalence
\begin{equation}
\label{eqn-intro-free-forgetful-homotopy-adjunction}
\homm_{\A}(E,F) \xrightarrow{\sim} \homm_{\nodhuA}(EA, (F,q_\bullet)).
\end{equation}

This motivates the following definition in
\S\ref{section-kleisli-category}:
\begin{defn}
\label{defn-intro-kleisli-category-of-an-ainfty-algebra}
Let $A$ be an $\Ainfty$-algebra in a monoidal DG category $\A$. Define
its \em Kleisli category \rm $\kleisliA$ to be the usual $\Ainfty$-category
\cite{Lefevre-SurLesAInftyCategories} defined by 
\begin{itemize}
\item Its objects are the objects of $\A$. 
\item For any $E,F \in \A$ the $\homm$-complex between them is
\begin{equation}
\homm_{\kleisliA}(E,F) := \homm_{\A}(E,FA).
\end{equation}
\item For any $E_1, E_2, \dots, E_{n+1} \in \A$ and any $\alpha_i \in 
\homm_{\kleisliA}(E_i, E_{i+1})$ define 
\begin{equation}
\label{eqn-intro-defn-of-ainfty-structure-on-kleisli-category}
m_n^{\kleisliA}(\alpha_1, \dots, \alpha_n) := 
E_1 \xrightarrow{\alpha_1} E_2A \xrightarrow{\alpha_2 A} \dots
\xrightarrow{\alpha_n A^{n-1}} E_{n+1}A^n \xrightarrow{E_{n+1}m_n^A}
E_{n+1}A.
\end{equation}
\end{itemize}
\end{defn}

We construct an $\Ainfty$-functor 
$f_\bullet\colon \kleisliA \rightarrow \freeA $
which sends any $E \in \A$ to $EA \in \freeA$ 
and whose $f_1$ is the Free-Forgetful 
adjunction map \eqref{eqn-intro-free-forgetful-homotopy-adjunction}. 
It is a quasi-equivalence for strongly homotopy unital $A$, 
The Kleisli category of a usual $\Ainfty$-category is it 
itself. Thus $f_\bullet$ generalises 
the Yoneda embedding of a usual $\Ainfty$-category into its
DG category of $\Ainfty$-modules \cite[\S7.1]{Lefevre-SurLesAInftyCategories}. 

\subsection{Derived category}

In \S\ref{section-the-derived-category-the-general-case} we define 
the \em derived category \rm of an $\Ainfty$-algebra $(A,m_\bullet)$  
in a monoidal DG category $\A$. By the Homotopy Lemma there 
is no need to invert quasi-isomorphisms. Following
\cite[\S4]{Lefevre-SurLesAInftyCategories}, we want 
the derived category $D(A)$ to be the homotopy category of the 
category $\nodhuA$ of $H$-unital $\Ainfty$-$A$-modules. 
However, we expect the derived category to be triangulated and
cocomplete, while for an arbitrary $\A$
the homotopy category $H^0(\nodhuA)$ is not apriori either.  
We need to fix an embedding of $\A$
into a cocomplete closed monoidal pretriangulated DG category $\B$. 
We also assume that $H^0(\A)$ is compactly generated within
$H^0(\B)$. When $\A$ is small, we can always take $\B = \modA$. Often 
we take $\B$ to be $\A$ itself, for example in the classical setting of $\A =
\modk$.  

Let $\noddinf\text{-}A^\B$ be the category of $\Ainfty$-$A$-modules in
$\B$. Its homotopy category is triangulated and cocomplete, so we define $D(A)$
to be the cocomplete triangulated hull of $\nodhuA$ in 
$H^0(\noddinf\text{-}A^\B)$. When $A$ is strongly
homotopy unital, the compact derived category $D_c(A)$ admits a
particularly nice description:
\begin{theorem}[Theorem
\ref{theorem-compact-derived-category-of-A-is-that-of-frees-and-kleisli}]
\label{theorem-intro-compact-derived-category-of-A-is-that-of-frees-and-kleisli}
Let $A$ be a strongly homotopy unital $\Ainfty$-algebra in a monoidal 
DG category $\A$. Let $S$ be a set of compact generators of $H^0(\A)$. Then 
$$ D_c(A) \simeq D_c(\freeAS) \simeq D_c(\kleisliAS). $$
\end{theorem}
Note that $D_c(\freeAS)$ and $D_c(\kleisliAS)$ are the usual compact
derived categories of the DG category $\freeAS$ and 
the $\Ainfty$-category $\kleisliAS$. This shows
that $D_c(A)$ is independent of the choice of $\B$. 

\subsection{$\Ainfty$-coalgebras}

In \S\ref{section-ainfty-coalgebras-and-comodules} we translate 
the definitions and results we obtained for
$\Ainfty$-algebras, modules and bimodules in \S\ref{section-ainfty-structures-in-monoidal-dg-categories}-\S\ref{section-the-derived-category}
to $\Ainfty$-coalgebras, comodules, and
bicomodules. Most translate straightforwardly and we give
all the key definitions. However there are some subtleties. The chief
is that the Homotopy Lemma fails for $\Ainfty$-comodules, 
see the introduction to
\S\ref{section-ainfty-coalgebras-and-comodules}. However, it turns 
out that it still holds for strongly homotopy counital
$\Ainfty$-comodules, see Lemma
\ref{lemma-coalgebra-analogue-of-the-homotopy-lemma}. 

The bottom line: working only with bicomodule homotopy counital 
$\Ainfty$-coalgebras and with strongly homotopy counital
$\Ainfty$-comodules over them, we have 
the analogues of all the definitions and results in 
\S\ref{section-ainfty-structures-in-monoidal-dg-categories}-\S\ref{section-strong-homotopy-unitality}. Consequently, we define 
the \em derived category of an $\Ainfty$-coalgebra $(C,
\Delta_\bullet)$ \rm as the cocomplete triangulated hull of $\conodshuC$ in 
$H^0(\conoddinf\text{-}C^\B)$. Here $\conodshuC$ is the category of 
strongly homotopy unital $\Ainfty$-$C$-comodules. We have 
for $D(C)$ the analogues of all the results in 
\S\ref{section-the-derived-category} provided the following additional 
assumption holds: free comodules $EC$ generated by compact $E \in \A$ 
are compact in $D(C)$. In particular we have:
\begin{theorem}[Theorem
\ref{theorem-compact-derived-category-of-bimodule-homotopy-unital-C-is-that-of-frees-and-cokleisli}]
Let $C$ be a bicomodule homotopy counital $\Ainfty$-coalgebra in a
monoidal DG category $\A$ and $S$ be a set of compact generators 
of $H^0(\A)$ in $H^0(\B)$.
If for every perfect $E \in \A$, the free comodule $EC$ is perfect,
then
$$ D_c(C) \simeq D_c(\free_S\text{-}C) \simeq D_c(\cokleisliCS). $$
\end{theorem}

The authors are aware that for usual DG and $\Ainfty$-coalgebras there
are subtleties involved in constructing their derived categories.
These are explored at length by Positselski in
\cite{Positselski-TwoKindsOfDerivedCategoriesKoszulDualityAndComoduleContramoduleCorrespondence}. 
We discuss this briefly in \S\ref{section-comparison-to-known-constructions}. 

\subsection{Module-comodule correspondence}

In \S\ref{section-homotopy-adjoint-Ainfty-algebras-and-coalgebras} we
define an $\Ainfty$-algebra $(A,m_\bullet)$ and an 
$\Ainfty$-coalgebra $(C,\Delta_\bullet)$ being \em homotopy adjoint\rm. 
We ask for $A$ and $C$ to be adjoint as objects of $\A$ with 
the unit $\eta_1\colon \id \rightarrow CA$ and counit 
$\epsilon_1\colon AC \rightarrow \id$ and for
coherent systems of higher homotopies $\eta_i\colon \id \rightarrow C^iA$ 
and $\epsilon_i\colon A^iC \rightarrow \id$ which describe 
how $\eta_1$ and $\epsilon_1$ interact with the operations
$m_\bullet$ and $\Delta_\bullet$. 
A non-trivial example is provided by any homotopy adjoint
triple $(L,F,R)$ of objects in $\A$:
the strict algebra $RF$ and strict the coalgebra $LF$ are homotopy adjoint, 
see Prop.~\ref{prps-LF-and-RF-are-homotopy-adjoint-in-a-monoidal-category}. 

Finally, we prove that any such pair are derived module-comodule equivalent: 
\begin{theorem}[Theorem \ref{theorem-module-comodule-correspondence-in-a-monoidal-dg-category}]
Let $(A,m_\bullet)$ be a strongly homotopy unital $\Ainfty$-algebra 
in a DG monoidal category $\A$. Let $(C,\Delta_\bullet)$
be a bicomodule homotopy counital $\Ainfty$-coalgebra. 
If $A$ and $C$ are homotopy adjoint as in 
Defn.~\ref{defn-homotopy-adjoint-ainfty-coalgebra-and-algebra},
then
$$ D_c(C) \simeq D_c(A). $$ 
\end{theorem}
The proof works by showing that if $C$ has a left homotopy adjoint in
$\A$, then any $\Ainfty$-coalgebra $(C,\Delta_\bullet)$ satisfies the
condition that free modules generated by perfect objects are perfect. 
Hence $D_c(C) \simeq D_c(\cokleisliC)$, while we also have $D_c(A) \simeq
D_c(\kleisliA)$. Finally, the homotopy adjunction of $A$ and $C$ 
induces an $\Ainfty$-quasi-equivalence $\kleisliA \simeq \cokleisliC$. 

\subsection{Examples and applications}

In \S\ref{section-examples-associative-algebras}-\S\ref{section-examples-dg-and-ainfty-categories}
we study how the classical notions of
associative/DG/$\Ainfty$-algebras, 
and DG/$\Ainfty$-categories fit into our setting. 

In \S\ref{section-classical-and-ainfty-monads} we set
$\A$ to be the DG bicategory $\DGFuntwocat$ of small DG categories, DG
functors and DG natural transformations. We get 
a new notion of an \em $\Ainfty$-monad \rm structure $(T,m_\bullet)$
on a DG endofunctor $T$ of a DG category $\C$. We define
its \em weak Eilenberg-Moore category $\eilmoorwk_T$ \rm which is just 
its DG category of homotopy unital $\Ainfty$-modules. This gives 
the right notion of the weak Eilenberg-Moore category for usual DG monads. 
In \S\ref{section-ainfty-modules-over-the-identity-functor} we look at
$\Ainfty$-modules over the identity endofunctor -- and 
show that these are the \em $\Ainfty$-idempotents \rm which appear in 
\cite[\S4]{GorskyHogancampWedrich-DerivedTracesOfSoergelCategories}. 

Finally, in \S\ref{section-enhanced-monads} we 
set $\A$ to be the bicategory $\enhcatkc$ 
of Karoubi complete enhanced triangulated categories, exact
functors, and natural transformation. We get  
a new notion of an \em enhanced exact monad \rm over an 
enhanced triangulated category and its \em weak Eilenberg-Moore
category\rm. In \S\ref{section-adjunction-monads-and-comonads}
for any adjoint pair $(F,R)$ of enhanced exact functors 
we construct enhanced monad and comonad structures 
on $RF$ and $FR$ which are strict and bimodule homotopy
unital and counital (Theorem
\ref{theorem-bimodule-homotopy-unitality-for-enhanced-adjunction-monads-and-comonads}). This leads to the following important result:
\begin{theorem}[Theorem
\ref{theorem-module-comodule-correspondence-for-enhanced-adjoint-triples}]
\label{theorem-intro-module-comodule-correspondence-for-enhanced-adjoint-triples}

Let $\C$ and $\D$ be enhanced triangulated categories, 
$F\colon \C \rightarrow \D$ be an enhanced exact functor, and 
$L,R\colon \D \rightarrow \C$ its left and right adjoints. 

Then the enhanced adjunction monad $RF$ and the enhanced adjunction monad $LF$ 
are derived module-comodule equivalent:
$D_c(RF) \simeq D_c(LF)$.
\end{theorem}
We will use this theorem as the key ingredient in our DG Barr-Beck 
construction, in other words - for homotopy descent and codescent. 
In a sense, it is an analogue of Positselski's derived comodule-contramodule 
correspondence
\cite[\S5]{Positselski-TwoKindsOfDerivedCategoriesKoszulDualityAndComoduleContramoduleCorrespondence}.

\em Acknowledgements: \rm We would like to thank Lino Amorim,
Alexander Efimov, Dmitri Kaledin, and Leonid Positselski 
for useful discussion. The first author would like to thank Kansas State
University for providing a stimulating research environment while
working on this paper. The third author would like to offer similar
thanks to Cardiff University. The second and the third author would
like to also thank Max-Planck-Institut f{\"u}r Mathematik Bonn. 

\section{Preliminaries}
\label{section-preliminaries}

\subsection{DG and $\Ainfty$-categories}

The main technical language of this paper is that of DG categories. 
For an overview in the same notation as this paper see 
\cite{AnnoLogvinenko-SphericalDGFunctors}, 
\S2-4. Other sources include
\cite{Keller-DerivingDGCategories},
\cite{Toen-TheHomotopyTheoryOfDGCategoriesAndDerivedMoritaTheory},
\cite{Toen-LecturesOnDGCategories}, and 
\cite{LuntsOrlov-UniquenessOfEnhancementForTriangulatedCategories}.

Throughout the paper we work in a fixed universe $\mathbb{U}$ of sets 
containing an infinite set. We also fix the base field or commutative 
ring $k$. We write $\modk$ for the category of 
$\mathbb{U}$-small complexes of $k$-modules. 
It is a cocomplete closed symmetric monoidal category
with monoidal operation $\otimes_k$ and unit $k$. A \em DG category \rm 
is a category enriched over $\modk$. In particular, any DG category 
is locally small. Given a small DG category $\A$, we write $\modA$ for 
the DG category of (right) $\A$-modules, that is -- of functors 
$\Aopp \rightarrow \modk$. 

For an introduction to the usual $\Ainfty$-categories see 
\cite{Keller-AInfinityAlgebrasModulesAndFunctorCategories}, 
for a comprehensive technical text --
\cite{Lefevre-SurLesAInftyCategories} and \cite{Lyubashenko-CategoryOfAinftyCategories}.  
For the summary of the technical details relevant to this paper, see
\cite[\S2]{AnnoLogvinenko-BarCategoryOfModulesAndHomotopyAdjunctionForTensorFunctors}.
In particular, in 
\S\ref{section-enhanced-monads}-\S\ref{section-adjunction-monads-and-comonads}
we make extensive use of bar categories of modules and bimodules introduced in \cite[\S3-4]{AnnoLogvinenko-BarCategoryOfModulesAndHomotopyAdjunctionForTensorFunctors}.

\subsection{Unbounded twisted complexes}
\label{section-unbounded-twisted-complexes}

To define $\Ainfty$-structures in a DG bicategory, 
we rewrite them in 
\S\ref{section-ainfty-structures-in-monoidal-dg-categories} in the
language of twisted complexes 
\cite{BondalKapranov-EnhancedTriangulatedCategories}. We need 
the unbounded version introduced in
\cite{AnnoLogvinenko-UnboundedTwistedComplexes}. We 
give a summary below: 

\begin{defn}
\label{defn-unbounded-twisted-complexes} 
Let $\A$ be a DG category with a fully faithful embedding into 
a DG category $\B$ which has countable direct sums and shifts. 
An \em unbounded twisted complex \rm over $\A$ relative to $\B$ consists 
of 
\begin{itemize}
\item $\forall\; i \in \mathbb{Z}$, an object $a_i$ of $\A$,
\item $\forall\; i,j \in \mathbb{Z}$, 
a degree $i - j + 1$ morphism $\alpha_{ij}\colon a_i \rightarrow a_j$ in $\A$,
\end{itemize}
satisfying 
\begin{itemize}
\item $\sum \alpha_{ij}$ is an endomorphism of $\bigoplus_{i \in
\mathbb{Z}} a_i[-i]$ in $\B$, 
\item The twisted complex condition 
\begin{equation}
\label{eqn-the-twisted-complex-condition}
(-1)^j d\alpha_{ij} + \sum_k \alpha_{kj} \circ \alpha_{ik} = 0. 
\end{equation}
\end{itemize}

We make such complexes into a DG category $\twcxub_\B(\A)$ by setting
\begin{equation}
\label{eqn-the-hom-complex-of-unbounded-twisted-complexes-naive}
\homm^\bullet_{\twcxub_{\B}(\A)}\bigl((a_i, \alpha_{ij}),(b_i,
\beta_{ij})\bigr) := \homm^\bullet_{\B}(\bigoplus_{k \in \mathbb{Z}} a_k[-k],
 \bigoplus_{l \in \mathbb{Z}} b_l[-l]) 
\end{equation}
where each $f \in \homm^q_\A(a_k, b_l)$ has degree $q + l - k$ and
\begin{equation}
\label{eqn-the-hom-complex-of-twisted-complexes-differential}
 df := (-1)^l d_{\A} f + \sum_{m \in
\mathbb{Z}}\left( \beta_{lm} \circ f - (-1)^{q + l -k} f 
\circ \alpha_{mk} \right).
\end{equation}
\end{defn}

We have a fully faithful convolution functor
$\conv\colon\; \twcxub_{\B}(\A) \hookrightarrow \modB$
and an embedding $\A \hookrightarrow \twcxub_{\B}(\A)$
sending objects to complexes concentrated in degree $0$.  

A DG category $\B$ 
\em admits convolutions of unbounded twisted complexes \rm
if $\B  \hookrightarrow \twcxub_{\B}(\B)$ is an equivalence. 
For such $\B$ the convolution functor $\twcxub_{\B}(\A)
\rightarrow \modB$ for any $\A \subseteq \B$ takes values in $\B$. 

Where the choice of $\B$ is not relevant we write 
$\twcxub(\A)$ for $\twcxub_{\B}(\A)$. We 
define $\twcxpls(\A)$ and $\twcxmns(\A)$ to be the full subcategories
of $\twcxub(\A)$ consisting of all
bounded above and all bounded below twisted complexes, respectively. 
A twisted complex $(a_i, \alpha_{ij})$ is \em one-sided \rm 
if $\alpha_{ij} = 0$ for $j \leq i$. We define $\pretriag^{\bullet}(\A)$,
to be the full subcategory of $\twcx^{\bullet}(\A)$
consisting of one-sided twisted complexes. 

\subsection{The setting}
\label{section-the-setting}
In this paper we work in a general setting of an arbitrary
monoidal DG category $\A$. See 
\cite[\S{VII.1}]{MacLane-CategoriesfortheWorkingMathematician}
for the definition of a monoidal category.
We denote the monoidal operation of $\A$ by $\otimes$, the identity 
object by $1_\A$, and the associator and left and right unitor
isomorphisms by $\alpha$, $\rho$, and $\lambda$. 

A monoidal category is a special instance of a bicategory:
it is a bicategory with one object. See
\cite{Benabou-IntroductionToBicategories} 
\cite{GyengeKoppensteinerLogvinenko-TheHeisenbergCategoryOfACategory}
for an introduction to bicategories. Their language is less intuitive, so we  
wrote this paper in the language of monoidal categories. All our
results work just as well for $\A$ a DG bicategory.

Another expository choice is the use of $\modA$ throughout 
\S\ref{section-ainfty-structures-in-monoidal-dg-categories}-\ref{section-strong-homotopy-unitality} as the ambient category $\B$
containing $\A$ which is necessary to define unbounded twisted
complexes over $\A$, cf.~\S\ref{section-unbounded-twisted-complexes}. 
In those sections the Yoneda embedding $\A
\hookrightarrow \modA$ can be replaced by a monoidal embedding of $\A$
into any monoidal DG category $\B$ which is closed, cocomplete (closed
under small direct sums) and admits convolutions of twisted
complexes. The last condition can also be relaxed to just being
strongly pretriangulated. Any cocomplete strongly pretriangulated
category admits convolutions of one-sided bounded below and bounded
above twisted complexes, and these are the only unbounded complexes we
consider. 

From \S\ref{section-the-derived-category} onwards, 
we use this ambient category not just for unbounded twisted
complexes, but also to construct derived categories of 
$\Ainfty$-algebras and $\Ainfty$-coalgebras. Hence we 
switch from using $\modA$ to the fully general notation where
it is denoted by $\B$. We also make some additional assumptions on
$\B$, see the introduction to \S\ref{section-the-derived-category}. 
When $\A$ is small, all these assumptions are still satisfied by $\modA$. 

When $\A$ is not small, we could enlarge the universe to make it small
and then still use $\modA$. However, this would make all objects of 
$\A$ compact, and yield a wrong derived category. Indeed, this is 
the reason for the formalism of an arbitrary ambient category $\B$ 
satisfying a list of assumptions. For example when $\A = \modC$ for 
small $\C$, we can set $\B$ to be $\A$ itself.  

\section{$\Ainfty$-structures in monoidal DG categories via twisted complexes}
\label{section-ainfty-structures-in-monoidal-dg-categories}

\subsection{$\Ainfty$-algebras}
\label{section-Ainfty-algebras-in-a-monoidal-category}

The central idea of this paper is to use unbounded twisted 
complexes to reformulate the definitions 
of $\Ainfty$-algebras and modules 
\cite[\S2]{Lefevre-SurLesAInftyCategories}. 
Traditionally, $\Ainfty$-algebra formalism was defined 
in the DG category $\modk$ of DG complexes of $k$-modules 
with the monoidal structure given by the tensor product of complexes
\cite[\S2]{Lefevre-SurLesAInftyCategories}. For 
$\Ainfty$-categories, one extends this from $\modk$ to 
$k_S$-$\modd$-$k_S$ for some set $S$ 
\cite[\S5]{Lefevre-SurLesAInftyCategories}. Here $k_S$ is the category whose
object set is $S$ and whose morphisms are the scalar multiples
of identity. 

In $k_S$-$\modd$-$k_S$, the defining equations of an $\Ainfty$-algebra
involve the internal differential $m_1$ of the object as 
a complex of $k$-modules. In an arbitraty monoidal DG category $\A$, 
the objects do not apriori possess an internal differential, see the
discussion after Cor.~\ref{cor-equivalence-of-our-definition-and-usual-for-mod-k}. 
The language of twisted complexes allows to define 
an $\Ainfty$-algebra structure on $a \in \A$ 
while referring explicitly only to operations 
$\left\{ m_i \right\}_{i \geq 2}$:

\begin{defn}
\label{defn-algebra-bar-construction-in-a-monoidal-category}
Let $\A$ be a monoidal DG category, let $A \in \A$ and let 
$\left\{m_i\right\}_{i \geq 2}$ be a collection of degree $2-i$
morphisms $A^i \rightarrow A$. 
The \em (non-augmented) bar-construction $\infbarnaug(A)$ \rm of $A$ 
is the collection of objects $A^{i+1}$ for all $i \geq 0$ each placed
in degree $-i$ and of degree $k-1$ maps 
$d_{(i+k)i}\colon A^{i+k} \rightarrow A^i$ 
defined by
\begin{equation}
\label{eqn-differentials-in-non-aug-bar-construction}
d_{(i+k)i} := (-1)^{(i-1)(k+1)} \sum_{j = 0}^{i-1} (-1)^{jk} \id^{i-j-1}
\otimes m_{k+1} \otimes \id^{j}. 
\end{equation} 

\begin{tiny}
\begin{equation}
\label{eqn-nonaugmented-bar-construction-of-A-m_i}
\begin{tikzcd}[column sep = 2.5cm]
\dots
\ar{r}[']{\begin{smallmatrix}A^3 m_2  - A^2m_2A + \\ + Am_2A^2 - m_2 A^3 \end{smallmatrix}}
\ar[bend left=15]{rr}[description]{A^2m_3 + Am_3A +  m_3 A^2}
\ar[bend left=20]{rrr}[description]{Am_4 - m_4A}
\ar[bend left=25]{rrrr}[description]{m_5}
&
A^4
\ar{r}[']{A^2m_2 - A m_2 A + m_2 A^2}
\ar[bend left=15]{rr}[description]{-Am_3 - m_3A}
\ar[bend left=20]{rrr}[description]{m_4}
& 
A^3
\ar{r}[']{Am_2 - m_2A}
\ar[bend left=15]{rr}[description]{m_3}
&
A^2
\ar{r}[']{m_2}
&
\underset{\degzero}{A}
\end{tikzcd}
\end{equation}
\end{tiny}

\end{defn}

\begin{defn}
\label{defn-ainfty-algebra-in-a-monoidal-category}
Let $\A$ be a monoidal DG category. An \em $\Ainfty$-algebra $(A,m_i)$ \rm 
in $\A$ is an object $A \in \A$ equipped 
with operations $m_i \colon A^i \rightarrow A$ for all $i \geq 2$
which are degree $2-i$ morphisms in $\A$ such that their non-augmented 
bar-construction $\infbarnaug(A)$ is a twisted complex over $\A$. 
\end{defn}

Next, we compare 
Defn.~\ref{defn-ainfty-algebra-in-a-monoidal-category} 
with the standard definition of an $\Ainfty$-algebra:
\begin{defn}[\cite{Lefevre-SurLesAInftyCategories}, D{\'e}finition 1.2.1.1]
Let $S$ be a set and $\A = k_S\text{-}\modd\text{-}k_S$ with the
tensor product monoidal structure. 
An \em $\Ainfty$-algebra \rm in $\A$ is an object $A \in \A$
equipped with degree $2-k$ morphisms $m_k\colon A^{k} \rightarrow A$ 
for all $k \geq 1$ such that $m_1$ is the internal differential of $A$, and 
\begin{equation}
\label{eqn-Ainfty-algebra-definition-equalities}
\forall\;\; i \geq 1
\quad \quad 
\sum_{j+k+l = i} (-1)^{jk + l} m_{j+1+l} \circ (\id^j \otimes m_k \otimes \id^l) = 0 
\end{equation}
\end{defn}
 
By \cite[Lemme 1.2.2.1]{Lefevre-SurLesAInftyCategories}
this is equivalent to the following. 
By the universal property of free coalgebras, the morphism  
$ \bigoplus_{i \geq 1} (A[1])^i \rightarrow A[1] $
induced by $\sum_{i \geq 1} m_i$ defines a coderivation 
$b\colon 
\bigoplus_{i \geq 1} (A[1])^i \rightarrow \bigoplus_{i \geq 1}
(A[1])^i$.
Operations $(m_k)_{k \geq 1}$ define an $\Ainfty$-algebra structure on
$A$ if and only if $b$ is a differential on the free coalgebra 
$\bigoplus_{i \geq 1} (A[1])^i$, i.e. $b^2 = 0$. 

The coderivation of $\bigoplus_{i \geq 1} (A[1])^i$
induced by $m_1$ is its internal differential. 
On the other hand, the differentials in  
\eqref{eqn-nonaugmented-bar-construction-of-A-m_i}
were chosen to ensure the following: 

\begin{lemma}
\label{lemma-total-endomorphism-of-nonaug-bar-constr-is-a-coderivation-induced-by-m_i}
Let $\A$ be a monoidal DG category, let $A \in \A$ and let 
$m_k\colon A^k \rightarrow A$ be a collection of degree $2-k$
of morphisms in $\A$. In $\modA$, let 
$$ \mu \colon  \bigoplus_{i \geq 1} A^i[i-1] \rightarrow
\bigoplus_{i \geq 1} A^i[i-1] $$
be the sum of all the differentials in the non-augmented bar-construction 
\eqref{eqn-nonaugmented-bar-construction-of-A-m_i}. 

Denote by $w\colon A[1] \rightarrow A$ the degree $1$ map given 
by $\id_A$. For any $i \geq 0$, we can view $w^i$ as an isomorphism 
$(A[1])^i \xrightarrow{\sim} A^i[i]$. The resulting isomorphism 
in $\modA$
$$ \sum_{i\geq 1} w^i\colon \bigoplus_{i \geq 1} (A[1])^i \rightarrow
\bigoplus_{i \geq 1} A^i[i] $$
identifies the coderivation of $\bigoplus (A[1])^i$ induced by 
$\sum_{i \geq 2} m_i$ with $\mu[1]$. 
\end{lemma}
Here and below we use the fact that 
the monoidal structure on $\A$ induces the monoidal structure on $\modA$
\cite[\S4.5]{GyengeKoppensteinerLogvinenko-TheHeisenbergCategoryOfACategory}.

\begin{proof}
Recall the definition of the coderivation of $\bigoplus_{i\geq 1} (A[1])^i$
induced by $\sum_{i \geq 2} m_i$. 
Let $s\colon A \rightarrow A[1]$ be the degree $-1$ map defined by
$\id_A$. It is the inverse of $w$, however  
$s^i \circ w^i = w^i \circ s^i = (-1)^{i(i-1)/2} \id_{A^k}$. 
Let $b_k\colon (A[1])^k \rightarrow A[1]$ be
$- s \circ m_k \circ w^k$. We have a degree $1$ map 
$ \sum_{i \geq 2} b_i\colon \bigoplus (A[1])^i \rightarrow A[1]$. 
By the universal property of free coalgebras 
it corresponds to the degree $1$ coderivation
$b$ of $\bigoplus_{i \geq 1} (A[1])^i$
whose components are the maps 
$(A[1])^{i+k} \rightarrow (A[1])^i$ with $k \geq 1$ given by 
$ \sum_{j = 0}^{i-1} \id^{i-j-1} \otimes b_{k+1} \otimes \id^j$. 

The isomorphism $\sum w^i$ and its inverse $\sum (-1)^{i(i-1)/2} s^i$
identify $b$ with the endomorphism of $\bigoplus_{i \geq 1} A^i[i]$ 
whose components $A^{i+k}[i+k] \rightarrow A^{i}[i]$
with $k \geq 1$ are 
\begin{small}
\begin{align*}
& \sum_{j = 0}^{i-1} (-1)^{(i+k)(i + k - 1)/2} w^i \circ (\id^{i-j-1} \otimes
b_{k+1} \otimes \id^j) \circ s^{i+k} = 
\\  
= & 
\sum_{j = 0}^{i-1} (-1)^{(i+k)(i + k - 1)/2 + 1 } w^i \circ 
(\id^{i-j-1} \otimes (s \circ m_{k+1} \circ w^{k+1}) \otimes \id^j) 
\circ s^{i+k} = 
\\
= &
\sum_{j = 0}^{i-1} (-1)^{(i+k)(i + k - 1)/2 + 1 + j } 
(w^{i-j-1} \otimes \id \otimes w^j) \circ 
(\id^{i-j-1} \otimes (m_{k+1} \circ w^{k+1}) \otimes \id^j) 
\circ s^{i+k} = 
\\
= &
\sum_{j = 0}^{i-1} (-1)^{(i+k)(i + k - 1)/2 + 1 + j + (i-1)(k+1) + j(k+1)} 
(\id^{i-j-1} \otimes m_{k+1} \otimes \id^j) 
\circ
w^{i+k} \circ s^{i+k} = 
\\
= &
\sum_{j = 0}^{i-1} (-1)^{1 + (i-1)(k+1) + jk)} 
(\id^{i-j-1} \otimes m_{k+1} \otimes \id^j). 
\end{align*}
\end{small}
This only differs from
\eqref{eqn-differentials-in-non-aug-bar-construction}
by a minus sign, whence the resulting endomorphism of 
$\bigoplus_{i \geq 1} A^i[i]$ equals $\mu[1] = -\mu$. 
\end{proof}

Our Defn.~\ref{defn-ainfty-algebra-in-a-monoidal-category} asks 
for \eqref{eqn-nonaugmented-bar-construction-of-A-m_i} to be a twisted
complex, that is –– adding $\mu$ to the internal differential of 
$\bigoplus_{i \geq 1} A^i[i-1]$ gives 
a new differential on the latter. This is equivalent to the condition
that adding $\mu[1]$ to the internal differential of 
$\bigoplus_{i \geq 1} A^i[i]$ also gives a new differential on the latter. 
The isomorphism $\sum_{i \geq 1} w^i$ identifies the internal
differential of $\bigoplus_{i \geq 1} A^i[i]$ with the internal
differential of $\bigoplus_{1 \geq 1} (A[1])^i$, i.e with the
coderivation induced by $m_1$. On the other hand, it identifies 
$\mu[1]$ with the coderivation induced by $\sum_{k \geq 2} m_k$.  
We thus obtain:
\begin{cor}
\label{cor-equivalence-of-our-definition-and-usual-for-mod-k}
For $\A$ either $\modk$ or $k_S$-$\modd$-$k_S$, the following
are equivalent:
\begin{enumerate}
\item Our Defn.~\ref{defn-ainfty-algebra-in-a-monoidal-category}
\item The usual definition \cite[D{\'e}finition
1.2.1.1]{Lefevre-SurLesAInftyCategories} of an $\Ainfty$-algebra. 
\item The coderivation induced by $\sum_{i \geq 1} m_i$ on 
the free coalgebra $\bigoplus_{i \geq 0} (A[1])^i$ is a differential. 
\end{enumerate}
\end{cor}

When $\A$ is an arbitrary monoidal DG category, its objects do 
not apriori have internal differentials. This is fixed by the 
Yoneda embedding $\A \hookrightarrow \modA$ which sends $A$ to 
$\homm_\A(-,A)$. The latter does have an internal differential
given by the differential of $\A$. However it doesn't 
respect the right action of $\A$ by pre-composition, and so  
isn't a morphism in $\modA$, but only in $\modd\text{-}k_\A$. 

For an arbitrary monoidal category, 
our Defn.~\ref{defn-ainfty-algebra-in-a-monoidal-category} is still
equivalent to the condition that adding the coderivation induced 
by $\sum_{i \geq 2} m_i$ to the natural differential of  
$\bigoplus_{i \geq 1} (A[1])^i$ gives a new differential. 
Thus $\Ainfty$-structures on $A$ are still in
bijection with differentials on $\bigoplus_{i \geq 1} (A[1])^i$ 
generalising \cite[Lemme 1.2.2.1]{Lefevre-SurLesAInftyCategories}. 

What about the definition in terms of the explicit
equalities \eqref{eqn-Ainfty-algebra-definition-equalities} for 
$(m_k)_{k \geq 1}$ to satisfy? Is 
our Defn.~\ref{defn-ainfty-algebra-in-a-monoidal-category}
equivalent to the internal differential $m_1$ of $A$
and the higher operations $(m_k)_{k \geq 2}$ satisfying
\eqref{eqn-Ainfty-algebra-definition-equalities} in some appropriate
category. No: we can not even write down
\eqref{eqn-Ainfty-algebra-definition-equalities}
for general $\A$. Since $m_1$ is not a morphism in $\modA$, it is not, apriori,
compatible with its monoidal structure. 
For example, we can't even define the operation $\id \otimes m_1$ on 
the module $\homm_\A(-,A^2)$. 
Similarly, we can't define any of the maps 
$\id^{i-l-1} \otimes m_1 \otimes \id^l: A^{i} \rightarrow A^{i}$ 
in \eqref{eqn-Ainfty-algebra-definition-equalities}. 

However, even though individual maps $\id^{i-l-1} \otimes m_1 \otimes \id^l$
do not exist, we can still make sense of their sum
$\sum_{i=0}^{i-1} \id^{i-l-1} \otimes m_1 \otimes \id^l$.
When $\A = \modk$, this is the natural differential $d_{\A}$ on 
$\homm_\A(-,A^i)$. Hence the sum of all terms involving $m_1$ in 
\eqref{eqn-Ainfty-algebra-definition-equalities} equals $d_\A m_i$.
Rewriting \eqref{eqn-Ainfty-algebra-definition-equalities} to 
replace the former with the latter yields those twisted
complex conditions for \eqref{eqn-nonaugmented-bar-construction-of-A-m_i}   
which involve the differentials of the arrows labelled with $m_i$.  
This is the way to compare  
Defn.~\ref{defn-ainfty-algebra-in-a-monoidal-category} with the 
original definition \cite[D{\'e}finition
1.2.1.1]{Lefevre-SurLesAInftyCategories} of an $\Ainfty$-algebra:

\begin{prps}
\label{prps-comparing-new-and-old-defns-of-ainfty-algebra}
Let $\A$ be a monoidal DG category, let $A \in \A$ and let 
$(m_i)_{i \geq 2}$  be a collection of degree $2-i$
morphisms $A^i \rightarrow A$ in $\A$. 

\begin{enumerate}
\item 
\label{item-defining-equalities-of-Ainfty-are-subset-of-twisted-complex-conditions}
Take the defining equalities 
\eqref{eqn-Ainfty-algebra-definition-equalities} in the original 
definition of $\Ainfty$-algebra and replace all terms involving $m_1$ with $d_\A m_i$:
\begin{equation}
\label{eqn-defining-equalities-for-new-definition-of-Ainfinity-algebra}
\forall\; i \geq 2 \quad\quad 
d_\A m_i + \sum_{\begin{smallmatrix}j+k+l = i, \; k \geq
2\end{smallmatrix}} (-1)^{jk+l} m_{j+1+l} \circ \left(\id^j \otimes
m_k \otimes \id^l\right) = 0. 
\end{equation}
The resulting equalities are a subset of twisted complex conditions  \eqref{eqn-the-twisted-complex-condition} for non-augmented bar
construction \eqref{eqn-nonaugmented-bar-construction-of-A-m_i}. It consists of those conditions which involve differentiating the twisted differentials $m_i\colon A^k \rightarrow A$.
 
\item 
\label{item-twcx-condition-on-m-i-imply-the-rest}
If the equalities \eqref{eqn-defining-equalities-for-new-definition-of-Ainfinity-algebra} hold, then so do the rest of the twisted complex conditions for \eqref{eqn-nonaugmented-bar-construction-of-A-m_i}. 
Thus  $(A,m_i)$ is an $\Ainfty$-algebra in
the sense of Defn.~\ref{defn-ainfty-algebra-in-a-monoidal-category} if and only if $m_i$ satisfy the equalities 
\eqref{eqn-defining-equalities-for-new-definition-of-Ainfinity-algebra}. 
\end{enumerate} 
\end{prps}

\begin{proof}
\underline{\eqref{item-defining-equalities-of-Ainfty-are-subset-of-twisted-complex-conditions}}:
Let $i \geq 2$. Taking the formula \eqref{eqn-differentials-in-non-aug-bar-construction} for the twisted differentials in the non-augmented bar-construction \eqref{eqn-nonaugmented-bar-construction-of-A-m_i} and using it to write out the twisted complex condition \eqref{eqn-the-twisted-complex-condition} for the twisted differential $m_i: A^i \rightarrow A$ in \eqref{eqn-nonaugmented-bar-construction-of-A-m_i} we obtain:
\begin{scriptsize}
\begin{align*}
& d_\A m_i + \sum_{p = 2}^{i-1} m_p \circ 
\left(
(-1)^{(p-1)(i-p+1)}
\sum_{q = 0}^{p-1}
(-1)^{q(i-p)} \id^{p-q-1} \otimes m_{i-p+1} \otimes \id^q 
\right) = 0 \; \Leftrightarrow
\\
\Leftrightarrow \;\; & d_\A m_i + 
\sum_{p = 2}^{i-1}
\sum_{q = 0}^{p-1}
(-1)^{(p-q+1)(i-p+1)+ q}
 m_p \circ 
 \id^{p-q-1} \otimes m_{i-p+1} \otimes \id^q 
  = 0.
\end{align*}
\end{scriptsize}
A change of summation indices $l = q$, $j = p-q-1$, and $k = i - j - l = i - p + 1$ transforms the above into the equality 
\eqref{eqn-defining-equalities-for-new-definition-of-Ainfinity-algebra}, as desired. 

\underline{\eqref{item-twcx-condition-on-m-i-imply-the-rest}}:
Let $\mu$ be the endomorphism of 
$\bigoplus_{i \geq 1} A^i[i-1]$ defined by the maps in 
\eqref{eqn-nonaugmented-bar-construction-of-A-m_i}. The twisted
complex condition \eqref{eqn-the-twisted-complex-condition} for each 
map $A^i \rightarrow A^j$ in \eqref{eqn-nonaugmented-bar-construction-of-A-m_i}
is equivalent to the vanising of 
the $A^i[i-1] \rightarrow A^j[j-1]$ component of 
\begin{equation}
\label{eqn-twisted-complex-condition-on-mu}
d_\A (\mu) + \mu \circ \mu. 
\end{equation}
It suffices therefore to show that 
\eqref{eqn-twisted-complex-condition-on-mu} is completely determined
by its  $A^k[k-1] \rightarrow A$ components. 
By Lemma  
\ref{lemma-total-endomorphism-of-nonaug-bar-constr-is-a-coderivation-induced-by-m_i},
natural isomorphisms $w^i\colon (A[1])^i \rightarrow A^i[i]$ identify 
$\mu[1]$ with the coderivation $b$ of $\bigoplus (A[1])^i$ induced by 
$\sum_{i \geq 2} m_i$. By the universal property of tensor coalgebras, 
any coderivation of $\bigoplus (A[1])^i$ is uniquely determined by its
$(A[1])^k \rightarrow A[1]$ components. In particular, this is true of
$d_\A (b) + b \circ b$. Hence, it is also true of 
$d_\A (\mu[1]) + \mu[1] \circ \mu[1]$, and thus of
\eqref{eqn-twisted-complex-condition-on-mu}. 
\end{proof}

We next define morphisms of $\Ainfty$-algebras in $\A$ in a similar
way:

\begin{defn} 
\label{defn-the-bar-construction-of-a-morphism-of-ainfty-algebras}
Let $(A, m_k)$ and $(B, n_k)$ be $\Ainfty$-algebras in $\A$. 
Let $(f_i)_{i \geq 1}$ be a collection of degree $1 - i$ 
morphisms $A^i \rightarrow B$. 
The \em bar-construction \rm $\infbar(f_\bullet)$ is the morphism 
$\infbarnaug(A) \rightarrow \infbarnaug(B)$ in $\pretriagmns(\A)$
whose $A^{i+k} \rightarrow B^i$ component is 
$$ \sum_{t_1 + \dots + t_i = i + k} 
(-1)^{\sum_{l=2}^{i}(1-t_l)\sum_{n=1}^l t_n}
f_{t_1}\otimes\ldots \otimes f_{t_i}. 
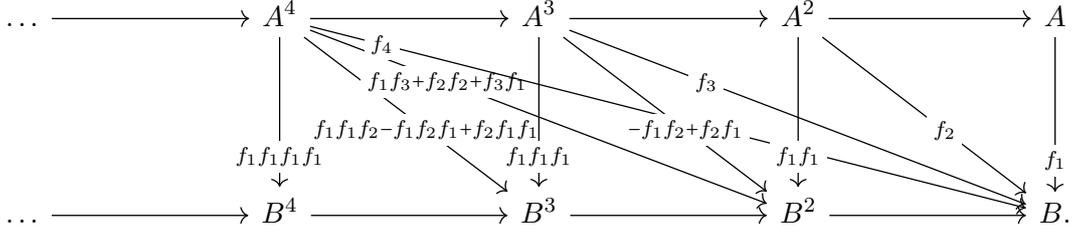
$$
\begin{small}
\begin{equation*}
\begin{tikzcd}[column sep = 2.6cm, row sep = 2cm]
\dots
\ar{r}
&
A^4
\ar{r}
\ar{d}[description, pos = 0.80]{f_1f_1f_1f_1}
\ar{dr}[description, pos = 0.60]{f_1f_1f_2 - f_1 f_2 f_1 + f_2 f_1 f_1}
\ar{drr}[description, pos = 0.30]{f_1f_3 + f_2 f_2 + f_3 f_1}
\ar{drrr}[description, pos = 0.10]{f_4}
&
A^3
\ar{r}
\ar{d}[description, pos = 0.80]{f_1f_1f_1}
\ar{dr}[description, pos = 0.60]{-f_1f_2 + f_2 f_1}
\ar{drr}[description, pos = 0.30]{f_3}
&
A^2
\ar{r}
\ar{d}[description, pos = 0.80]{f_1f_1}
\ar{dr}[description, pos = 0.60]{f_2}
&
A
\ar{d}[description, pos = 0.80]{f_1}
\\
\dots
\ar{r}
&
B^4
\ar{r}
&
B^3
\ar{r}
&
B^2
\ar{r}
&
B.
\end{tikzcd}
\end{equation*}
\end{small}
\end{defn}
\begin{defn}
\label{defn-morphism-of-ainfty-algebras}
A morphism $f_\bullet\colon (A,m_k) \rightarrow (B, n_k)$ of 
$\Ainfty$-algebras is a collection 
$(f_i)_{i \geq 1}$ of degree $1 - i$ morphisms $A^i \rightarrow B$ 
whose bar-construction is a closed degree $0$ morphism of twisted
complexes. 
\end{defn}

As before, this definition translates into the language of twisted
complexes the corresponding morphism of tensor coalgebras. We also 
have the following:

\begin{prps}
\label{prps-defining-equalities-of-Ainfty-morphism}
Let $\A$ be a monoidal DG category,
let $(A, m_k)$ and $(B, n_k)$ be $\Ainfty$-algebras in $\A$, 
and let $(f_i)_{i \geq 1}$ be a collection of degree $1 - i$ morphisms 
$A^i \rightarrow B$.

\begin{enumerate}
\item 
\label{item-defining-equalities-of-Ainfty-morphism-are-subset-of-twisted-complex-conditions}
Take the defining equalities in the  
definition of a morphism of $\Ainfty$-algebras in
\cite[Defn.~1.2.1.2]{Lefevre-SurLesAInftyCategories} 
and replace all terms involving $m_1$ with $d_\A f_i$:
\begin{scriptsize}
\begin{align}
\label{eqn-defining-equalities-for-new-definition-of-Ainfinity-algebra-morphism}
\forall\; i \geq 1 \quad\quad 
d_\A f_i & + 
\sum_{i_1 + \dots + i_r = i} (-1)^{\sum_{2\leq u < r}\left( (1-i_u)
\sum_{1 \leq v \leq u} i_v \right)} m_r \circ \left(f_{i_1} \otimes \dots
\otimes f_{i_r}\right) - 
\\
\nonumber
& -  
\sum_{\begin{smallmatrix}j+k+l = i, \\ k \geq 2\end{smallmatrix}}
(-1)^{jk+l} f_{j+1+l} \circ \left(\id^j \otimes m_k \otimes
\id^l\right) = 0. 
\end{align}
\end{scriptsize}
The resulting equalities are a subset of the set of equalities which 
equate each component $A^i \rightarrow B^j$ 
of the twisted complex map $d\infbar(B_\bullet)$ with zero. 
This subset consists of the equalities for the components $A^i \rightarrow B$.
 
\item 
\label{item-closed-condition-on-f-i-imply-the-rest}
If the equalities 
\eqref{eqn-defining-equalities-for-new-definition-of-Ainfinity-algebra-morphism}
hold, then so do the rest of the equalities for the components of
$d\infbar(B_\bullet)$ to be zero. 
Thus $f_\bullet$ is a morphism of $\Ainfty$-algebra in
the sense of Defn.~\ref{defn-morphism-of-ainfty-algebras}
if and only if $f_i$ satisfy the equalities 
\eqref{eqn-defining-equalities-for-new-definition-of-Ainfinity-algebra-morphism}. 
\end{enumerate} 
\end{prps}
\begin{proof}
Similar to the proof of
Prps.~\ref{prps-comparing-new-and-old-defns-of-ainfty-algebra}. 
% with the following difference. 
% In \eqref{item-closed-condition-on-f-i-imply-the-rest} we can not
% apply the universal property of tensor coalgebra morphisms to the
% differential of the coalgebra morphism 
% $F: \bigoplus (A[1])^i \rightarrow \bigoplus (B[1])^i$ induced by
% $\sum_{i\geq 1} f_i$. This is because $dF$ is not a coalgebra
% morphism. It is, however, an $(F,F)$-coderivation, Therefore
% by the universal property of the coderivations of tensor
% colagebras \cite[Lemme~1.1.2.2]{Lefevre-SurLesAInftyCategories} 
% the map $dF$ is uniquely determined by its $(A[1])^k \rightarrow B[1]$ 
% components. 
\end{proof}

\subsection{$\Ainfty$-modules}
\label{section-Ainfty-modules-in-a-monoidal-category}

We define left and right $\Ainfty$-modules using the same ideas 
we used to define $\Ainfty$-algebras in
\S\ref{section-Ainfty-modules-in-a-monoidal-category}. 
\begin{defn}
\label{defn-right-module-bar-construction-in-a-monoidal-category}
Let $(A,m_i)$ be an $\Ainfty$-algebra in a monoidal DG category $\A$. 
Let $E \in \A$ and let $\left\{p_i\right\}_{i \geq 2}$ be a collection 
of degree $2-i$ morphisms $E \otimes A^{i-1} \rightarrow E$. 
The \em right module bar-construction $\infbar(E)$ \rm of $(E,p_i)$
comprises objects $E \otimes A^{i}$ for $i \geq 0$ 
placed in degree $-i$ and degree $1-k$ maps 
$d_{(i+k)i}\colon E \otimes A^{i+k-1} \rightarrow E \otimes A^{i-1}$ 
with
\begin{scriptsize}
\begin{equation}
\label{eqn-differentials-in-right-module-bar-construction}
d_{(i+k)i} := (-1)^{(i-1)(k+1)} 
\left(\sum_{j = 0}^{i-2} \left( (-1)^{jk} \id^{i-j-1} \otimes m_{k+1} \otimes
\id^{j}\right) + (-1)^{(i-1)k}p_{k+1}\otimes \id^{i-1} \right). 
\end{equation} 
\end{scriptsize}
\begin{tiny}
\begin{equation}
\label{eqn-right-module-bar-construction-of-A-m_i}
\begin{tikzcd}[column sep = 2.4cm]
\dots
\ar{r}[']{\begin{smallmatrix}EA^2 m_2  - EAm_2A + \\ + Em_2A^2 - p_2 A^3 \end{smallmatrix}}
\ar[bend left=20]{rr}[description]{EA m_3 + Em_3A + p_3 A^2 }
\ar[bend left=25]{rrr}[description]{Em_4 - p_4A}
\ar[bend left=30]{rrrr}[description]{p_5}
&
EA^3
\ar{r}[']{EAm_2 - E m_2 A + p_2 A^2}
\ar[bend left=20]{rr}[description]{-Em_3 - p_3A}
\ar[bend left=25]{rrr}[description]{p_4}
& 
EA^2
\ar{r}[']{Em_2 - p_2A}
\ar[bend left=20]{rr}[description]{p_3}
&
EA
\ar{r}[']{p_2}
&
\underset{\degzero}{E}.
\end{tikzcd}
\end{equation}
\end{tiny}
\end{defn}

\begin{defn}
\label{defn-left-module-bar-construction-in-a-monoidal-category}
For $E \in \A$ and a collection 
$\left\{p_i\right\}_{i \geq 2}$ of
maps $ A^{i-1}\otimes E \rightarrow E$ of degree $2-i$,
its \em left module bar-construction $\infbar(E)$ \rm
comprises objects $A^{i}\otimes E$ for $i \geq 0$ placed
in degree $-i$ and degree $1-k$ maps 
$d_{(i+k)i}\colon A^{i+k-1}\otimes E \rightarrow A^{i-1}\otimes E$ 
with
 \begin{equation}
 \label{eqn-differentials-in-left-module-bar-construction}
 d_{(i+k)i} := (-1)^{(i-1)(k+1)} 
 \left(\sum_{j = 1}^{i-1} \left( (-1)^{jk} \id^{i-j-1} \otimes m_{k+1} \otimes
 \id^{j}\right) + \id^{i-1} \otimes p_{k+1} \right). 
 \end{equation} 
% 
% \begin{tiny}
% \begin{equation}
% \label{eqn-left-module-bar-construction-of-A-m_i}
% \begin{tikzcd}[column sep = 2.4cm]
% \dots
% \ar{r}[']{\begin{smallmatrix}A^3 p_2  - A^2m_2E + \\ + Am_2AE - m_2 A^2E \end{smallmatrix}}
% \ar[bend left=25]{rr}[description]{A^2 p_3 + Am_3E + m_3 AE }
% \ar[bend left=30]{rrr}[description]{Ap_4 - m_4E}
% \ar[bend left=35]{rrrr}[description]{p_5}
% &
% A^3E
% \ar{r}[']{A^2p_2 - A m_2 E + m_2 AE}
% \ar[bend left=25]{rr}[description]{-Ap_3 - m_3E}
% \ar[bend left=30]{rrr}[description]{p_4}
% & 
% A^2E
% \ar{r}[']{Ap_2 - m_2E}
% \ar[bend left=25]{rr}[description]{p_3}
% &
% AE
% \ar{r}[']{p_2}
% &
% \underset{\degzero}{E}.
% \end{tikzcd}
% \end{equation}
% \end{tiny}
\end{defn}

Note that the defining formulas
\eqref{eqn-differentials-in-right-module-bar-construction} 
of the right module and the left module bar-constructions can be
obtained from the algebra bar-construction formula 
\eqref{eqn-differentials-in-non-aug-bar-construction}
by replacing each $m_i$ with $p_i$ if its domain involves $E$.
 
\begin{defn}
Let $\A$ be a monoidal DG category and let $(A,m_i)$ be an
$\Ainfty$-algebra in $\A$. A \em right (resp. left) $\Ainfty$-module 
$(E, p_i)$ over $A$ \rm is an object $E \in \A$ and a collection  
$\left\{p_i\right\}_{i \geq 2}$ of degree $2-i$
morphisms $E \otimes A^{i-1} \rightarrow E$ (resp. $A^{i-1} \otimes E
\rightarrow E$) such that $\infbar(E)$ is a twisted complex. 
\end{defn} 

\begin{defn} 
\label{defn-right-module-bar-constructions-of-an-Ainfty-morphism}
Let $\A$ be a monoidal DG category and let $(A,m_i)$ be an
$\Ainfty$-algebra in $\A$.
Let $(E, p_k)$ and $(F, q_k)$ be right $\Ainfty$-modules over $A$ in $\A$. 

A \em degree $j$ morphism \rm $f_\bullet\colon (E,p_k) \rightarrow (F, q_k)$ of 
right $\Ainfty$-$A$-modules is a collection $(f_i)_{i \geq 1}$ of 
degree $j - i + 1$ morphisms $E \otimes A^{i-1} \rightarrow F$. Its
\em bar-construction \rm $\infbar(f_\bullet)$ is 
the morphism $\infbar(E) \rightarrow \infbar(F)$ in $\pretriagmns(\A)$
whose components are  
$$
E \otimes A^{i+k-1} \rightarrow F \otimes A^{i-1}\colon \;
(-1)^{j(i-1)}  f_{k+1}\otimes \id^{i-1}. $$
We illustrate the case when $f_\bullet$ is of odd degree:
\begin{small}
\begin{equation*}
\begin{tikzcd}[column sep = 2cm, row sep = 2cm]
\dots
\ar{r}
&
EA^3
\ar{r}
\ar{d}[description, pos = 0.85]{-f_1A^3}
\ar{dr}[description, pos = 0.65]{f_2 A^2}
\ar{drr}[description, pos = 0.30]{-f_3 A} 
\ar{drrr}[description, pos = 0.10]{f_4}
&
EA^2
\ar{r}
\ar{d}[description, pos = 0.85]{f_1 A^2}
\ar{dr}[description, pos = 0.70]{-f_2A}
\ar{drr}[description, pos = 0.30]{f_3}
&
EA
\ar{r}
\ar{d}[description, pos = 0.85]{-f_1A}
\ar{dr}[description, pos = 0.65]{f_2}
&
E
\ar{d}[description, pos = 0.85]{f_1}
\\
\dots
\ar{r}
&
FA^3
\ar{r}
&
FA^2
\ar{r}
&
FA
\ar{r}
&
F.
\end{tikzcd}
\end{equation*}
\end{small}
\end{defn}

The corresponding definition for the left $\Ainfty$-modules differs only
in signs:
\begin{defn} 
\label{defn-left-module-bar-constructions-of-an-Ainfty-morphism}
Let $\A$ be a monoidal DG category and let $(A,m_i)$ be an
$\Ainfty$-algebra in $\A$.
Let $(E, p_k)$ and $(F, q_k)$ be left $\Ainfty$-modules over $A$ in $\A$. 

A degree $j$ morphism $f_\bullet\colon (E,p_k) \rightarrow (F, q_k)$ of 
left $\Ainfty$-$A$-modules is a collection $(f_i)_{i \geq 1}$ of 
degree $j - i + 1$ morphisms $A^{i-1} \otimes E \rightarrow F$. 
Its \em bar-construction \rm $\infbar(f_\bullet)$ is 
the morphism $\infbar(E) \rightarrow \infbar(F)$ in $\pretriagmns(\A)$
whose components are 
$$ A^{i+k-1} \otimes E \rightarrow A^{i-1} \otimes F \colon \;
(-1)^{(j+k)(i-1)}  \id^{i-1} \otimes f_{k+1}. $$
\end{defn}

We define the DG categories of left and right modules over $A$ in the
unique way where the bar-constructions become
faithful functors into $\pretriagmns(\A)$:
\begin{defn}
Let $\A$ be a monoidal DG category and $A$ an $\Ainfty$-algebra in
$\A$. Define the \em DG category $\nodA$ of
right $\Ainfty$-$A$-modules in $\A$ \rm by:
\begin{itemize}
\item Its objects are right $\Ainfty$-$A$-modules in $\A$,
\item For any $E,F \in \obj \noddinf A$, 
$\homm^\bullet_{\noddinf A}(E,F)$
consists of $\Ainfty$-morphisms $f_\bullet\colon E \rightarrow F$
with their natural grading. The differential and
the composition are induced from those of their bar-constructions. 
\item The identity morphism of $E \in \noddinf A$ is the morphism 
$(f_\bullet)$ with $f_1 = \id_E$ and $f_{\geq 2} = 0$ whose
corresponding twisted complex morphism is $\id_{\infbar(E)}$. 
\end{itemize}
The \em DG category $\Anod$ of left $\Ainfty$-$A$-modules 
in $\A$ \rm is defined analogously.  
\end{defn}

\subsection{$\Ainfty$-bimodules}
\label{section-Ainfty-bimodules-in-a-monoidal-category}

The definitions for $\Ainfty$-bimodules are similar to those for left
and right $\Ainfty$-modules, but we define the bimodule bar-construction
to be a bigraded collection of objects and morphisms whose
totalisation can be identified with the coderivations and the
morphisms of the corresponding free bi-comodules. It 
is thus a twisted bicomplex, 
see \cite{AnnoLogvinenko-UnboundedTwistedComplexes}, \S 4.

\begin{defn}
\label{defn-bimodule-bar-construction-in-a-monoidal-category}
Let $(A,m_i)$ and $(B,n_i)$ be $\Ainfty$-algebras 
in a monoidal DG category $\A$. 
Let $M \in \A$ and let $\left\{p_{ij}\right\}_{i + j \geq 1}$ be a collection 
of degree $1-i-j$ morphisms $A^i \otimes M \otimes B^{j} \rightarrow M$. 
The \em bimodule bar-construction $\infbar(M)$ \rm 
comprises objects $A^i \otimes M \otimes B^{j}$ with $i,j \geq 0$ 
placed in bidegree $-i,-j$ and degree $1+p+q-k-l$ maps 
$A^{p} \otimes M \otimes B^{q} \rightarrow A^i \otimes M \otimes B^j$ 
best defined as the components of the maps 
$ A^{p} \otimes M \otimes B^{q} \longrightarrow 
\bigoplus_{i + j = p + q - k} A^i \otimes M \otimes B^j $
defined by
\begin{equation}
\label{eqn-differentials-in-the-bimodule-bar-construction}
(-1)^{(i+j)(k+1)} 
\sum_{r = 0}^{i+j} (-1)^{rk} \id^{i + j - r} \otimes (pmn)_{k+1}
\otimes \id^{r},
\end{equation} 
where $(pm)_{k+1}$ denotes the unique operation ––– 
either $p_{s,t}$ with $s+t = k$ or $m_{k+1}$ or $n_{k+1}$ --- 
that can be applied to the corresponding length $k+1$ factor
of $A^{p} \otimes M \otimes B^{q}$. 

% \begin{tiny}
% \begin{equation*}
% \begin{tikzcd}[row sep=3cm, column sep = 3.5cm]
% A^2MB^2
% \ar{r}[description]{A^2Mn_2 - A^2p_{01}B}
% \ar{d}[pos = 0.75, description]{Ap_{10}B^2 - m_2MB^2}
% \ar[bend left=10]{rr}[pos=0.8, description]{A^2p_{02}}
% \ar[bend right=35]{dd}[pos=0.85, description]{p_{20}B^2}
% \ar{rd}[pos = 0.7, description]{Ap_{11}B}
% \ar{rdd}[pos = 0.86, description]{-p_{21}B}
% \ar{rrd}[pos = 0.8, description]{Ap_{12}}
% \ar[bend right=10]{rrdd}[pos = 0.8, description]{p_{22}}
% &
% A^2MB
% \ar{r}[pos = 0.7, description]{A^2p_{01}}
% \ar{d}[pos = 0.78, description]{-Ap_{10}B + m_2MB}
% \ar[bend left=35]{dd}[pos = 0.86, description]{-p_{20}B}
% \ar{rd}[pos = 0.7, description]{-Ap_{11}}
% \ar{rdd}[pos=0.75, description]{p_{21}}
% &
% A^2 M
% \ar{d}[pos = 0.6, description]{Ap_{10} - m_2M}
% \ar[bend left=35]{dd}[pos=0.85, description]{p_{20}}
% \\
% AMB^2
% \ar{r}[pos = 0.65, description]{AMn_2 - Ap_{01}B}
% \ar{d}[description]{p_{10}B^2}
% \ar[bend left=10]{rr}[pos = 0.83, description]{-Ap_{02}}
% \ar{rd}[pos = 0.83, description]{-p_{11}B}
% \ar{rrd}[pos = 0.75, description]{p_{12}}
% &
% AMB
% \ar{r}[pos = 0.75, description]{Ap_{01}}
% \ar{d}[pos = 0.62, description]{-p_{10	}B}
% \ar{rd}[description]{p_{11}}
% &
% AM
% \ar{d}[description]{p_{10}}
% \\
% MB^2
% \ar{r}[pos = 0.75, description]{Mn_2 - p_{01}B}
% \ar[bend right = 10]{rr}[pos = 0.9, description]{p_{02}}
% &
% MB
% \ar{r}[pos = 0.55, description]{p_{01}}
% &
% M
% \end{tikzcd}
% \end{equation*}
% \end{tiny}
\end{defn}
Note that the defining formula
\eqref{eqn-differentials-in-the-bimodule-bar-construction}
in the bimodule bar-construction is
almost identical to the analogous 
formula \eqref{eqn-differentials-in-non-aug-bar-construction}
in the algebra bar-construction, only with 
the bimodule operations $p_{s,t}$ with $s+t = k$ being used 
along with the algebra operations $m_{k+1}$ and $n_{k+1}$ as approriate, 
and the formula now defining simultaneously the differentials 
from a given $A^{p} \otimes M \otimes B^q$ to 
all $A^{i} \otimes M \otimes B^j$ with $i + j = p + q - k$.   

\begin{defn}
An \em $\Ainfty$-bimodule \rm over $\Ainfty$-algebras $(A,m_i)$ and $(B,n_i)$ 
in a monoidal DG category $\A$ is
an object $M \in \A$ and a collection $\left\{p_{ij}\right\}_{i + j \geq 1}$ 
of degree $1-i-j$ morphisms $A^i \otimes M \otimes B^{j} \rightarrow M$
such that $\infbar(M)$ is a twisted complex. 
\end{defn}

The remaining definitions are then analogous to those for left and
right modules:

\begin{defn} 
Let $\A$ be a monoidal DG category and let $(A,m_i)$ and $(B, n_i)$ be 
$\Ainfty$-algebras in $\A$. Let $(M, p_{ij})$ and $(N, q_{ij})$ be
$\Ainfty$-$A$-$B$-bimodules. 

A \em degree $k$ morphism \rm $f_{\bullet\bullet}\colon 
(M,p_{ij}) \rightarrow (N, q_{ij})$ of 
$\Ainfty$-$A$-$B$-bimodules is a collection $(f_{lm})_{l + m \geq 0}$ of 
degree $k - l - m$ morphisms $A^l \otimes M \otimes B^{m} \rightarrow N$. 
Its \em bar-construction \rm $\infbar(f_{\bullet\bullet})$ is 
the morphism $\infbar(M) \rightarrow \infbar(N)$ in $\twbiosmns(\A)$
whose components are  
$$
A^{i+l} \otimes M \otimes B^{j+m} \rightarrow A^i \otimes N \otimes B^{j}\colon \;
(-1)^{i(l+m) + k (i+j)}  \id^{i} \otimes f_{l,m} \otimes \id^{j}. $$
% We illustrate the case when $f_{\bullet\bullet}$ is of odd degree:
% \begin{equation}
% \begin{tikzcd}[row sep = 1.25cm]
% AMB
% \ar[dotted]{r}
% \ar[dotted]{d}
% \ar[dotted]{dr}
% \ar[bend right = 15]{ddrr}[pos = 0.6, description]{Af_{00}B}
% \ar[bend right = 20]{dddrrr}[pos = 0.52, description]{f_{11}}
% \ar[bend right = 17]{dddrr}[pos = 0.525, description]{- f_{10}B}
% \ar[bend left = 10]{ddrrr}[pos = 0.6, description]{Af_{01}}
% &
% AM
% \ar[dotted]{d}
% \ar[bend left = 25]{ddrr}[pos=0.5, description]{-Af_{00}}
% \ar[bend left = 27]{dddrr}[pos=0.45, description]{f_{10}}
% &
% &
% \\
% MB
% \ar[dotted]{r}
% \ar[bend right = 25]{ddrr}[pos=0.35, description]{-f_{00}B~}
% \ar[bend right = 25]{ddrrr}[pos=0.35, description]{f_{01}}
% &
% M
% \ar[bend left = 17]{ddrr}[pos = 0.3, description]{f_{00}}
% &
% &
% \\
% &
% &
% ANB
% \ar[dotted]{r}
% \ar[dotted]{d}
% \ar[dotted]{dr}
% &
% AN
% \ar[dotted]{d}
% \\
% &
% &
% NB
% \ar[dotted]{r}
% &
% N.
% \end{tikzcd}
% \end{equation}
% 
 \end{defn}
 
\begin{defn}
Let $(A,m_i)$ and $(B,n_i)$ be $\Ainfty$-algebras in a monoidal DG category 
$\A$. Define the \em DG category $\AnodB$ of
$\Ainfty$-$A$-$B$-bimodules in $\A$ \rm by:
\begin{itemize}
\item Its objects are $\Ainfty$-$A$-$B$-bimodules in $\A$,
\item For any $M,N \in \obj A\text{-}\noddinf\text{-}B$, the complex 
$\homm^\bullet_{A\text{-}\noddinf\text{-}B}(M,N)$
consists of $\Ainfty$-morphisms $f_{\bullet\bullet}\colon M \rightarrow N$
with their natural grading. The differential and the composition
are induced from those of their bar-constructions.
\item The identity morphism of $M \in A\text{-}\noddinf\text{-}B$ is 
the morphism $(f_{\bullet\bullet})$ with 
$f_{00} = \id_M$ and $f_{ij} = 0$ for $i + j \geq 1$ whose
corresponding twisted complex morphism is $\id_{\infbar(M)}$. 
\end{itemize}
\end{defn}

Next, we show that the category of $\Ainfty$-bimodules is isomorphic
to the category of left $\Ainfty$-modules in the category of right
$\Ainfty$-modules and vice versa. However, since $\nodA$ and $\Anod$
do not have a natural monoidal structure, we first need to explain 
how such setup fits into our framework of definitions. 
\begin{defn}
Define the DG bicategory 
\begin{equation}
\nodAbic
:= 
\quad 
\begin{tikzcd}[column sep = 5em]
\bullet
\ar{r}{\nodA}
&
\bullet 
\ar[out=30, in=-30,loop,distance=6em]{}{\A}
\end{tikzcd}
\end{equation}
whose $1$-composition is given by the monoidal structure on $\A$ and
by the functor $\A \otimes \nodA \rightarrow \nodA$
which sends $(Q, (E,p_i))$ to $(Q \otimes E, \id \otimes p_i)$. 
Define similarly the DG bicategory $\Anodbic$. 
% \begin{equation}
% \Anodbic
% := 
% \quad 
% \begin{tikzcd}[column sep = 5em]
% \bullet 
% \ar[out=150, in=-150,loop,distance=6em]{}[']{\A}
% \ar{r}{\Anod}
% &
% \bullet
% \end{tikzcd}
% \end{equation} 
\end{defn}

Any other $\Ainfty$-algebra in $\A$ is also an $\Ainfty$-algebra in $\nodAbic$
and $\Anodbic$. We can therefore define:
\begin{defn}
Let $(A,m_i)$ and $(B,n_i)$ be $\Ainfty$-algebras in 
a monoidal DG category $\A$. 
Define $(\Anod)\text{-}B$ to be the DG category of right 
$\Ainfty$-$B$-modules in $\Anodbic$ whose underlying $1$-morphisms 
lie in $\Anod$. Similarly, define $A\text{-}\left(\nodB\right)$ 
to be the DG category of left $\Ainfty$-$A$-modules in $\nodBbic$ 
whose underlying $1$-morphisms lie in $\nodB$.   
\end{defn}

In  \cite[Defn. 4.2]{AnnoLogvinenko-UnboundedTwistedComplexes} we define the functors 
\begin{align*}
\cxrow\colon 
\twcx_\B^{\pm}\left(\twcx_\B^{\pm}\left(\A\right)\right) 
& \rightarrow 
\twbicx_\B^{\pm}(\A), \\\
\cxcol\colon 
\twcx_\B^{\pm}\left(\twcx_\B^{\pm}\left(\A\right)\right) 
& \rightarrow 
\twbicx_\B^{\pm}(\A), 
\end{align*}
that make a bicomplex out of a twisted complex of twisted complexes, interpreting
the complexes (with a sign twist) as rows or columns of the bicomplex respectively.

\begin{theorem}
\label{theorem-ainfty-bimodules-are-ainfty-modules-in-cat-of-ainfty-modules}
Let $(A,m_i), (B,n_i)$ be $\Ainfty$-algebras in 
a monoidal DG category $\A$. 
There exist isomorphisms of DG categories 
\begin{equation}
(\Anod)\text{-}B 
\simeq 
\AnodB 
\simeq
A\text{-}\left(\nodB\right) 
\end{equation}
intertwining the functors $\cxcol \circ \infbarA \circ \infbarB$, 
$\infbar^{A\text{-}B}$, and $\cxrow \circ \infbarB \circ \infbarA$ into 
$\twbiosmns \A$. 
\end{theorem}
\begin{proof}
It can be readily checked that the fully faithful functors 
$$ \cxcol \circ \infbar^A \circ \infbar^B\colon 
(\Anod)\text{-}B \rightarrow \twbiosmns \A, $$
$$ \infbar^{A \text{-} B}\colon 
\AnodB \rightarrow \twbiosmns \A, $$
have the same image in $\twbiosmns \A$ and hence we have
an isomorphism 
$$ (\Anod)\text{-}B \rightarrow \AnodB. $$

The other isomorphism is obtained considering
fully faithful functors $\infbar^{A \text{-} B}$ and
$$\cxrow \circ \infbar^B \circ \infbar^A\colon 
A\text{-}(\nodB) \rightarrow \twbiosmns \A.$$
Their images in $\twbiosmns \A$ do not coincide, however they 
are identified by the sign-twisting automorphism of $\twbiosmns \A$
which sends a twisted bicomplex $(a_{ij}, \alpha_{ijkl})$ to 
$(a_{ij}, (-1)^{ij+kl} \alpha_{ijkl})$
and a morphism $(f_{ijkl})$ of twisted bicomplexes to
$((-1)^{ij+kl} f_{ijkl})$. 
\end{proof}

\subsection{Free modules and bimodules}
\label{section-free-modules-and-bimodules-over-Ainfty-algebra}

\begin{defn}
\label{defn-free-Ainfty-modules}
Let $\A$ be a monoidal DG category and $(A,m_i)$ be an $\Ainfty$-algebra
in it. 
Let $E \in \A$. The \em free right $\Ainfty$-$A$-module generated by $E$ \rm is the module $(E \otimes A, p_i)$ where 
$ p_i\colon E \otimes A^i \rightarrow E \otimes A $
is the map $\id_E \otimes m_i$. 
\end{defn}

We can view any $A^i$ with $i \geq 1$ as the free right 
$\Ainfty$-$A$-module generated by $A^{i-1}$.  

\begin{defn}
Let $\A$ be a monoidal DG category and $(A,m_i)$ be an $\Ainfty$-algebra
in it. 
The \em category $\freeA$ of free right $\Ainfty$-$A$-modules \rm
is the full subcategory of $\nodA$ consisting of free modules. 
The \em free module \rm functor 
$\free\colon \A \rightarrow \freeA$, 
sends each $E \in \A$ to $E \otimes A$ and each 
$\alpha \in \homm_\A(E,F)$ to the morphism whose 
first component is $\alpha \otimes \id$ and whose higher components are zero. 
The \em forgetful \rm functor $\forget\colon \nodA \rightarrow \A$
sends any $(E,p_\bullet) \in \nodA$ to $E \in \A$ and 
any $f_\bullet \in
\homm_{\freeA}((E,p_\bullet),(F,q_\bullet))$ to 
$f_1 \in \homm_\A(E,F)$. 
\end{defn}

The corresponding notions for left modules and bimodules are defined
similarly.

\subsection{Yoneda embedding}
\label{section-Yoneda-embedding-of-A-modules-into-repA-modules}

We state the notions and results of sections \S\ref{section-Yoneda-embedding-of-A-modules-into-repA-modules}-\ref{section-bar-resolution} 
only for right $\Ainfty$-$A$-modules. The left module case is similar. 

The standard Yoneda embedding of $\nodA$ is the fully faithful functor 
\begin{equation}
\label{eqn-the-standard-yoneda-for-nodA}
\nodA \rightarrow \modd\text{-}\left(\nodA\right)
\end{equation}
which sends $(E,p_\bullet)$ to 
$\homm_{\nodA}(-, (E,p_\bullet))$. We now define a more economical 
version of this, where the Yoneda embedding is only used 
to embed $E$ into $\modA$. 

In \cite[\S4.5]{GyengeKoppensteinerLogvinenko-TheHeisenbergCategoryOfACategory}
it was shown that if $\A$ is a monoidal DG category, then there is a
natural monoidal structure on $\modA$ such that the Yoneda embedding
$\A \hookrightarrow \modA$ is strong monoidal. 
An $\Ainfty$-algebra $A$ in $\A$ can therefore be also viewed as 
an $\Ainfty$-algebra in $\modA$. Where the difference matters, we 
write $A$ and $\repA$ for the corresponding algebras in $\A$ and 
$\modA$, respectively.

\begin{defn}
\label{defn-yoneda-embedding-into-repA-modules}
Let $(A,m_\bullet)$ be an $\Ainfty$-algebra in a monoidal DG category $\A$.
Define the \em Yoneda embedding \rm 
\begin{equation}
\label{eqn-yoneda-embedding-into-T^*-modules}
\Upsilon\colon \nodA \rightarrow \nodrepA
\end{equation}
to be the fully faithful functor induced by the monoidal embedding
$\A \hookrightarrow \modA$.  
\end{defn}

The following is both a direct analogue of the classical result
for modules over ordinary monads in \cite[\S5]{Street-TheFormalTheoryOfMonads} 
and its homotopy generalisation:

\begin{prps}
\label{prps-yoneda-embedding-gives-both-fiber-and-homotopy-fiber-square}
The following is both a fiber square of DG-categories and a homotopy 
fiber square thereof: 
\begin{equation}
\begin{tikzcd}[column sep = 2cm]
\label{eqn-yoneda-embedding-of-Ainfty-A-modules-as-homotopy-fiber-square}
\nodA
\ar[hookrightarrow]{r}{\yoneda} 
\ar{d}{\forget}
&
\nodrepA
\ar{d}{\forget}
\\
\A
\ar[hookrightarrow]{r}{\yoneda} 
&
\modA.
\end{tikzcd}
\end{equation}
\end{prps}
\begin{proof}
For a fiber square, as the Yoneda embedding of $\A$ into
$\modA$ is fully faithful, it suffices to show that $\nodA$ is the
full subcategory of $\nodrepA$ consisting of $(E,p_\bullet)$ with $E
\in \A$. Since $\A$ is a full monoidal subcategory of $\modA$, 
for any such $(E,p_\bullet)$ its structure morphisms 
$p_i\colon E\otimes A^{i-1} \rightarrow E$ all lie in $\A$, and so 
do all components of all $\Ainfty$-morphisms of such modules.  

To show that 
\eqref{eqn-yoneda-embedding-of-Ainfty-A-modules-as-homotopy-fiber-square}
is a homotopy fiber square, it suffices to show that 
\begin{equation}
\label{eqn-the-forgetful-functor-from-nodrepA-to-modA}
\forget\colon \nodrepA \rightarrow \modA, 
\end{equation}
is a fibration. This is because 
the category of small DG-categories with the Dwyer-Kan
model structure due to Tabuada
\cite{Tabuada-UneStructureDeCategorieDeModelesDeQuillenSurLaCategorieDesDG-Categories}
is right proper, and in a right proper model category any pushout
along a fibration is also the homotopy pushout, see 
\cite[IX, \S4.1]{BousfieldKan-HomotopyLimitsCompletionsAndLocalizations}
\cite[Prop.~1.19]{Barwick-OnLeftAndRightModelCategoriesAndLeftAndRightBousfieldLocalizations}. 

Now, a DG-functor is a fibration if
it is surjective on complexes of morphisms and reflects homotopy
equivalences, see 
\cite[Defn.~2.1]{Toen-TheHomotopyTheoryOfDGCategoriesAndDerivedMoritaTheory}. 
The functor \eqref{eqn-the-forgetful-functor-from-nodrepA-to-modA}
satisfies the former as it has a right inverse
which sends any $E \in \modA$ to $(E,p_\bullet)$ with 
$p_i = 0$ for all $i$
and sends any $f\colon E \rightarrow F$ to $(f_\bullet)$ with 
$f_1 = f$ and $f_i = 0$ for $i > 1$. 
It satisfies the latter as given $(E,p_\bullet) \in \nodrepA$ 
and a homotopy equivalence 
$\phi\colon E \xrightarrow{\sim} F$ in $\modA$, 
we can transfer the $\Ainfty$-structure along $\phi$ 
(\cite{AnnoLogvinenko-UnboundedTwistedComplexes}, Theorem A.2)
to construct $(F,q_\bullet)$ homotopy equivalent to $(E, p_\bullet)$. 
\end{proof}

An useful consequence of the universal property of a homotopy fiber square is: 
\begin{cor}
Let $(E,p_\bullet) \in \nodrepA$. It is homotopy equivalent to
some $(F,q_\bullet) \in \nodA$ if and only if $E \in \modA$ is homotopy
equivalent to $F \in \A$. 
\end{cor}

\subsection{Twisted complexes of $\Ainfty$-modules}
\label{section-twisted-complexes-of-ainfty-modules}

We now fix our conventions regarding twisted 
complexes of $\Ainfty$-modules. As explained in
\S\ref{section-unbounded-twisted-complexes} unbounded 
twisted complexes over a DG category $\A$ need to be defined relative 
to its fully faithful embedding into some DG category closed 
under shifts and infinite direct sums. 

Our conventions are: we define $\twcxub \A$ and $\twcxub \modA$ 
relative to $\modA$. Thus twisted complexes in $\twcxub \A$ and
$\twcxub \modA$ can have infinite number of differentials and/or morphism 
components emerge from a single object, but only if their sum still defines 
a morphism in $\modA$. Since $\modA$ admits change of differential, 
the convolution functor $\twcxub \modA \hookrightarrow \modA$ 
is an equivalence. 

We define both $\twcxub \nodA$ and $\twcxub \nodrepA$ relative to $\nodrepA$. 
Since the latter admits convolutions of unbounded twisted complexes
\cite[Cor.~5.14]{AnnoLogvinenko-UnboundedTwistedComplexes}, 
the natural embedding $\nodrepA \hookrightarrow \twcxub \nodrepA$
is an equivalence. On the other hand, $\nodA \hookrightarrow \twcxub
\nodA$ is an equivalence if and only if $\A \hookrightarrow \twcxub \A$
is, and similarly for $\twcxpls$, $\pretriagmns$, et cetera \cite[Cor.~5.13]{AnnoLogvinenko-UnboundedTwistedComplexes}. 

\subsection{The homotopy lemma}
\label{section-the-homotopy-lemma}
In this section we give an analogue of the theorem that all $\Ainfty$-quasi-isomorphisms of usual $\Ainfty$-modules are homotopy equivalences,
cf.~\cite[Prop.~2.4.1.1]{Lefevre-SurLesAInftyCategories}. 
This is proved by showing that a usual $\Ainfty$-module is
null-homotopic if and only if its underlying complex of vector spaces 
is acyclic.

In our generality, it is not clear how to even state 
the latter result, since there are no underlying complexes of vector spaces. 
However, a complex of vector spaces is acyclic if and only if it 
is null-homotopic in $\modk$. We thus define:

\begin{defn}
\label{defn-acyclicity-of-modules-and-twisted-complexes}
An $\Ainfty$-$A$-module $(E,p_\bullet)$ is \em acyclic \rm if 
$E$ is null-homotopic~in~$\A$. A twisted complex 
in $\twcxub\nodA$ is \em acyclic \rm if its underlying
twisted complex in $\twcxub\A$ is null-homotopic.  
\end{defn}

We then have: 

\begin{lemma}[The homotopy lemma]~
\label{lemma-the-homotopy-lemma-for-nodA}
\begin{enumerate}
\item 
\label{item-acyclic-if-and-only-if-acyclic-in-A}
An $\Ainfty$-$A$-module $(E,p_\bullet)$ is null-homotopic  
if and only if it is acyclic. 
\item 
\label{item-complex-acyclic-if-and-only-if-acyclic-in-pretriagA}
A twisted complex in $\twcxub\nodA$ is
null-homotopic if and only if it is acyclic. 
\item 
\label{item-homotopy-equivalence-if-and-only-if-one-in-A}
A closed degree zero 
morphism $f_\bullet\colon (E,p_\bullet) \rightarrow (F,q_\bullet)$
is a homotopy equivalence in $\nodA$ if and only if $f_1$
is a homotopy equivalence in $\A$. 
\end{enumerate}
\end{lemma}

\begin{proof}
\eqref{item-acyclic-if-and-only-if-acyclic-in-A}:
The ``only if'' implication is clear. For the ``if'' implication, let
$h$ be the contracting homotopy of $E$ in $\A$:
a degree $-1$ endomorphism of $E$ with $dh = \id_E$. It suffices 
to construct the contracting homotopy $h_\bullet$ of $(E,p_\bullet)$ in 
$\nodA$. In other words, we need to find $h_\bullet$ such that 
$$ d(h_1, h_2, \dots , h_n, \dots) = (\id_E, 0, 0, \dots, 0, \dots). $$

By our definition of $\nodA$, 
the differential of any $x_\bullet: (E,p_\bullet) \rightarrow (E,p_\bullet)$ is
computed by differentiating the map $\infbar x_\bullet$ of twisted
complexes over $\A$. Note, that the first $n$ terms of $dx_\bullet$ 
are completely determined by $x_1, \dots, x_n$. Define 
$$ d(x_1, \dots, x_n) := \bigl( (dx_\bullet)_1, \dots, (dx_\bullet)_n \bigr). $$ 
In this notation, the diagram above makes it clear that for any 
$x_{n+1}\colon EA^n \rightarrow E$ 
\begin{align*}
d(0,\dots, 0, x_{n+1}) & = (0,\dots, 0, dx_{n+1}). 
\end{align*}

We now proceed by induction. Suppose we've found $h_1, \dots, h_n$ such that 
\begin{equation*}
d(h_1, h_2, \dots, h_n) = (\id_E, 0, \dots, 0)  
\end{equation*}
We start with $n = 1$ where we set $h_1 = h$. 
For the induction step, observe that  
$$ d(h_1, h_2, \dots, h_n,0) = (\id_E, 0, \dots,0, x_{n+1}), $$
for some $x_{n+1}\colon EA^n \rightarrow E$. The map $(\id_E, 0, \dots, x_{n+1})$ 
is a boundary and the map $(\id_E, 0, \dots, 0)$ is closed. 
So $(0,\dots, 0, x_{n+1})$ must also be closed. We conclude that 
that $dx_{n+1} = 0$, and hence $d(h \circ
x_{n+1})= x_{n+1}$. Set 
$h_{n+1} = - h \circ x_{n+1}$, then  
\begin{align*}
d(h_1, \dots, h_n, h_{n+1}) & = 
d(h_1, \dots, h_n, 0) + 
d(0, \dots, 0 , h_{n+1}) = \\
& = (\id_E, 0, \dots, 0, x_{n+1}) + 
(0,0,\dots, 0, -x_{n+1}) \\ 
& = (\id_E, 0, \dots, 0, 0). 
\end{align*}

\eqref{item-complex-acyclic-if-and-only-if-acyclic-in-pretriagA}:
By \cite{AnnoLogvinenko-UnboundedTwistedComplexes}, Prop.~5.12
we have a fully faithful functor
$$ \Phi\colon \twcxub\left(\nodA\right) \hookrightarrow \nodtwcxA. $$
By construction, $\Phi$  commutes with the forgetful functors
into $\twcxub \A$. The desired assertion follows 
by applying \eqref{item-acyclic-if-and-only-if-acyclic-in-A} 
to $\nodtwcxA$ or rather to $\nodtwcxrepA$. 

\eqref{item-homotopy-equivalence-if-and-only-if-one-in-A}:
The morphism $f_\bullet\colon (a,p_\bullet) \rightarrow
(b,q_\bullet)$ is a homotopy equivalence if and only if \begin{equation}
\label{eqn-two-step-twisted-complex-in-nodA}
(a,p_\bullet) \xrightarrow{f_\bullet} \underset{\degzero}{(b,q_\bullet)},
\end{equation}
is null-homotopic in $\pretriag \nodA$. 
The desired assertion follows by applying 
\eqref{item-complex-acyclic-if-and-only-if-acyclic-in-pretriagA}
to the twisted complex \eqref{eqn-two-step-twisted-complex-in-nodA}. 
\end{proof}

\subsection{Bar-construction as a complex of $\Ainfty$-$A$-modules}
\label{section-bar-construction-as-a-complex-of-ainfty-A-modules}

We now upgrade the bar-construction of an $\Ainfty$-$A$-module 
from a twisted complex of objects of $\A$ to a twisted complex of 
$\Ainfty$-$A$-modules. 

We follow the method detailed in 
\cite[\S2.10]{AnnoLogvinenko-BarCategoryOfModulesAndHomotopyAdjunctionForTensorFunctors}:
\begin{defn}
\label{defn-morphisms-pi_i-and-mu_i}
Let $(A,m_\bullet)$ be an $\Ainfty$-algebra in a monoidal DG category $\A$.
Let $(E,p_\bullet)$ be a right $\Ainfty$-$A$-module. 

For any $i \geq 1$ consider $E \otimes A^i$ as the free $A$-module  
$E \otimes A^{i-1} \otimes A$.
Define a morphism of right $\Ainfty$-$A$-modules
$$ \pi_i\colon E \otimes A^{i-1} \rightarrow (E, p_\bullet) $$ 
by setting each $(\pi_i)_j$ to be the map
$$ E \otimes A^{i-1} \otimes A^{j-1} \xrightarrow{(-1)^{(i-1)(j-1)}p_{i + j - 1}} E. $$

For left $\Ainfty$-$A$-modules, the definition is similar,  but with $(\pi_i)_j = p_{i+j-1}$. 
\end{defn}

When $E$ is a free module $F \otimes A$, 
we write instead of $\pi_{i}$
$$ \id \otimes \mu_i \colon F \otimes A^i \rightarrow F \otimes A $$ 
for the corresponding $\Ainfty$-morphism because its
$j$-th component is $(-1)^{(i-1)(j-1)}\id \otimes m_{i+j-1}$. 
When $F = \id_\A$, we further write $\mu_i\colon  A^i \rightarrow A$
for this $\Ainfty$-morphism. 

On the other hand, given a morphism $\alpha \colon E \rightarrow F$
in $\A$, by abuse of notation we write $\alpha \otimes \id$ for
the right $\Ainfty$-$A$-module morphism 
$$ \free(\alpha)\colon E \otimes A \rightarrow F \otimes A. $$
It is the strict $\Ainfty$-morphism whose first component is $\alpha
\otimes \id$. 

\begin{prps}
\label{prps-bar-construction-as-a-complex-of-Ainfty-modules}
Let $(A,m_\bullet)$ be an $\Ainfty$-algebra in a monoidal DG category $\A$. Let $(E,p_\bullet)$ be a right $\Ainfty$-$A$-module. 

Take the right module bar-construction of $(E,p_\bullet)$ and lift it from $\A$ to $\nodA$ as follows:
\begin{itemize}
	\item Replace its objects by $(E,p_\bullet)$ and free $\Ainfty$-modules $E \otimes A^i$ with $ i > 0$.  
	\item For differentials, in the summands where the operation $m_{k}$ or $p_{k}$ doesn't involve the rightmost copy of $A$ replace this operation by its image under the free functor. 
	\item In those summands where it does --- replace it by $\mu_{k}$ or $\pi_{k}$, respectively.  
\end{itemize}

Then the result is a twisted complex of right $\Ainfty$-$A$-modules:
\begin{tiny}
\begin{equation}
\label{eqn-right-module-bar-construction-Ainfty-version}
\begin{tikzcd}[column sep = 2.4cm]
\dots
\ar{r}[']{\begin{smallmatrix}EA^2 \mu_2  - EAm_2A + \\ + Em_2A^2 - p_2 A^3 \end{smallmatrix}}
\ar[bend left=20]{rr}[description]{EA \mu_3 + Em_3A + p_3 A^2 }
\ar[bend left=25]{rrr}[description]{E\mu_4 - p_4A}
\ar[bend left=30]{rrrr}[description]{\pi_5}
&
EA^3
\ar{r}[']{EA\mu_2 - E m_2 A + p_2 A^2}
\ar[bend left=20]{rr}[description]{-E\mu_3 - p_3A}
\ar[bend left=25]{rrr}[description]{\pi_4}
& 
EA^2
\ar{r}[']{E\mu_2 - p_2A}
\ar[bend left=20]{rr}[description]{\pi_3}
&
EA
\ar{r}[']{\pi_2}
&
\underset{\degzero}{E}.
\end{tikzcd}
\end{equation}
\end{tiny}
\end{prps}

\begin{proof}
See 
\cite[\S2.10]{AnnoLogvinenko-BarCategoryOfModulesAndHomotopyAdjunctionForTensorFunctors}. 
\end{proof}

We then upgrade $\infbar$ from a functor
into $\pretriagmns(\A)$ to a functor into $\pretriagmns(\nodA)$:
\begin{defn}
\label{defn-bar-construction-as-functor-into-pretriagmns-nodA}
\begin{enumerate}
\item For any object $(E, p_\bullet) \in \nodA$ 
define
$$ \infbar(E,p_\bullet) \in \pretriagmns(\nodA) $$
to be the twisted complex 
\eqref{eqn-right-module-bar-construction-Ainfty-version}. 

\item For any morphism $f_\bullet\colon (E, p_\bullet) \rightarrow  
(F, q_\bullet)$ in $\nodA$ define
$$ \infbar(f_\bullet) \in
\homm_{\pretriagmns(\nodA)}\left(\infbar(E,p_\bullet),
\infbar(F,q_\bullet)\right) $$
to be the lift of the bar-construction of $f_\bullet$ from $\A$ to $\nodA$ where in each component where $f_k$ doesn't involve the rightmost $A$ we replace $f_k$ by its image under the free functor, and in each component where it does we replace it by the morphism $f_{\bullet + k}\colon E \otimes A^k \rightarrow (F,q_\bullet)$ defined by $ (f_{\bullet + i})_j = f_{j+i}$. 

As usual, we illustrate the case when $f_\bullet$ is of odd degree:
\end{enumerate}
\begin{scriptsize}
\begin{equation}
\label{eqn-bar-complex-map-corresponding-to-Ainfty-morphism-Ainfty-version}
\begin{tikzcd}[row sep=1.5cm, column sep = 2.0cm]
\dots
\ar{r}{}
& 
EA^3
\ar{r}
\ar{d}[description, pos=0.6]{-f_1 A^3}
\ar{dr}[description, near start]{f_2 A^2}
\ar{drr}[description, near start]{-f_3A}
\ar{drrr}[description, near start]{f_{\bullet + 3}}
& 
EA^2
\ar{r}
\ar{d}[description, pos=0.6]{f_1A^2}
\ar{dr}[description, near start]{-f_2A}
\ar{drr}[description, near start]{f_{\bullet +2}}
&
EA
\ar{r}
\ar{d}[description, pos=0.6]{-f_1A}
\ar{dr}[description, near start]{f_{\bullet+1}}
&
(E,p_\bullet) 
\ar{d}[description, pos=0.6]{f_{\bullet}}
\\
\dots
\ar{r}{}
& 
FA^3
\ar{r}
& 
FA^2
\ar{r}
&
FA
\ar{r}
&
\underset{\degzero}{(F,q_\bullet)},
\end{tikzcd}
\end{equation}
\end{scriptsize}
\end{defn}

\begin{prps}
The assignments in 
Definition \ref{defn-bar-construction-as-functor-into-pretriagmns-nodA}
define the DG-functor
\begin{equation}
\infbar: \nodA \rightarrow \pretriagmns(\nodA)
\end{equation}
such that the following diagram of functors commutes:
\begin{equation}
\label{eqn-two-infbar-interwoven-by-forgetful-functor-nodA-to-A}
\begin{tikzcd} 
\nodA
\ar{r}{\infbar}
\ar{dr}[']{\infbar}
& 
\pretriagmns(\nodA)
\ar{d}{\forget}
\\
&
\pretriagmns(\A).
\end{tikzcd}
\end{equation}
\end{prps}
\begin{proof}
See 
\cite[\S2.10]{AnnoLogvinenko-BarCategoryOfModulesAndHomotopyAdjunctionForTensorFunctors}.  
\end{proof}

\subsection{Bar-resolution}
\label{section-bar-resolution}

Having realised the bar-construction of an $\Ainfty$-$A$-module as 
a twisted complex in $\nodA$, we arrive naturally at the notion of the 
bar-resolution. It is the free part of the bar-construction. 
Thus, when the bar-construction is acyclic, it gives a resolution 
of an $\Ainfty$-$T$-module by free modules:

\begin{defn}
\label{defn-bar-resolution-functor-for-Ainfty-A-modules}
Define the \em bar-resolution \rm DG-functor
\begin{equation}
\infbarres\colon 
\nodA \rightarrow \pretriagmns(\freeA)
\end{equation}
by setting $\infbarres(E,p_\bullet)$ to be the shifted subcomplex 
$\infbar(E,p_\bullet)_{\deg \leq -1}[-1]$
of $\infbar(E, p_\bullet)$
\begin{equation}
\label{eqn-bar-resolution-of-Ainfty-A-module}
\begin{tikzcd}[column sep = 2.75cm]
\dots
\ar{r}[']{\begin{smallmatrix}- EA^2 \mu_2  + EAm_2A - \\ - Em_2A^2 + p_2 A^3 \end{smallmatrix}}
\ar[bend left=15]{rr}[description]{- EA \mu_3 - Em_3A - p_3 A^2 }
\ar[bend left=20]{rrr}[description]{-E\mu_4 + p_4A}
&
EA^3
\ar{r}[']{- EA\mu_2 + E m_2 A - p_2 A^2}
\ar[bend left=15]{rr}[description]{E\mu_3 + p_3A}
& 
EA^2
\ar{r}[']{-E\mu_2 + p_2A}
&
\underset{\degzero}{EA}
\end{tikzcd}
\end{equation}
and setting $\infbarres(f_\bullet)$ to be the
corresponding restriction of the twisted complex map 
$\infbar(f_\bullet)$ sign-twisted by $(-1)^{\deg(f)}$. 
We illustrate the odd degree case:
\begin{equation}
\label{eqn-bar-resolution-of-Ainfty-morphism}
\begin{tikzcd}[row sep=1.5cm, column sep = 2.0cm]
\dots
\ar{r}{}
& 
EA^3
\ar{r}
\ar{d}[description, pos=0.75]{f_1 A^3}
\ar{dr}[description, pos = 0.5]{-f_2 A^2}
\ar{drr}[description, near start]{f_3A}
& 
EA^2
\ar{r}
\ar{d}[description, pos=0.75]{- f_1A^2}
\ar{dr}[description, pos = 0.5]{f_2A}
&
EA
\ar{d}[description, pos=0.75]{f_1A}
\\
\dots
\ar{r}
& 
FA^3
\ar{r}
& 
FA^2
\ar{r}
&
\underset{\degzero}{FA}
\end{tikzcd}
\end{equation}
This defines a DG-functor because
$\infbar$ maps $\Ainfty$-$A$-modules to one-sided
twisted complexes and $\Ainfty$-morphisms thereof 
to one-sided maps of twisted complexes. 
\end{defn}

By definition, we have a natural inclusion 
$\freeA \hookrightarrow \nodA$, 
and therefore a natural inclusion  
\begin{equation}
\label{eqn-inclusion-of-pretriagmns-freeA-into-pretriagmns-nodA}
\pretriagmns(\freeA) \hookrightarrow \pretriagmns(\nodA). 
\end{equation}
Using
\eqref{eqn-inclusion-of-pretriagmns-freeA-into-pretriagmns-nodA}, 
we can consider the bar-resolution to be the functor
\begin{equation*}
\infbarres\colon \nodA \rightarrow \pretriagmns(\nodA). 
\end{equation*}
By abuse of notation, write  
\begin{equation*}
\id_{\nodA}\colon \nodA \hookrightarrow
\pretriagmns(\nodA), 
\end{equation*}
for the tautological fully faithful inclusion which sends each element of 
$\nodA$ to itself considered as a twisted complex concentrated 
in degree $0$. 
\begin{defn}
\label{defn-natural-tranformation-rho-Ainfty-version}
Define the homotopy natural transformation 
\begin{equation}
\label{eqn-natural-tranformation-rho-Ainfty-version}
\rho\colon\infbarres \rightarrow \id_{\nodA}
\end{equation}
of DG-functors $\nodA \rightarrow \pretriagmns(\nodA)$
by setting $\rho(E,p_\bullet)$ to consist of the components
of $\infbar(E,p_\bullet)$ which were discarded to obtain 
$\infbarres(E, p_\bullet)$, that is:
\begin{equation}
\label{eqn-natural-transformation-infbarres-into-identity-Ainfty-version}
\begin{tikzcd}[row sep=1.0cm, column sep = 1.0cm]
\dots
\ar{r}{}
& 
EA^3
\ar{r}
\ar{drr}[description]{\pi_4}
& 
EA^2
\ar{r}
\ar{dr}[description]{\pi_3}
&
EA
\ar{d}[description]{\pi_2}
\\
& 
& 
&
\underset{\degzero}{(E,p_\bullet)}. 
\end{tikzcd}
\end{equation}
\end{defn}

By calling $\rho$ a homotopy natural transformation we mean that it
gives a natural transformation of the induced functors on the homotopy 
categories. To see this, note that for any $f_\bullet\colon
(E,p_\bullet) \rightarrow (F, q_\bullet)$ in $\nodA$ 
we can rewrite the twisted complex morphism $\infbar(f_\bullet)$
depicted on 
\eqref{eqn-bar-complex-map-corresponding-to-Ainfty-morphism-Ainfty-version}
as
\begin{equation}
\label{eqn-infbar-f_bullet-rewritten-as-square-with-bar-resolution-Ainfty-version}
\begin{tikzcd}[column sep = 3cm]
\infbarres(E,p_\bullet) 
\ar{r}{\rho}
\ar{d}[description]{(-1)^{\deg(f)} \infbarres(f_\bullet)}
\ar{dr}[description]{f_{\bullet + \bullet}}
&
(E,p_\bullet)
\ar{d}[description]{f_\bullet}
\\
\infbarres(F, q_\bullet)
\ar{r}{\rho}
&
(F,q_\bullet),
\end{tikzcd}
\end{equation}
where $f_{\bullet + \bullet}$ is a twisted complex
morphism whose 
$EA^i \rightarrow (b,q_\bullet)$ component is $f_{\bullet + i}$.  
If $f_\bullet$ is closed of degree $0$, then so is $\infbar(f_\bullet)$. 
We then see from 
\eqref{eqn-infbar-f_bullet-rewritten-as-square-with-bar-resolution-Ainfty-version}
that $\rho$ is natural on $f_\bullet$ up to the homotopy given 
by $f_{\bullet + \bullet}$. Note, that this is also true when 
$df_\bullet$ is strict, and thus $df_{\bullet + \bullet} = 0$. 

\begin{prps}
\label{prps-natural-transformation-rho-is-acyclic-on-H-unital-modules-Ainfty-version}
The homotopy natural transformation
\eqref{eqn-natural-tranformation-rho-Ainfty-version}
is a homotopy equivalence on objects of $\nodA$
whose bar-construction is acyclic. 
\end{prps}

Recall that we say that a twisted complex over $\nodA$ is acyclic 
if its underlying twisted complex over $\A$ is null-homotopic. 

\begin{proof}
For any $(E,p_\bullet) \in \nodA$, 
the total complex of the $\pretriagmns(\nodA)$ map
$\rho(E,p_\bullet)$
is $\infbar(E,p_\bullet)$. 
By the Homotopy Lemma (Lemma \ref{lemma-the-homotopy-lemma-for-nodA})
if $\infbar (E,p_\bullet) \in \pretriagmns(\nodA)$ is acyclic, then 
it is null-homotopic, and hence $\rho(E,p_\bullet)$ is a homotopy equivalence. 
\end{proof}

\section{Strong homotopy unitality}
\label{section-strong-homotopy-unitality}

Throughout this section, let $\A$ be a monoidal DG category
and $(A,m_i)$ be an $\Ainfty$-algebra in $\A$. 

\subsection{Unitality conditions for $A$} 
\label{section-unitality-conditions-for-algebras}

The simplest notions of unitality one can ask for in the case of an 
$\Ainfty$-algebra are strict unitality and (weak) homotopy unitality:

\begin{defn}
\label{defn-strict-and-homotopy-unitality}
We say that $A$ is \em strictly unital \rm if there exists a unit
morphism $\eta\colon \id \rightarrow A$ in $\A$ 
such that:
\begin{enumerate}
\item $m_2 \circ (\eta \otimes \id) = m_2 \circ (\id \otimes \eta) = \id_A$. 
\item $m_i \circ (\id^j \otimes \eta \otimes \id^k) = 0$ for $i \geq
3$ and all appropriate $j$ and $k$. 
\end{enumerate}
 
We say that $A$ is \em weakly homotopy unital \rm if $(A, m_2)$ is a unital  
algebra in $H^0(\A)$. Explicitly, this means
that there exists a unit morphism $\eta$ in $\A$ as above 
such that 
\begin{equation}
\label{eqn-weak-homotopy-unitality-conditions}
m_2 \circ (\eta \otimes \id) = \id_A + d h^r 
\quad \text{ and } \quad 
m_2 \circ (\id \otimes {\eta}) = \id_A + d h^l, 
\end{equation}
for some degree $-1$ endomorphisms $h^l, h^r$ of $A$ in $\A$. 
\end{defn}

The following notion is based on the one introduced  
in \cite[Defn.~4.1.2.5]{Lefevre-SurLesAInftyCategories}:

\begin{defn}
We say that $A$ is \em $H$-unital \rm if its non-augmented bar
construction  $\infbarnaug A$ is null-homotopic in $\pretriagmns \A$. 
\end{defn}

It is clear that strict unitality implies weak 
homotopy unitality with $h^l = h ^r = 0$.
To see that strict unitality implies $H$-unitality, we consider
the standard contracting homotopy of the twisted complex
$\infbarnaug A$ whose components are the 
maps $\eta \otimes \id^i\colon A^i \rightarrow A^{i+1}$. 
However, when $\eta$ is only a weak homotopy unit, it doesn't seem 
to be possible to cook up such natural contracting homotopy 
with it.

In \cite[Cor.~4.1.2.7]{Lefevre-SurLesAInftyCategories} it is shown 
that weak homotopy unitality implies $H$-unitality for the usual
$\Ainfty$-algebras. In \S\ref{section-Ainfty-algebras-in-a-monoidal-category}
we've shown that this translates in our context to the situation 
where $\A = \kmodk$ with the monoidal structure given by the tensor
product. Unfortunately, the proof in loc.~cit. relies on the existence of 
minimal models which uses the fact that in $k$-$\modd$-$k$
every bimodule is homotopic to the sum of its cohomologies. 
This already doesn't hold when $\A = \BmodB$ 
over an arbitrary DG category $\B$. There is no way to generalise that proof
for an arbitrary monoidal DG category $\A$, and the authors believe that
in our generality weak homotopy unitality doesn't imply $H$-unitality. 

To fix this we introduce the notion of strong homotopy unitality:

\begin{defn}
\label{defn-strong-homotopy-unitality-for-ainfty-algebras}
$\Ainfty$-algebra $(A,m_i)$ is \em strongly homotopy unital \rm if there 
exists a unit morphism $\eta\colon \id \rightarrow A$ in $\A$ and
degree $-1$ endomorphisms $h^r_{\bullet}$ and $h^l_{\bullet}$ of $A$, 
respectively, such that 
\begin{equation}
\label{eqn-strong-homotopy-unitality-conditions}
\mu_2 \circ {\eta}A = \id_A + d h^r_{\bullet}  \text{ in } \nodA
\quad \text{ and } \quad 
\mu_2 \circ A{\eta} = \id_A + d h^l_{\bullet} \text{ in } \Anod. 
\end{equation}
Here $\mu_2$ denotes the $\nodA$
and $\Anod$ morphisms $A^2 \rightarrow A$ of Defn.~\ref{defn-morphisms-pi_i-and-mu_i}. 
\end{defn}

The existence of either $h^l_\bullet$ or $h^r_\bullet$
is enough for us to be able to construct 
the contracting homotopy of $\infbarnaug A$ in the same way as 
in the strictly unital case:
\begin{lemma}
\label{lemma-strong-homotopy-unital-implies-H-unital}
If $A$ is strongly homotopy unital, then it is $H$-unital.  
\end{lemma}
\begin{proof}
As a twisted complex, non-augmented algebra bar-construction $\infbarnaug A$ coincides with right module bar-construction $\infbarr A$. It suffices therefore to show that $\infbarr A$ is null-homotopic in $\pretriagmns \A$. 

Consider the following endomorphism of $\infbarr A$ in $\pretriagmns \A$:
\begin{small}
\begin{equation*}
\beta : = 
\begin{tikzcd}[column sep = 3.25cm, row sep = 1cm]
\cdots\quad
A^4
\ar{r}[description]{A^2m_2 - A m_2 A + m_2 A^2}
\ar[bend left=10]{rr}[description]{-Am_3 - m_3A}
\ar[bend left=11]{rrr}[description]{m_4}
& 
A^3
\ar{r}[description]{Am_2 - m_2A}
\ar[bend left=10]{rr}[description]{m_3}
\ar{dl}[description]{\eta A^3}
&
A^2
\ar{r}[description]{m_2}
\ar{dl}[description]{- \eta A^2}
&
\underset{\degzero}{A}
\ar{dl}[description]{\eta A}
\\
\cdots\quad
A^4
\ar{r}[description]{A^2m_2 - A m_2 A + m_2 A^2}
\ar[bend right=10]{rr}[description]{-Am_3 - m_3A}
\ar[bend right=11]{rrr}[description]{m_4}
& 
A^3
\ar{r}[description]{Am_2 - m_2A}
\ar[bend right=10]{rr}[description]{m_3}
&
A^2
\ar{r}[description]{m_2}
&
\underset{\degzero}{A}. 
\end{tikzcd}
\end{equation*}
\end{small}
It can be readily checked that
\begin{equation*}
d\beta : = 
\begin{tikzcd}[column sep = 3.25cm, row sep  = 2.2cm]
\cdots\quad
A^4
\ar{r}
\ar[shift left = 0.25cm]{d}[description, pos = 0.82]{m_2A^3 \circ \eta A^4}
\ar{dr}[description, pos = 0.82]{- m_3A^2 \circ \eta A^4}
\ar{drr}[description, pos = 0.82]{m_4A \circ \eta A^4}
\ar{drrr}[description, pos = 0.8]{- m_5 \circ \eta A^4}
&
A^3
\ar{r}
\ar{d}[description, pos = 0.68]{m_2A^2 \circ \eta A^3}
\ar{dr}[description, pos = 0.7]{- m_3A \circ \eta A^3}
\ar{drr}[description, pos = 0.67]{m_4 \circ \eta A^3}
&
A^2
\ar{r}
\ar{d}[description, pos = 0.37]{m_2 A \circ \eta A^2}
\ar{dr}[description, pos=0.52]{- m_3 \circ \eta A^2}
&
A 
\ar{d}[description, pos = 0.4]{m_2 \circ \eta A}
\\
\cdots\quad
A^4
\ar{r}
&
A^3
\ar{r}
&
A^2
\ar{r}
&
A,
\end{tikzcd}
\end{equation*}
thus it is the bar-construction of 
the morphism $\mu_2 \circ \eta A$ in $\nodA$. 

Since $A$ is strongly homotopy unital, we have in $\nodA$
$$ \mu_2 \circ \eta A = \id_A + dh^r_{\bullet}. $$
Hence in $\pretriagmns \A$ we have
$$ d\beta = \id_{\infbarr A} + d\left(\infbar
h^r_{\bullet}\right), $$
whence $\infbarr A$ is contractible, as desired. 
\end{proof}

We see in the next section that the condition of strong homotopy
unitality is enough to ensure that for $\Ainfty$-$A$-modules
$H$-unitality is equivalent to ordinary homotopy unitality. However, to have 
a good notion of strong homotopy unitality for $A$-modules, we need
the following stronger notion for $A$ itself:

\begin{defn}
\label{def-bimodule-homotopy-unitality}
Let $(A,m_i)$ be an $\Ainfty$-algebra in a monoidal DG category $\A$. 
We say that $A$ is \em bimodule homotopy unital \rm if there is 
a degree $0$ morphism 
$$ \bareta_{\bullet\bullet}\colon \id_\A \rightarrow A, $$
in $\AnodA$ such that
\begin{equation}
\label{eqn-bimodule-homotopy-unitality-condition}
(d\bareta)_{\bullet\bullet}
= \left\{
 (d\bareta)_{10} = \id_A,  (d\bareta)_{01} = \id_A
\right\}. 
\end{equation}
Here the $\Ainfty$-$A$-$A$-bimodule structures on $A$ and $\id_\A$ are
given by the natural action of $A$ on itself with $p_{ij} = m_{i+j+1}$, 
and the zero action of $A$ on $\id_\A$ with $p_{ij} = 0$. 
\end{defn}

This definition is best understood by the means of the following result:

\begin{theorem}
\label{theorem-conditions-for-bimodule-homotopy-unitality}
$A$ is bimodule homotopy unital if and only if there exist
\begin{itemize}
\item a degree $0$ morphism $\eta\colon \id_{\A} \rightarrow A$ in
$\A$,  
\item a degree $-1$ morphism $h^l_\bullet\colon A \rightarrow A$
in $\Anod$,
\item a degree $-1$ morphism $h^r_\bullet\colon A \rightarrow A$ 
in $\nodA$,
\item a degree $-2$ morphism $\kappa_{\bullet\bullet} \colon A^2
\rightarrow A$ in $\AnodA$, 
\end{itemize}
such that in these categories 
\begin{align*}
d\eta &= 0, \\
dh^r_\bullet &= \id_{\A} - \mu_2 \circ \eta{A}, \\
dh^l_\bullet &= \id_{\A} - \mu_2 \circ A\eta, \\
d\kappa_{\bullet\bullet} &= \mu_2 \circ (h^l_\bullet{A} -
{A}h^r_\bullet) - \mu_3 \circ A \eta A.
\end{align*}
\end{theorem}

Here in each category, $\eta{A}$, $A{\eta}$,
$h^l_\bullet{A}$, ${A}h^r_\bullet$ and $A\eta{A}$ denote 
$\Ainfty$-morphisms whose only non-zero components they specify. 
E.g.~$A\eta{A}$ is a degree $0$ morphism $A^2 \rightarrow A^3$ in $\A$,
however in $\AnodA$ it denotes the degree $2$ 
morphism $\id_A \rightarrow A^3$ whose only non-zero component is
the $(1,1)$-th component $A\eta{A}$. Moreover, in each
respective category, $\mu_k\colon A^k \rightarrow A$ is the
$\Ainfty$-morphism defined in 
\S\ref{section-bar-construction-as-a-complex-of-ainfty-A-modules}. 
The bimodule version is defined by $(\mu_k)_{ij} = (-1)^{(k-1)j} m_{i+j+1}$, 
and left/right module versions are obtained from it by setting $i$ or 
$j$ to be zero. 

\begin{proof}
A degree $0$ morphism 
$\bareta_{\bullet\bullet} \colon \id \rightarrow A$ 
in the category of $\Ainfty$ $A$-$A$-bimodules is 
a collection of degree $2-i-j$ natural transformations
$$ \bareta_{ij}\colon A^{i-1} A^{j-1} \rightarrow A \quad \quad i,j \geq 1. $$

Let $\infbarbi(\bareta_{\bullet\bullet})$ be the induced map of
twisted bicomplexes. Its source
$\infbarbi(\id_\A)$ breaks up as a direct sum of four
twisted bicomplexes: 
\begin{itemize}
\item $\id_\A$ concentrated in degree $(0,0)$, 
\item $\infbarl(A)[1]$ concentrated in column $0$,
\item $\widehat{\infbarr(A)}[1]$ concentrated in row $0$, 
\item $\infbarbi(A^2)[1,1]$. 
\end{itemize}
Here $\widehat{(-)}$ is the automorphism of the DG category of twisted
complexes which sign-twists every differential $\alpha_{ij}$ and every 
morphism component $f_{ij}$ by $(-1)^{i+j}$. It is isomorphic to the 
identity functor via the isomorphism $(E, \alpha_{ij}) \simeq 
\widehat{(E, \alpha_{ij})}$ whose components are the maps 
$(-1)^{i} \id\colon E_i \rightarrow E_i$. 

Since $\widehat{\infbarr(A)} \simeq \infbarr(A)$,  
$\infbarbi(\id_\A)$ is isomorphic 
to the total complex of the following twisted bicomplex 
of twisted bicomplexes: 
\begin{equation}
\label{eqn-bar-bicomplex-of-Id-subcomplexes-ainfty}
\begin{tikzcd}
\infbarbi(A^2) 
\ar{r}{0}
\ar{d}{0}
&
\infbarl(A) 
\ar{d}{0}
\\
\infbarr(A)
\ar{r}{0}
& 
\id_\A
\end{tikzcd}
\end{equation}
In the rest of the proof we use this to implicitly identify 
$\infbarbi(\bareta_{\bullet\bullet})$ with
\eqref{eqn-bar-bicomplex-of-Id-subcomplexes-ainfty}. 

For any $\Ainfty$-bimodule morphism 
$f_{\bullet\bullet}\colon \id \rightarrow A$
write $f^0$, $f^r$, $f^l$, and $f^{bi}$ for the restrictions of the
twisted bicomplex map $\infbarbi(f_{\bullet\bullet})$
to the four elements of
$\eqref{eqn-bar-bicomplex-of-Id-subcomplexes-ainfty}$. 
These factor through the subcomplexes of $\infbarbi(A)$ given by 
the degree $(0,0)$ element, the $0$th row, the $0$th column, and 
the whole of $\infbarbi(A)$. We identify the first three with 
$A$, $\infbarr(A)$, and $\infbarl(A)$, respectively. 
The restriction of $\infbarbi(f_{\bullet\bullet})$
to each of $\id_\A$, $\infbarr(A)$, $\infbarl(A)$, and $\infbarbi(A^2)$
is given by the respective bar-construction of 
the morphisms of $f_{\bullet\bullet}$ which originate in 
that twisted complex plus the corrections which come from the 
bimodule bar-constructions of those morphisms of 
of $\bareta_{\bullet\bullet}$ which originate in the higher 
degrees of the bicomplex.  

Viewing $\bareta_{\bullet\bullet}$ as a collection
of maps from the elements of $\infbarbi(\id_\A)$ to $A$, 
use the identification of  $\infbarbi(\id_\A)$ and 
$\eqref{eqn-bar-bicomplex-of-Id-subcomplexes-ainfty}$
to write $\eta$, $h^r_\bullet$, $h^l_\bullet$, and
$\kappa_{\bullet\bullet}$ for the restrictions
of $\bareta_{\bullet\bullet}$ to 
$\id_\A$, $\infbarr(A)$, $\infbarl(A)$, and $\infbarbi(A^2)$. 
In other words, $\eta = \bareta_{00}$, $h^r_i = (-1)^{i-1}
\bareta_{0i}$, $h^l_i = \bareta_{i0}$, and $\kappa_{ij} = 
\bareta_{(i+1)(j+1)}$. 

Conversely we use the identification of 
$\eqref{eqn-bar-bicomplex-of-Id-subcomplexes-ainfty}$
with $\infbarbi(\id_\A)$, to view any one of
$\eta$, $h^r_\bullet$, $h^l_\bullet$, or $\kappa_{\bullet\bullet}$
as a collection of maps from the elements of
$\infbarbi(\id_\A)$ to $A$, 
and thus as an $\Ainfty$-bimodule morphism $\id_\A \rightarrow A$.
In this sense
\begin{equation}
\label{eqn-decomposition-of-bareta}
\bareta_{\bullet\bullet} = \eta + h^r_\bullet + h^l_\bullet + 
\kappa_{\bullet\bullet}. 
\end{equation}
We then write e.g.~$\eta^r$ or $(h^r_\bullet)^{bi}$ for 
the corresponding component of its bar-construction. 
For example, $(h^r_\bullet)^0 = (h^r_\bullet)^l = 0$, 
while $(h^r_\bullet)^r = \infbarr(h^r_\bullet)$, and similarly for
$h^l_\bullet$. Also  
\begin{equation}
\label{eqn-bareta-infbarl-A-second-component-r}
\eta^r = 
\begin{tikzcd}
\dots 
\ar{r}
&
A^3  
\ar{r}
\ar{dl}[description]{\eta A^3}
&
A^2  
\ar{r}
\ar{dl}[description]{-\eta A^2}
&
A  
\ar{dl}[description]{\eta A}
\\
\dots 
\ar{r}
&
A^3  
\ar{r}
&
A^2  
\ar{r}
&
A, 
\end{tikzcd}
\end{equation}
\begin{equation}
\label{eqn-bareta-infbarl-A-second-component-l}
\eta^l = 
\begin{tikzcd}
\dots 
\ar{r}
&
A^3  
\ar{r}
\ar{dl}[description]{A^3 \eta}
&
A^2  
\ar{r}
\ar{dl}[description]{A^2 \eta}
&
A  
\ar{dl}[description]{A \eta}
\\
\dots 
\ar{r}
&
A^3  
\ar{r}
&
A^2  
\ar{r}
&
A. 
\end{tikzcd}
\end{equation}

With this notation in mind, we have:
\begin{itemize}
\item $(\bareta_{\bullet\bullet})^0$ 
is the degree $0$ morphism $\eta\colon \id_\A \rightarrow A$ itself. 
\item $(\bareta_{\bullet\bullet})^{r}$
is the degree $-1$ morphism $ \infbarr(h^r_\bullet) + \eta^r$. 
\item $(\bareta_{\bullet\bullet})^l$
is the degree $-1$ morphism $ \infbarl(h^l_\bullet) + \eta^l$. 
\item $(\bareta_{\bullet\bullet})^{bi}$
is the degree $-2$ morphism $\infbarbi(\kappa_{\bullet \bullet}) 
+ (h^l_\bullet)^{bi} + (h^{r}_{\bullet})^{bi} +
\eta^{bi}$. 
\end{itemize}

The condition that $(d\bareta)_{\bullet\bullet}$ has only two 
non-zero components
$$ (d\bareta)_{10} = \id_A \quad\text{ and }\quad (d\bareta)_{01} =
\id_A, $$
is therefore equivalent to:
\begin{itemize}
\item $d\left((\bareta_{\bullet\bullet})^{0}\right) = 0$. This
simply means that $d\eta = 0$ in $\A$. 

\item $d((\bareta_{\bullet\bullet})^l) = \infbarl(\id^l_A)$,where
$\id^l_A$ is the identity map of $A$ in $\Anod$. By above we
have
$$ d((\bareta_{\bullet\bullet})^l) =  \infbarl(dh^l_\bullet) +
d(\eta^l) $$
and while $\eta^l$ is not a left-module 
bar-construction of something, we have $d(\eta^l) = \infbarl(\mu_2 \circ
A\eta)$. We conclude that this condition is equivalent to  
\begin{equation}
\label{eqn-proof-bimodule-homotopy-unitality-l-condition}
dh^l_\bullet = \id^l_A - \mu_2 \circ A\eta \quad \quad \text{ in }
\quad \Anod. 
\end{equation}

\item $d((\bareta_{\bullet\bullet})^r) = \infbarr(\id^r_A)$, 
where $\id^r_A$ is the identity map of $A$ in $\nodA$. 
Similarly, this condition is equivalent to 
\begin{equation}
\label{eqn-proof-bimodule-homotopy-unitality-r-condition}
dh^r_\bullet = \id^r_A - \mu_2 \circ \eta{A}
\quad \quad \text{ in } \quad \nodA. 
\end{equation}

\item $d((\bareta_{\bullet\bullet})^{bi}) = (\id^{l}_T)^{bi} + 
(\id^{r}_T)^{bi}$. 
From \eqref{eqn-decomposition-of-bareta} we have
$$ d((\bareta_{\bullet\bullet})^{bi}) = \infbarbi(d\kappa_{\bullet
\bullet})
+ d(h^{l}_\bullet)^{bi} + d(h^{r}_{\bullet})^{bi} + d\eta^{bi}. $$
We have $d\eta^{bi} =  (d_{bimod}\eta)^{bi}$. Now
note that $(d_{{bimod}}\eta)_{00} = d\eta = 0$ and 
$$ (d_{\text{bimod}}\eta)_{ij} = m_{i+j+1} \circ A^i \eta A^j 
\quad \quad i + j \geq 1. $$
Split $d_{{bimod}}\eta$ up into the summands
originating in $\infbarr(A)$, $\infbarl(A)$, and $\infbarbi(A^2)$.
We can write these summands as $\mu_2 \circ {\eta}A - 
\mu_3 \circ A{\eta}A$, $\mu_2 \circ A{\eta} - \mu_3 \circ A{\eta}A$, 
and $\mu_3 \circ A{\eta}A$, respectively. Indeed, we have componentwise 
\begin{align*}
(\mu_2 \circ {\eta}A)_{ij} & = m_{i+j+1} \circ A^i (\eta A) A^{j-1},
\quad \quad i \geq 0, j \geq 1,
 \\
(\mu_2 \circ {A\eta})_{ij} & = m_{i+j+1} \circ A^{i-1} (A\eta) A^j, 
\quad \quad i \geq 1, j \geq 0,
\\
(\mu_3 \circ A{\eta}A)_{ij} & = m_{i+j+1} \circ A^{i-1}(A{\eta}A) A^{j-1},
\quad \quad i \geq 1, j \geq 1,
\end{align*}
with the remaining components being zero. Thus 
$$ d_{{bimod}}\eta = \mu_2 \circ {\eta}A + \mu_2 \circ A{\eta}   
- \mu_3 \circ A{\eta}A. $$

Next, split $d_{\text{bimod}} h^l_\bullet$ into the summands
originating in $\infbarl(A)$ and  $\infbarbi(A^2)$. The former 
has the same components as $\Anod$ morphism $d_{{lmod}} h^l_\bullet$ 
which equals $\id_A^{l} - \mu_2 \circ A{\eta}$ by
\eqref{eqn-proof-bimodule-homotopy-unitality-l-condition}. We can 
therefore write this summand in $\AnodA$ as  
$\id_A^{l} - \mu_2 \circ A{\eta} + \mu_3 \circ A{\eta}A$. 
As before, $\mu_3 \circ A{\eta}A$ is needed to cancel out 
the components of $\mu_2 \circ A{\eta}$ which originate outside $\infbarl(A)$.  
On the other hand, the summand of $d_{\text{bimod}} h^l_\bullet$ 
originating in $\infbarbi(A^2)$ can be rewritten as 
$\mu_2 \circ h^l_\bullet{A}$. Thus 
\begin{align*}
d_{{bimod}} h^l_\bullet
 = \id_A^{l} - \mu_2 \circ A{\eta} + \mu_3 \circ A{\eta}A + \mu_2 \circ h^l_\bullet{A}, 
\end{align*}
and computing $d_{{bimod}} h^r_\bullet$ in a similar way we obtain
\begin{align*}
d_{{bimod}} h^r_\bullet
&=  \id_A^{r} - \mu_2 \circ {\eta}A + \mu_3 \circ A{\eta}A - \mu_2 \circ {A}h^r_\bullet. 
\end{align*}

We conclude that the condition that 
$d((\bareta_{\bullet\bullet})^{bi}) = (\id^{l}_T)^{bi} + (\id^{r}_T)^{bi}$
is equivalent to 
$$ \infbarbi(d\kappa_{\bullet \bullet}) = \left(\mu_2 \circ
(h^l_\bullet{A} - {A}h^r_\bullet) + \mu_3 \circ A \eta A\right)^{bi}. $$
As $\mu_2 \circ (h^l_\bullet{A} - {A}h^r_\bullet) + \mu_3 \circ A \eta A$ 
originates in $\infbarbi(A^2)$, we have 
$$ \infbarbi \left( \mu_2 \circ (h^l_\bullet{A} - {A}h^r_\bullet) + \mu_3
\circ A \eta A\right) = \left( \mu_2 \circ (h^l_\bullet{A} - {A}h^r_\bullet)
+ \mu_3 \circ A \eta A \right)^{bi}, $$
and therefore the above is equivalent to 
$$ d\kappa_{\bullet\bullet} = \mu_2 \circ (h^l_\bullet{A} -
{A}h^r_\bullet) + \mu_3 \circ A \eta A \quad \quad \text{ in } \AnodA. $$
\end{itemize}
\end{proof}

\begin{cor}
Bimodule homotopy unitality implies strong homotopy unitality. 
\end{cor}

\subsection{Unitality conditions for $A$-modules} 
\label{section-unitality-conditions-for-A-modules}

The notions defined in the previous section have analogues 
for $\Ainfty$-$A$-modules:
\begin{defn}
A right $\Ainfty$-$A$-module $(E,p_\bullet)$ is \em $H$-unital \rm 
if its bar-construction $\infbar(E,p_\bullet)$ is acyclic. 
We write $\nodhuA$ for the full subcategory of $\nodA$ comprising
$H$-unital modules. 

If $A$ is strongly homotopy unital, then: 
\begin{itemize}
\item  $(E,p_\bullet)$ is \em strictly
unital \rm if in $\A$ the composition
$$ EA^i \xrightarrow{EA^j\eta A^{i-j}} EA^{i+1} \xrightarrow{p_{i+2}} E $$
equals $\id_E$ for $i = 0$ and $0$ otherwise. We write
$\moddinf\text{-}A$ for the full subcategory of $\nodA$ consisting of
strictly unital modules. 

\item $(E,p_\bullet)$ is \em homotopy unital \rm if $(E,p_2)$ is
a strictly unital $H^0(A)$-module in $H^0(\A)$. Explicitly, 
we have a degree $-1$ endomorphism $h$ of $E$ in $\A$ with 
\begin{equation}
\label{item-homotopy-unitality-condition-for-A-modules}
E \xrightarrow{E\eta} EA \xrightarrow{p_2} E = \id_E + dh. 
\end{equation}
\end{itemize}
These notions are defined in the same way for left $\Ainfty$-$A$-modules.
\end{defn}

The notion of strong homotopy unitality for $\Ainfty$-$A$-modules is
more difficult. It turns out that for a good notion we need 
to replace $EA$ in 
\eqref{item-homotopy-unitality-condition-for-A-modules} by
the whole bar-resolution $\infbarres(E,p_\bullet)$. The map $p_2\colon
EA \rightarrow A$ in $\A$ lifts naturally to the map 
$\rho\colon \infbarres(E,p_\bullet) \rightarrow (E,p_\bullet)$ in 
$\pretriagmns (\noddinf(\A))$. 
However, the closed degree zero map $E\eta\colon E \rightarrow EA$ in
$\A$ isn't even closed as a map in $\noddinf(A)$. 
We therefore need to construct its lift to a closed map in 
$\pretriagmns (\noddinf(A))$:
\begin{equation}
\chi = 
\begin{tikzcd}[row sep=1.3cm]
& & & 
(E,p_\bullet)
\ar{d}[description]{\chi_{{1}\bullet}}
\ar{dl}[description]{\chi_{{2}\bullet}}
\ar{dll}[description]{\chi_{{3}\bullet}}
\ar[phantom]{dlll}[description]{\cdots}
\\
\dots 
\ar{r}
\ar[bend left=20]{rr}
\ar[bend left=20]{rrr}
&
EA^3 
\ar{r}
\ar[bend left=20]{rr}
&
EA^2
\ar{r}
&
EA. 
\end{tikzcd}
\end{equation}

Strong homotopy unitality of $A$ is not sufficient for this, 
but bimodule homotopy unitality is.
View $A$ and $\id_\A$ as $\Ainfty$ $A$-$A$-bimodules with the
natural action on the former and the zero action on the latter. 
The bimodule homotopy unitality structure on $A$ is an 
$\AnodA$-morphism $\bareta_{\bullet\bullet}\colon \id_\A \rightarrow A$
with $d\bareta_{\bullet\bullet} = \id^l_A + \id^r_A$. 
Here $\id^l_A$ and $\id^r_A$ are the $\Ainfty$-bimodule 
maps $\id_\A \rightarrow A$ whose only component is $\id_A$ 
in degree $(1,0)$ and $(0,1)$, respectively. 
See \S\ref{section-unitality-conditions-for-algebras} for more
details. 

With this, we construct $\chi$ as follows.  
Write $\delta_i\colon \id_A \rightarrow
A^i$ for the morphism of right $\Ainfty$-$A$-modules whose sole component 
is $\id_{A^i}$. Write 
$$ \Delta\colon \id_\A \rightarrow \infbarl(\id_\A) $$
for the closed degree $0$ 
map $\sum (-1)^i\delta_i$ of twisted complexes of right
$\Ainfty$-$A$-modules. 

We now make implicit use of the category isomorphism 
$\AnodA \simeq A\text{-}(\nodA)$ of Theorem  
\ref{theorem-ainfty-bimodules-are-ainfty-modules-in-cat-of-ainfty-modules}
to view $\id_\A$ and $\A$ as left $\Ainfty$-$A$-modules in 
the category of right $\Ainfty$-$A$-modules and 
$\bareta_{\bullet\bullet}$ as a morphism thereof. 
We thus define
\begin{equation}
\label{eqn-bicomplex-version-of-chi-underlying}
\chi^0 : = \id_\A \xrightarrow{\Delta}
\infbarl(\id_\A)
\xrightarrow{\infbarl(\bareta_{\bullet\bullet})}
\infbarl(A) \quad \quad \text{ in } \pretriagmns\left(\nodA\right).
\end{equation}

\begin{lemma}
$\chi^0$ is a degree $0$ closed morphism. 
\end{lemma}
\begin{proof}
It is of degree $0$ because both $\Delta$ and $\bareta_{\bullet\bullet}$ 
were defined to be of degree $0$.  
Since $\Delta$ is a closed map, to show
$\chi^0$ to be closed it remains to show that 
$$ \id_\A \xrightarrow{\Delta}
\infbarl(\id_\A)
\xrightarrow{\infbarl(d \bareta_{\bullet\bullet})} \infbarl(A)
 \; = 0 \quad \quad \quad \text{ in } \pretriagmns\left(\nodA\right). 
$$

We have  
$d\bareta_{\bullet\bullet} = \id^l_A + \id^r_A$ in $\AnodA$. 
% In $A\text{-}(\nodA)$ this becomes a morphism 
% $f_\bullet\colon \id_A \rightarrow A$ which has two components: $f_1 =
% \delta_1$ and $f_2 = \id_A$. Its left module bar construction 
% $\infbarl(f_\bullet)$ is the morphism of twisted complexes whose 
% components are maps $(-1)^i A^i \delta_1$ and $\id_{A^i}$ from each $A^i$ in
% $\infbarl(\id_\A)$. Its composition with $\Delta = \sum (-1)^i
% \delta_i$ is zero, because  
% $$ A^i \delta_1 \circ \delta_i = \delta_{i+1}. $$
The isomorphism 
of Theorem  
\ref{theorem-ainfty-bimodules-are-ainfty-modules-in-cat-of-ainfty-modules}
intertwines $\infbarbi$ with sign-twisted 
$\cxrow \circ \infbarr \circ \infbarl$, so it remains to show
that $\infbarbi(\id^l_A + \id^r_A)$ composes to zero with
$\cxrow \circ \infbarr\left((-1)^{ij + kl}\Delta\right)$. 
The latter is the diagonal map which takes $\infbarr(\id_\A)$, 
viewed as a bicomplex concentrated in the row $0$, and maps its
elements to the diagonals of $\infbarbi(\id_\A)$ using the sign-twisted
identity maps $(-1)^{i}\id\colon A^{i+j} \rightarrow A^iA^j$. 
It composes to zero with the bar-construction
of any $\Ainfty$-bimodule morphism 
$\id_A \rightarrow A$ whose components
in every diagonal sum to zero with alternating signs. 
In particular, with that of $\id^l_A + \id^r_A$. 
\end{proof}

Let $(E,p_\bullet) \in \nodA$. Let $(E,0) \in \nodA$ be
$E$ with the zero action of $A$. 
We have canonically
$(E,0) \simeq E \otimes \id_\A$
and
$\infbarres (E,0) \simeq E \otimes \infbarres(\id_\A)$.
Twisted complexes $\infbarres(\id_\A)$ and 
$\infbarl(A)$ differ only in signs. We identify
them with the isomorphism $\phi = \sum_i (-1)^{i+1}\id_{A^i}$
and thus have a closed degree $0$ map 
\begin{equation}
\label{eqn-bicomplex-version-of-chi-zero} 
E(\phi\circ \chi^0) := 
(E,0) \xrightarrow{E\chi^0}
E\infbarl(A)
\xrightarrow{E\phi}
\infbarres(E,0) \quad \text{ in } \pretriagmns\left(\nodA\right). 
\end{equation}
We also have non-closed degree $0$ maps 
$\id\colon (E,p_\bullet) \rightarrow (E,0)$ and 
$\id\colon (E,0) \rightarrow (E,p_\bullet)$
whose only non-zero component is $\id_E$. 

\begin{defn}
\label{defn-a-infinity-version-of-map-chi}
For any $(E,p_\bullet) \in \nodA$ 
define a degree $0$ map  
$$ \chi \colon (E,p_\bullet) \xrightarrow{\id}
(E,0) \xrightarrow{E(\phi\circ \chi^0)}
\infbarres(E,0) \xrightarrow{\infbarres(\id)}
\infbarres(E,p_\bullet) 
\quad \quad
\text{ in } \pretriagmns\left(\nodA\right).
$$ 
\end{defn}

Explicitly, the components of $\chi$ are 
$\chi_{i\bullet}\colon (E,p_\bullet) \rightarrow EA^i$
and we have
$\chi_{ij}  = 0$ if $i>j$ and
$\chi_{ij} = (-1)^{(i-1)j}EA^{i-1}\sum_{k=0}^{j-i} (-1)^k
\eta_{k,j-i-k}$ otherwise. 

\begin{prps}
The maps $\chi$ in
Defn.~\ref{defn-a-infinity-version-of-map-chi}
define a closed, degree zero natural transformation
$\chi\colon \id \rightarrow \infbarres$
of functors $\nodA \rightarrow \pretriagmns(\nodA)$. 
\end{prps}
\begin{proof}
Let $f_{\bullet}\colon (E,p_\bullet) \rightarrow (F, q_\bullet)$ be
a morphism in $\nodA$ and  
$f_\bullet\colon (E,0) \rightarrow (F,0)$ be the morphism with the same components. Consider the following diagram:
\begin{equation}
\begin{tikzcd}[column sep = 1.9cm]
(E,p_\bullet) 
\ar{r}{\id}
\ar{d}[description]{f_\bullet}
&
(E,0) 
\ar{r}{E(\phi \circ \chi^0)}
\ar{d}[description]{f_\bullet}
&
\infbarres(E,0) 
\ar{r}{\infbarres(\id)}
\ar{d}[description]{\infbarres(f_\bullet)}
&
\infbarres(E,p_\bullet) 
\ar{d}[description]{\infbarres(f_\bullet)}
\\
(F,q_\bullet) 
\ar{r}{\id}
&
(F,0) 
\ar{r}{F(\phi \circ \chi^0)}
&
\infbarres(F,0) 
\ar{r}{\infbarres(\id)}
&
\infbarres(F,q_\bullet).
\end{tikzcd}
\end{equation}
The left and the right squares commute trivially. For the central
square, we compute the upper-right and the lower-left
compositions around its perimeter. These are twisted complex maps 
$(E,0) \rightarrow \infbarres(F,0)$ and thus comprise 
$\Ainfty$-morphisms $(E,0) \rightarrow FA^i$ for $i \geq 1$. 
In the lower-left composition, the non-zero components 
of $(E,0) \rightarrow FA^i$ are the morphisms  
$EA^{i+j-1} \rightarrow  EA^{i}$ with $j \geq 0 $
in $\A$ given by
$$ \sum_{k = 0}^{j} 
(-1)^{(i-1)(i+j-k)}FA^{i-1}\left(\sum_{l = 0}^{j-k} (-1)^l \bareta_{l,j-k-l}\right) 
\circ  
(-1)^{\deg(f)(i+j-k+1)} f_{k+1} A^{i-1} A^{j-k}. $$
In the upper-right composition, the same components 
of $(E,0) \rightarrow FA^i$ are non-zero, but now they are given by
$$ \sum_{k = 0}^{j} 
(-1)^{\deg(f)(i+1)} f_{k+1} A^{i-1} A
\circ
(-1)^{(i+k-1)(i+j)}E A^k A^{i-1}\left(\sum_{l = 0}^{j-k} (-1)^l
\bareta_{l,j-k-l}\right).
$$
These two expressions are equal by DG bifunctoriality of the monoidal
operation. Thus the perimeter commutes, and $\chi$ is a natural transformation. 

Next, note that $d(\id)\colon (E,p_\bullet) \rightarrow (E,0)$
and $d(\id)\colon (E,0) \rightarrow (E,p_\bullet)$
are the maps $- \pi_1 \circ \id $ and $\id \circ \pi_1$ where
$\pi_1\colon (E,p_\bullet) \rightarrow (E,p_\bullet)$ is as 
defined in
\S\ref{section-bar-construction-as-a-complex-of-ainfty-A-modules}.
Thus
\begin{align*}
d(\chi_{(E,p_\bullet)}) = 
& \infbarres(\pi_1) \circ \infbarres(\id) \circ 
E(\phi \circ \chi^0) \circ \id
\; - \;  0 \; - 
\\
& - \;
\infbarres(\id) \circ
E(\phi \circ \chi^0) \circ
\id \circ \pi_1 = \infbarres(\pi_1) \circ \chi_{(E,p_\bullet)} -
\chi_{(E,p_\bullet)} \circ \pi_1. 
\end{align*}
This is zero since $\chi$ is a natural transformation. 
\end{proof}

We now define a good notion of strong homotopy unitality for
$\Ainfty$-$A$-modules:
\begin{defn}
\label{defn-strong-homotopy-unitality}
Let $A$ be bimodule homotopy unital. 
We say that a module $(E,p_\bullet) \in \nodA$ is \em strongly homotopy 
unital \rm if there exists a degree $-1$ endomorphism $h_\bullet$ 
of $(E,p_\bullet)$ in $\nodA$  such that 
$$ (E,p_\bullet) \xrightarrow{\chi} \infbarres(E,p_\bullet)
\xrightarrow{\rho} (E,p_\bullet) \quad = \quad \id + dh_\bullet. $$
\end{defn}

\begin{lemma}
\label{lemma-strong-homotopy-unitality-stable-under-homotopy-equivalences}
Strong homotopy unitality is stable under homotopy equivalences.  
\end{lemma}
\begin{proof}
Our definition of strong homotopy unitality asks that in 
the homotopy category $\rho \circ \chi(E,p_\bullet) =
\id_{(E,p_\bullet)}$. Since $\chi$ is natural transformation and $\rho$ is a
homotopy natural transformation, their composition $\rho \circ \chi$ is also 
a homotopy natural transformation. In any category, 
the set of objects on which any
given natural transformation of the identity functor evaluates to $\id$ 
is closed under isomorphisms. Applying this to $H^0(\nodA)$, we obtain 
the desired assertion. 
\end{proof}

The following is a surprisingly non-trivial result which makes use of 
all of the higher homotopies provided by the bimodule homotopy
unitality of $A$:

\begin{lemma}
\label{lemma-free-modules-are-strong-homotopy-unital}
If $A$ is bimodule homotopy unital, then free $A$-modules are strongly
homotopy unital. 
\end{lemma} \begin{proof}
Let $E \in \A$. The composition $\rho\circ\chi$ for $EA$ is simply
the same composition for $A$ multiplied by $E$ on the left,
so it is enough to prove that $A$ as a right module over itself
is strongly homotopy unital.

Consider the bar-construction of $\chi^0$. It is a closed map
of degree $0$ between twisted complexes
$\infbarr(\id_\A)\to\infbarr(\infbarl(A))$.
Note that we can present $\infbarr(\id_\A)$ as the trivial twisted
complex $\widehat{\infbarr(A)} \xrightarrow{0} \id_\A$. 
Rewriting $\infbarr(\infbarl(\id_A))$  and $\infbarr(\infbarl(A))$
similarly, then the bar-construction of $\chi^0$ becomes
\begin{equation}
\label{eqn-decomposition-of-delta-bareta}
\begin{tikzcd}
\widehat{\infbarr(A)} 
\ar{r}{0}
\ar{d}[description]{\Delta_{1}}
\ar{dr}[description]{\Delta_{10}}
&
\id_{\A}
\ar{d}[description]{\Delta_0}
\\
\infbarbi(A^2)[1] \oplus \widehat{\infbarr(A)} 
\ar{r}{0}
\ar{d}[description]{\eta_{1}}
\ar{dr}[description]{\eta_{10}}
& 
\infbarl(A)[1] \oplus \id_\A 
\ar{d}[description]{\eta_0}
\\
\infbarr(\infbarres(A))
\ar{r}{\infbarr(\rho)}
&
\underset{\degzero}{\infbarr(A)}. 
\end{tikzcd}
\end{equation}
Here $\Delta_i$ and $\eta_i$ are the components of
$\infbarr(\Delta)$ and $\infbarr(\infbarl(\bareta_{\bullet\bullet}))$,
respectively. 

Now observe that the composition 
$
\eta_1\circ\Delta_1\circ \phi\colon \infbarr(A) \to \infbarr(\infbarres(A))
$
has the same components as $\infbarr(\chi)$. On the other hand, we have
$$
d(\eta_0\circ\Delta_{10}+\eta_{10}\circ\Delta_1)+\infbarr(\rho)\circ\eta_1\circ\Delta_1 = 0.
$$
Since $\eta_0\circ\Delta_{10}\circ\phi$ equals $(\bareta_{\bullet\bullet})^l$ composed
with the isomorphism of twisted complexes $\infbarr(A)\simeq\infbarl(A)$ whose components are
negative identity maps, we have $d(\eta_0\circ\Delta_{10}\circ\phi)=-\infbarr(\id_A)$. Then
$$
\infbarr(\chi) - \infbarr(\id)+d(\eta_{10}\circ\Delta_1\circ\phi)=0,
$$
which is what we want.

\end{proof}

%end of edits

We can now prove the main theorem of this section which 
can be compared to similar results for usual 
$\Ainfty$-modules, cf.~\cite[\S4.1.3]{Lefevre-SurLesAInftyCategories}.

\begin{theorem}
\label{theorem-tfae-unitality-conditions-for-A-modules}
Let $A$ be strongly homotopy unital and let $(E,p_\bullet)$ be an 
$\Ainfty$-$A$-module. The following are equivalent:
\begin{enumerate}
\item 
\label{item-A-module-is-homotopy-unital}
$(E,p_\bullet)$ is homotopy unital, 
\item 
\label{item-A-module-is-H-unital}
$(E,p_\bullet)$ is $H$-unital. 
\end{enumerate}
If $A$ is bimodule homotopy unital, these are further equivalent to:
\begin{enumerate}
\setcounter{enumi}{2} 
\item 
\label{item-A-module-is-strong-homotopy-unital}
$(E,p_\bullet)$ is strongly homotopy unital. 
\end{enumerate}
\end{theorem}

We need two preliminary results. The first constructs a specific 
contracting homotopy of $\infbarl A$, the bar-construction of $A$ as 
a left $A$-$\Ainfty$-module in $\nodA$:

\begin{lemma}
\label{lemma-the-contracting-homotopy-zeta-of-infbarl-a}
Let $A$ be bimodule homotopy unital with the unit
$\bareta_{\bullet\bullet}: \id \rightarrow A$. Let $\delta_{ij}\colon
A^i \rightarrow A^j$ be the single component morphism 
$\id_{A_j}$ if $j \geq i \geq 0$ and $0$ otherwise. 
Let $\zeta\colon \infbarl A \rightarrow \infbarl \id$ be the map 
whose $ij$-th component is $(-1)^{ij}\delta_{-(i-1),-j}$. 

Then the following composition is a contracting homotopy of $\infbarl{A}$:
$$ \infbarl{A} \xrightarrow{\zeta} \infbarl{\id_A}
\xrightarrow{\bareta} \infbarl{A}. $$
\end{lemma}
\begin{proof}
Direct computation.
\end{proof}

Let $(E,p_\bullet)$ be an $\Ainfty$-$A$-module. Recall that
the composition $\rho \chi$ being homotopic to $\id_{(E,p_\bullet)}$ is
the definition of strong homotopy unitality. We now prove that, 
on the contrary, the composition $\chi \rho$ is always homotopic 
to $\id_{\infbarres(E,p_\bullet)}$:
 
\begin{lemma}[Chi-Rho Lemma]
\label{lemma-chi-rho-lemma}
Let $A$ be bimodule homotopy unital and let $(E,p_\bullet) \in \nodA$. There is a degree $-1$ endomorphism 
$\xi$ of $\infbarres (E,p_\bullet)$ with 
$$d\xi = \id_{\infbarres(E,p_\bullet)} - \chi \circ \rho. $$ 
\end{lemma}
\begin{proof}
Recall the isomorphism  
$E\infbarl(A) \xrightarrow{E\phi} \infbarres(E,0)$
of twisted complexes over $\nodA$ used in the construction of the map $\chi$. 
Write 
$$ E(\zeta\circ\bareta)\colon \infbarres(E,0) \rightarrow \infbarres(E,0) $$
for the degree $-1$ morphism induced via $E\phi$
from the contracting homotopy $\zeta\circ\bareta$ of $\infbarl(A)$
of Lemma \ref{lemma-the-contracting-homotopy-zeta-of-infbarl-a}. 
Let $\id\colon (E,p_\bullet) \rightarrow (E,0)$
and $\id\colon (E,0) \rightarrow (E,p_\bullet)$
be the non-closed maps 
whose only non-zero component is $\id_E$. Their differentials 
are $\pi_1 \circ \id$ and $\id \circ \pi_1$ where $\pi_1: (E,p_\bullet)
\rightarrow (E,p_\bullet)$ is as in 
\S\ref{section-bar-construction-as-a-complex-of-ainfty-A-modules}.

Set $\xi$ to be the composition 
\begin{equation}
\infbarres(E,p_\bullet)
\xrightarrow{ \infbarres(\id)} 
\infbarres(E,0)
\xrightarrow{E(\zeta\circ\bareta)}
\infbarres(E,0)
\xrightarrow{\infbarres(\id) }
\infbarres(E,p_\bullet). 
\end{equation}
We then have
\begin{align*}
d(\xi) = 
\infbarres(\pi_1)\xi + \id - \xi\infbarres(\pi_1) = 
\id - [\xi,\infbarres(\pi_1)]. 
\end{align*}
One readily verifies that $[\xi,\infbarres(\pi_1)] = \chi\rho$, whence
the desired assertion. 
\end{proof}

If $A$ is only strong homotopy
unital, we do not have $\bareta$ but only $\eta$,
$h^l_\bullet$, and $h^r_\bullet$. 
Hence we do not have the contracting homotopy
$\zeta \circ \bareta$ of Lemma
\ref{lemma-the-contracting-homotopy-zeta-of-infbarl-a}, but 
we do have its first component $\zeta_1$ which involves only $\eta$ and
$h^l_\bullet$. The first component is the image
under the forgetful functor 
$\forget\colon \pretriagmns \nodA \rightarrow \pretriagmns \A$. 
Thus $\zeta_1$ is a contracting homotopy of $\forget(\infbarl(A))$,
and we still have a version of Chi-Rho Lemma for the first components: 
\begin{cor}
\label{cor-chi-rho-lemma-for-the-first-components}
Let $A$ be strongly homotopy unital and let $(E,p_\bullet) \in \nodA$. 
There is a degree $-1$ endomorphism 
$\xi_1$ of $\forget(\infbarres (E,p_\bullet))$ in $\pretriagmns \A$ with 
$$d\xi_1 = \id_{\forget(\infbarres(E,p_\bullet))} - E\eta \circ \forget(\rho). $$ 
\end{cor}

\begin{proof}[Proof of Theorem~\ref{theorem-tfae-unitality-conditions-for-A-modules}]
$\eqref{item-A-module-is-homotopy-unital} 
\Leftrightarrow 
\eqref{item-A-module-is-H-unital}:$

Observe that $\infbar(E,p_\bullet)$ is the total complex of the
twisted complex map $\forget(\rho)$.  Thus $(E,p_\bullet)$ is $H$-unital if 
and only if $\forget(\rho)$ is a homotopy equivalence. By the 
first component version of Chi-Rho Lemma
(Cor.~\ref{cor-chi-rho-lemma-for-the-first-components}), 
$\forget(\rho)$ has a left homotopy inverse $E\eta$. Thus
$\forget(\rho)$ is a homotopy equivalence if and only if $E\eta$ is
also a right homotopy inverse of $\forget(\rho)$. Since 
$\forget(\rho) \circ E\eta = p_2 \circ \eta$, 
$E\eta$ is a right homotopy inverse of $\forget(\rho)$
if and only if $p_2 \circ \eta$ is homotopic to $\id_E$, 
which is the definition of $(E,p_\bullet)$ being homotopy unital.

$\eqref{item-A-module-is-strong-homotopy-unital} 
\Leftrightarrow 
\eqref{item-A-module-is-H-unital}:$

This is similar, but with $A$ bimodule homotopy unital we can 
witness the firepower of our fully armed and operational 
Chi-Rho Lemma (Lemma \ref{lemma-chi-rho-lemma}).

By Homotopy Lemma, $\forget(\rho)$ is a homotopy equivalence if and
only if $\rho$ is. By Chi-Rho Lemma, $\rho$ has a left homotopy inverse $\chi$.
Thus $(E,p_\bullet)$ is $H$-unital if and only if $\chi$ is also 
a right homotopy inverse of $\rho$, which is the definition of
strong homotopy unitality. 
\end{proof}

\subsection{Free-Forgetful homotopy adjunction}
\label{section-free-forgetful-homotopy-adjunction}

We now give the main application
of strong homotopy unitality: $\free\colon \A \rightarrow \nodA$ and
$\forget\colon \nodA \rightarrow \A$  
defined in \S\ref{section-free-modules-and-bimodules-over-Ainfty-algebra}
become homotopy adjoint functors if we restrict to the subcategory
$\nodhuA$ of homotopy unital modules. Note that as $\forget$ 
discards all the higher components of
$\Ainfty$-morphisms, there is no hope for a genuine adjunction. 

\begin{defn}
\label{defn-free-forgetful-adjunction-unit-counit}
Let $A$ be a strongly homotopy unital $\Ainfty$-algebra in a monoidal DG
category $\A$. Define the \em unit \rm natural transformation 
$\id_{A} \rightarrow \forget \free$
by setting it to be $ E \xrightarrow{E\eta} EA$ for any $E \in \A$. 
Define the \em counit \rm homotopy natural transformation 
$\free \forget \rightarrow \id_{\nodA}$
by setting it to be
$EA \xrightarrow{\pi_2} (E,p_\bullet)$
for any $(E,p_\bullet) \in \nodA$. 
\end{defn}

We first show that the counit map above is indeed a homotopy natural
transformation. Recall that $\pi_2$ is one of the maps $\pi_i: EA^i
\rightarrow (E,p_\bullet)$ defined
in \S\ref{section-bar-construction-as-a-complex-of-ainfty-A-modules}. 
\begin{lemma}
Let $A$ be an $\Ainfty$-algebra in a monoidal DG category $\A$. 
Then $\pi_2\colon EA \rightarrow (E,p_\bullet)$ defines a homotopy 
natural transformation $\free\forget \rightarrow \id_{\nodA}$. 
\end{lemma}
\begin{proof}
We need to show that for any closed degree $0$ morphism
$f_\bullet\colon (E,p_\bullet) \rightarrow (F,q_\bullet)$ in $\nodA$
the following square commutes up to homotopy:
\begin{equation}
\label{eqn-pi_2-homotopy-natural-transformation-square}
\begin{tikzcd}
EA 
\ar{r}{\pi_2}
\ar{d}{f_1 A}
&
(E,p_\bullet)
\ar{d}{f_\bullet}
\\
FA
\ar{r}{\pi_2}
&
(F,q_{\bullet}).  
\end{tikzcd}
\end{equation}

Observe that the square \eqref{eqn-pi_2-homotopy-natural-transformation-square}
is the restriction of the square
\eqref{eqn-infbar-f_bullet-rewritten-as-square-with-bar-resolution-Ainfty-version}
of one-sided twisted complexes to the leading term of each complex. 
We saw in \S\ref{section-bar-resolution} that the square 
\eqref{eqn-infbar-f_bullet-rewritten-as-square-with-bar-resolution-Ainfty-version} is a presentation of 
$\infbar(f_\bullet)$ 
and thus for closed $f$ it commutes up to the homotopy 
given by $f_{\bullet+\bullet}: \infbarres(E,p_\bullet) \rightarrow
(F,q_\bullet)$. It follows that
\eqref{eqn-pi_2-homotopy-natural-transformation-square} commutes up
to the homotopy given by $f_{\bullet+1}$, 
the restriction of $f_{\bullet+\bullet}$ to the leading terms. 
\end{proof}

\begin{theorem}
\label{theorem-free-forgetful-homotopy-adjunction}
Let $A$ be a strongly homotopy unital $\Ainfty$-algebra in a monoidal DG
category $\A$. The functors $\free$ and $\forget$ with the unit and
the counit maps of Definition \ref{defn-free-forgetful-adjunction-unit-counit}
are a homotopy adjoint pair of functors $\A \leftrightarrows \nodhuA$. 
\end{theorem}
\begin{proof}
Let $E \in \A$. The composition 
$\free \xrightarrow{\free(\unit)} \free\forget\free
\xrightarrow{\counit} \free$
evaluated at $E$ is a $\nodA$ morphism
$EA \xrightarrow{E \eta A} EA^2 \xrightarrow{E \mu_2} EA$. 
This morphism is homotopic to $\id_{EA}$ since 
$A \xrightarrow{\eta A} A^2 \xrightarrow{\mu_2} A$
is homotopic to $\id_A$ by
the definition of strong homotopy unitality of $A$.

Let $(E,p_\bullet) \in \nodhuA$. The composition 
$\forget \xrightarrow{\unit} \forget\free\forget
\xrightarrow{\forget(\counit)} \forget$
evaluated at $(E,p_{\bullet})$ is an $\A$ morphism 
$E \xrightarrow{E\eta} EA \xrightarrow{p_2} E$. 
It being homotopic to $\id_E$ is the definition of 
homotopy unitality of $(E, p_\bullet)$. 
\end{proof}

\subsection{Kleisli category}
\label{section-kleisli-category}

Given an $\Ainfty$-algebra $A$ in a monoidal DG category $\A$ we can 
construct a (non-unital) $\Ainfty$-category in the sense of 
\cite[\S5]{Lefevre-SurLesAInftyCategories} in the same way 
as we construct the Kleisli category of a monad:

\begin{defn}
\label{defn-kleisli-category-of-an-ainfty-algebra}
Let $A$ be an $\Ainfty$-algebra in a monoidal DG category $\A$. Define
its \em Kleisli category \rm $\kleisliA$ to be the $\Ainfty$-category
(in the classical sense of \cite{Lefevre-SurLesAInftyCategories})
defined by the following data:
\begin{itemize}
\item Its objects are the objects of $\A$. 
\item For any $E,F \in \A$ the $\homm$-complex between them is
\begin{equation}
\homm_{\kleisliA}(E,F) := \homm_{\A}(E,FA).
\end{equation}
\item For any $E_1, E_2, \dots, E_{n+1} \in \A$ and any $\alpha_i \in 
\homm_{\kleisliA}(E_i, E_{i+1})$ define 
\begin{equation}
\label{eqn-defn-of-ainfty-structure-on-kleisli-category}
m_n^{\kleisliA}(\alpha_1, \dots, \alpha_n) := 
E_1 \xrightarrow{\alpha_1} E_2A \xrightarrow{\alpha_2 A} \dots
\xrightarrow{\alpha_n A^{n-1}} E_{n+1}A^n \xrightarrow{E_{n+1}m_n^A}
E_{n+1}A.
\end{equation}
\end{itemize}
\end{defn}

\begin{lemma}
\label{lemma-ainfty-category-kleisliA-is-well-defined}
The $\Ainfty$-category $\kleisliA$ in Definition
\ref{defn-kleisli-category-of-an-ainfty-algebra}
is well-defined. 
\end{lemma} 
\begin{proof}
We need to show that the operations $m_i^{\kleisliA}$
satisfy the equations
\eqref{eqn-defining-equalities-for-new-definition-of-Ainfinity-algebra}. 
These can be reduced to the same equations for $m_i^A$. 
Consider the equation
\begin{equation}
\label{eqn-dm3-equation-for-the-kleisli-category}
d m^{\kleisliA}_3 + m^{\kleisliA}_2 \circ \left( \id \otimes
m^{\kleisliA}_2 - m^{\kleisliA}_2 \otimes \id \right) = 0.  
\end{equation}
For any 
$E_1 \xrightarrow{\alpha_1} E_2 \xrightarrow{\alpha_2} E_3
\xrightarrow{\alpha_3} E_4 \in \kleisliA$
the morphism $m^{\kleisliA}_2 \circ
\left( \id \otimes m^{\kleisliA}_2\right) (\alpha_3,
\alpha_2, \alpha_1)$ in $\kleisliA$ is defined by the composition in $\A$ which
forms the upper right perimeter of the following commutative diagram:
\begin{equation}
\begin{tikzcd}[column sep = 1.5cm, row sep = 0.5cm]
E_1
\ar{r}{\alpha_1}
&
E_2A
\ar{r}{\alpha_2 A}
&
E_3A^2
\ar{r}{E_3 m^A_2}
\ar{d}{\alpha_3 A^2}
&
E_3 A
\ar{d}{\alpha_3 A}
\\
& 
&
E_4 A^3
\ar{r}{E_4 A m^A_2}
&
E_4 A^2
\ar{d}{E_4 m^A_2}
\\
& & &
E_4 A.
\end{tikzcd}
\end{equation}
The diagram commutes by the bifunctoriality of the monoidal operation
of $\A$, thus  
$$ 
m^{\kleisliA}_2 \circ \left( \id \otimes m^{\kleisliA}_2\right) 
(\alpha_3, \alpha_2, \alpha_1)
=  
\id_{E_4} \otimes \left(m^A_2 \circ \left( \id \otimes
m^A_2\right)\right)
\circ \left(\alpha_3 A^2 \circ \alpha_2 A \circ \alpha_1\right). 
$$
Similarly, we have 
$$
m^{\kleisliA}_2 \circ \left( m^{\kleisliA}_2  \otimes \id \right) 
(\alpha_3, \alpha_2, \alpha_1)
=  
\id_{E_4} \otimes \left(
m^A_2 \circ \left( m^A_2 \otimes \id \right)\right) 
\circ \left(\alpha_3 A^2 \circ \alpha_2 A \circ \alpha_1\right), 
$$
$$ 
dm^{\kleisliA}_3 (\alpha_3, \alpha_2, \alpha_1) = 
(\id_{E_4} \otimes dm^{\A}_3) \circ (\alpha_3 A^2 \circ \alpha_2 A
\circ \alpha_1). 
$$
We conclude that \eqref{eqn-dm3-equation-for-the-kleisli-category} applied 
to $(\alpha_3, \alpha_2, \alpha_1)$ is $\id_{E_4}$ tensored
with the analogous equation for $\A$ and pre-composed with
$\alpha_3 A^2 \circ \alpha_2 A \circ \alpha_1$. 

The remaining equations 
\eqref{eqn-defining-equalities-for-new-definition-of-Ainfinity-algebra}
reduce similarly to those for $m^A_i$. 
\end{proof}

When $A$ is homotopy unital, so is clearly $\kleisliA$:

\begin{defn}
Let $A$ be a homotopy unital $\Ainfty$-algebra in a monoidal DG
category $\A$. Define a homotopy unital structure on $\kleisliA$ by setting
for any $E \in \A$ the unit morphism 
$\id_E \in \homm_{\kleisliA}(E,E)$
to be the morphism corresponding to the morphism
$E\eta\colon E \rightarrow EA$ in $\A$. Here $\eta$ is the unit
morphism $\id_\A \rightarrow A$ of $A$. 
\end{defn}

When $A$ is strongly homotopy unital, the Free-Forgetful adjunction
ensures that $\kleisliA$ is quasi-equivalent to the DG category 
$\freeA$:

\begin{defn}
\label{defn-kleisli-to-free-Ainfty-functor}
Let $A$ be an $\Ainfty$-algebra in a monoidal DG category $\A$. 
Define an $\Ainfty$-functor
$ f_\bullet\colon \kleisliA \rightarrow \freeA $
by setting 
$f_{\text{obj}}\colon \obj(\kleisliA) \rightarrow \obj(\freeA)$
be the map $E \mapsto EA$ and by setting 
$$ f_i: \homm_{\kleisliA}(E_i, E_{i+1}) \otimes_k 
\dots \otimes_k \homm_{\kleisliA}(E_1 , E_2) \rightarrow
\homm_{\freeA}(E_1A, E_{i+1}A) $$
to be the map which sends any $\alpha_i \otimes \dots \otimes
\alpha_1$ to the following morphism in $\nodA$: 
$$ E_1 A \xrightarrow{\alpha_1A} E_2 A^2 \xrightarrow{\alpha_2A^2}
\dots \xrightarrow{\alpha_i A^i} E_{i+1} A^{i+1} \xrightarrow{E_{i+1} \mu_{i+1}}
E_{i+1} A. $$
\end{defn}

\begin{theorem}
\label{theorem-ainfty-quasi-equivalence-from-kleisli-to-free}
The $\Ainfty$-functor 
$ f_\bullet\colon \kleisliA \rightarrow \freeA $
of Defn.~\ref{defn-kleisli-to-free-Ainfty-functor} is
well-defined. If $A$ is strongly homotopy unital, it is a quasi-equivalence. 
\end{theorem}
\begin{proof}
To show that $f_\bullet$ is well defined we need to show that its
bar-construction $\infbar(f_\bullet)$ is a closed morphism of twisted
complexes. By
Prps.~\ref{prps-defining-equalities-of-Ainfty-morphism}
it suffices to show that the equations
\eqref{eqn-defining-equalities-for-new-definition-of-Ainfinity-algebra-morphism}
hold for $f_\bullet$. Similarly to the proof of Lemma
\ref{lemma-ainfty-category-kleisliA-is-well-defined}, the $i$-th equation 
\eqref{eqn-defining-equalities-for-new-definition-of-Ainfinity-algebra-morphism}
for $f_\bullet$ reduces to the 
$(i+1)$-st equation for $m^A_{\bullet}$. To be more precise, it
reduces to the $\nodA$-lift of this equation which is obtained
by replacing the twisted complex $\infbarnaug(A)$ in $\A$ with 
the twisted complex $\infbar(A)$ in $\nodA$ as described in 
Prps.~\ref{prps-bar-construction-as-a-complex-of-Ainfty-modules}. 

To show that $f_\bullet$ is a quasi-equivalence, we need to show that
$f_{obj}$ is quasi-essentially surjective and $f_1$ is a quasi-isomorphism 
on all $\homm$-complexes. For the former, $f_{obj}$ is actually
bijective. For the latter, under the identification 
of $\homm_{\kleisliA}(E,F)$ with $\homm_{\A}(E, FA)$, the map
\begin{equation}
\label{eqn-f_1-on-homm-spaces-from-A-to-free-A}
f_1\colon \homm_{\A}(E, FA) \rightarrow \homm_{\freeA}(EA,FA)
\end{equation}
sends any $\alpha\colon E \rightarrow FA$ to the composition 
$ EA \xrightarrow{\alpha A} FA^2 \xrightarrow{F\mu_2} FA. $
Its first composant is the functor $\free$ applied to $\alpha$ and
the second is the counit of the Free-Forgetful adjunction. Thus 
\eqref{eqn-f_1-on-homm-spaces-from-A-to-free-A} is the Free-Forgetful
adjunction map. By Theorem \ref{theorem-free-forgetful-homotopy-adjunction}
it is a quasi-isomorphism since $A$ is strongly homotopy unital. 
\end{proof}

\section{The derived category}
\label{section-the-derived-category}

Our results so far were formulated for a small monoidal DG category $\A$ 
and its Yoneda embedding into $\modA$. As 
explained in \S\ref{section-the-setting}, they work the same for 
a non-small $\A$ if we replace $\modA$ with a cocomplete, strongly
pretriangulated, closed monoidal DG category $\B$. Using $\modA$ 
instead of a general $\B$ was merely an exposition choice, as 
any particulars specific to $\modA$ and the Yoneda embedding were 
not relevant to our considerations so far.  

Now this changes. To describe the compact
objects in the derived category $D(A)$ we look at 
the objects of $\A$ which are compact in $H^0(\B)$.
When $\B = \modA$ this is the whole of $\A$, however
we also want to treat the cases when it is not. For example, 
when $\A = \modC$ for a small DG category $\C$ and $\B = \A$. 

We thus switch to working with a general 
ambient monoidal DG category $\B$ which is closed, cocomplete and
strongly pretriangulated. We also assume that 
$\A$ contains a set of compact objects which generate it in $H^0(\B)$. 
Then all compact objects of $H^0(\A)$ also form a set 
\cite[Theorem 5.3]{Keller-DerivingDGCategories}. When $\A$ is small, 
this assumption holds for $\B = \modA$. 

To consider $\Ainfty$-structures in $\B$, we set its ambient category 
to be $\B$ itself. Any $\Ainfty$-algebra $(A, \mu_i)$ in $\A$ is 
also an $\Ainfty$-algebra in $\B$. We write 
$\noddinf\text{-}A^\B$ for the category of right $\Ainfty$-$A$-modules
in $\B$. For any $C \subseteq \obj \B$ we write
$\noddinf\text{-}A^C$ for the full subcategory comprising modules whose
underlying objects lie in $C$. 

\subsection{The general case}
\label{section-the-derived-category-the-general-case}

If $A$ is a DG algebra over $k$, then the unbounded derived 
category $D(A)$ of $A$ is the category of DG $A$-modules 
localised by quasi-isomorphisms. The localisation is achieved by taking
the Verdier quotient by $\acyc$, the subcategory of acyclic
modules. Thus we set $D(A) : = H^0(\modd\text{-}A)/\acyc$. 
 
Let $A$ be a classical (non-unital) $\Ainfty$-algebra considered in 
\cite[\S1]{Lefevre-SurLesAInftyCategories}. An elegant approach to
defining the derived category is given in
\cite[\S4]{Lefevre-SurLesAInftyCategories}. First, for $A$ strictly
unital, we define $D(A)$ to be $H^0(\moddinf\text{-}A)$, where 
$\modd_\infty\text{-}A$ is the full subcategory of $\nodA$ comprising
strictly unital $A$-modules. We do not factor out acyclics, since
by the Homotopy Lemma all quasi-isomorphisms are already homotopy 
equivalences. For a general non-unital $A$, we consider its augmentation 
$A^+ := A \oplus k$ to a strictly unital algebra. 
The derived category $D(A)$ is defined to be the kernel of 
the derived functor $D(A^+) \rightarrow D(k)$ of 
extension of scalars along the augmentation projection $A^+ \rightarrow k$. 
The logic is that when $A$ is at least $H$-unital the derived functor 
of restriction of scalars is fully faithful, 
so we have a canonical semi-orthogonal decomposition 
of $D(A^+)$ into $D(A)$ and the derived category $D(k)$ of the
augmentation unit. 

The extension of scalars functor $D(A^+) \rightarrow D(k)$
sends any $(E,p_\bullet) \in \moddinf\text{-}A^+$ to its bar-construction 
as an $A$-module. Thus $D(A) \simeq H^0(\nodhuA)$, where
$\nodhuA$ is the full subcategory of $\nodA$ consisting of
$H$-unital modules. 

We expect the unbounded derived category to be triangulated and
cocomplete. In our fully general setting, $H^0(\nodA)$ and $H^0(\nodhuA)$ are
apriori neither. Indeed, both are only triangulated and cocomplete if $\A$ is, 
cf.~\cite[Cor.~5.13]{AnnoLogvinenko-UnboundedTwistedComplexes}.  
In particular, $H^0(\noddinf\text{-}A^\B)$ is triangulated and 
cocomplete. We thus define: 

\begin{defn}
\label{defn-the-derived-category-of-an-ainfty-algebra}
Let $A$ be an $\Ainfty$-algebra in a monoidal DG category $\A$.
The \em (unbounded) derived category \em $D(A)$ of $A$ is the triangulated 
cocomplete hull of $H^0(\nodhuA)$ in $H^0(\noddinf\text{-}A^\B)$. 
The \em compact derived category \rm $D_c(A)$ is the full subcategory 
of $D(A)$ comprising its compact objects. 

A module $(E, p_\bullet) \in \nodhuA$ is \em perfect \rm if it is
a compact object in $D(A)$. Let $\nodhupfA$ denote 
the full subcategory of $\nodhuA$ comprising perfect modules. 
\end{defn}

In this generality, we don't even know whether all free $A$-modules
lie in $D(A)$. 

\subsection{$H$-unital case}
\label{section-the-derived-category-the-h-unital-case}

When $A$ is $H$-unital, $D(A)$ contains all free $A$-modules and is 
their cocomplete triangulated hull:
\begin{prps}
\label{prps-derived-category-of-A-is-triangulated-cocomplete-hull-of-freeA}
Let $A$ be an $H$-unital $\Ainfty$-algebra in a monoidal DG
category $\A$. Then 
$ \left< \freeA \right>_{tr,cc} = D(A) =
H^0(\noddinfhu\text{-}A^{\left<\A\right>_{tr,cc}})$.
\end{prps}
\begin{proof}
For any $E \in \A$ we have $\infbar(EA) = E \otimes \infbarnaug(A)$.
Since $A$ is $H$-unital, $\infbarnaug(A)$ is null-homotopic, 
and hence so is $\infbar(EA)$. Thus $\free(A)$ is contained in 
$\nodhuA$ and therefore in $\noddinfhu\text{-}A^{\left<\A\right>_{tr,cc}}$. 
Taking 
triangulated cocomplete hulls we obtain the inclusions
$\left< \freeA \right>_{tr,cc} \subseteq D(A) \subseteq
H^0(\noddinfhu\text{-}A^{\left<\A\right>_{tr,cc}})$.

Conversely, any $(E,p_{\bullet}) \in \noddinfhu\text{-}A^{\left<\A\right>_{tr,cc}}$
is homotopy equivalent to its bar-resolution
$\infbarres(E,p_{\bullet})$ by
Prop.~\ref{prps-natural-transformation-rho-is-acyclic-on-H-unital-modules-Ainfty-version}. Since the object $E$ lies in $\left<\A\right>_{tr,cc}$,
each free module $EA^i$ lies in $\left< \freeA \right>_{tr,cc}$. 
Since $\infbarres(E,p_{\bullet})$ is a bounded above one-sided 
twisted complex of $EA^i$s, its convolution also lies in 
$\left< \freeA \right>_{tr,cc}$. Thus the inclusions above are
equalities. 
\end{proof}

\begin{defn}
Let $A$ be an $H$-unital $\Ainfty$-algebra in a monoidal DG
category $\A$. Let $\freepfA$ be the full subcategory of $\freeA$
comprising perfect modules. 
\end{defn}

To give similar descriptions of $D(A)$ and $D_c(A)$ in terms
of $\freepfA$ we need to know whether the perfect free $A$-modules 
generate everything. We assumed that $\A$ is generated by its compact 
objects in $H^0(\B)$, but it isn't apriori true that 
an $\Ainfty$-$A$-module $(E,p_\bullet)$ is perfect if and only if
its underlying object $E$ is perfect. 

\subsection{Strongly homotopy unital case}
\label{section-the-derived-category-the-strong-homotopy-unital-case}

When $A$ is strongly homotopy unital, the Free-Forgetful
homotopy adjunction allows for the best description of $D(A)$.  

\begin{defn}
An object $E \in \A$ is \em perfect \rm if it is a compact object 
in $H^0(\B)$. We denote by $\Aperf$ the full subcategory of $\A$
comprising the perfect objects.  
\end{defn}

By our assumptions, $\Aperf$ is a small category. 

\begin{lemma}
\label{lemma-perfect-in-nodhuaA-iff-perfect-in-A}
Let $A$ be a strongly homotopy unital $\Ainfty$-algebra in 
a monoidal DG category $\A$. For any $E \in \A$, the free module 
$EA \in \nodhuA$ is perfect if and only if $E \in \A$ is perfect.
\end{lemma}
\begin{proof}
The Free-Forgetful
homotopy adjunction of Theorem \ref{theorem-free-forgetful-homotopy-adjunction}
gives a genuine adjunction on the level of the homotopy categories:
$$ \homm_{H^0(\noddinfhu\text{-}A^{\B})}(EA, -) \simeq 
\homm_{H^0(\B)}(E, \forget(-)). $$

By definition, $D(A)$ is a full subcategory of
$H^0(\noddinfhu\text{-}A^{\B})$ and, as $A$ is strongly homotopy unital, 
the free functor factors through this full subcategory. Since the 
forgetful functor commutes with small direct sums, it follows that 
$\homm_{D(A)}(EA,-)$ commutes with them if and only if
$\homm_{H^0(\B)}(E,-)$ does. 
\end{proof}
 
Let $S \subset \Aperf$ be any set of generators of $H^0(\A)$ in 
$H^0(\B)$. Denote by $\freeAS$ the full subcategory of $\nodA$ 
comprising the free modules on objects $S$. 
  
\begin{prps}
\label{prps-perfect-iff-lies-in-hperf-of-perfect-generator-frees}
Let $A$ be a strongly homotopy unital $\Ainfty$-algebra in 
a monoidal DG category $\A$. 
Let $S$ be a set of compact generators of $H^0(\A)$ in $H^0(\B)$. 
Then $\freeAS$ is a set of compact generators of $D(A)$ and
the following are equivalent:
\begin{enumerate}
\item 
\label{item-A-infinity-module-is-perfect}
$(E,p_\bullet) \in \nodhuA$ is perfect. 
\item 
\label{item-A-infinity-module-lies-in-hperf-of-perfect-generator-frees}
$(E,p_\bullet) \in \nodhuA$ is a homotopy direct summand of something in 
$\pretriag(\freeAS))$. 
\end{enumerate}
If these equivalent conditions hold, then furthermore $(E,p_\bullet)$
is a homotopy direct summand of a finite truncation of its bar resolution
$\infbarres(E,p_\bullet)$. 
\end{prps}
\begin{proof}

Since $S$ generates $\A$ in $H^0(\B)$, 
$\free(S)$ generates $\free(\A)$ in 
$H^0(\noddinf\text{-}A^{\B})$. The objects of $\freeA$ and $\freeAS$
are the same as those of $\free(\A)$ and $\free(S)$, and 
$\freeA$ generates $D(A)$ by Prop.~\ref{prps-derived-category-of-A-is-triangulated-cocomplete-hull-of-freeA}. Thus $\freeAS$ is a set of compact
generators of $D(A)$. 

$\eqref{item-A-infinity-module-is-perfect} \Rightarrow
\eqref{item-A-infinity-module-lies-in-hperf-of-perfect-generator-frees}$:
By \cite[Theorem 5.3]{Keller-DerivingDGCategories} the compact objects
of $D(A)$ lie in the Karoubi completion of the triangulated hull of 
$\freeAS$ in $D(A)$. Hence any perfect $(E,p_\bullet)$ is a homotopy
direct summand of something in $\pretriag(\freeAS)$. 

$\eqref{item-A-infinity-module-lies-in-hperf-of-perfect-generator-frees}
\Rightarrow \eqref{item-A-infinity-module-is-perfect}$:
By Lemma \ref{lemma-perfect-in-nodhuaA-iff-perfect-in-A} the objects
of $\freeAS$ are perfect in $\nodhuA$. If $(E,p_\bullet)$ is a homotopy 
direct summand of something in $\pretriag(\freeAS))$, its image in $D(A)$ 
lies in the Karoubi completion of the triangulated hull of $\freeAS$. 
It is therefore also compact, since both taking triangulated hull and 
taking Karoubi completion preserves compactness. Hence $(E,p_\bullet)$ is
perfect. 

The final assertion that $(E,p_\bullet)$
is a homotopy direct summand of a finite truncation of its bar resolution
$\infbarres(E,p_\bullet)$ follows by the same argument in the
proof of \cite[Theorem 5.3]{Keller-DerivingDGCategories}. That
proof uses arbitrary resolution of a compact object, and the bar resolution 
is an instance of one. 
\end{proof}

In \S\ref{section-kleisli-category} we defined the
Kleisli category $\kleisliA$ of $A$. It is an $\Ainfty$-category in 
the classical sense \cite{Lefevre-SurLesAInftyCategories}. Let 
$\kleisliAS$ be its full subcategory on the objects of $S$. 
\begin{theorem}
\label{theorem-compact-derived-category-of-A-is-that-of-frees-and-kleisli}
Let $A$ be a strongly homotopy unital $\Ainfty$-algebra in a monoidal 
DG category $\A$. Let $S$ be a set of compact generators of $H^0(\A)$ 
in $H^0(\B)$. Then 
$$ D_c(A) \simeq D_c(\freeAS) \simeq D_c(\kleisliAS). $$
\end{theorem}
\begin{proof}
The second equivalence $D_c(\freeAS) \simeq D_c(\kleisliAS)$ is induced 
by the $\Ainfty$-quasi-equivalence $\kleisliAS \rightarrow \freeAS$
constructed in Theorem
\ref{theorem-ainfty-quasi-equivalence-from-kleisli-to-free}. 
To construct the first equivalence, we note that 
$D_c(\freeAS) = H^0(\hperf(\freeAS))$. 
By the universal property of $\hperf$ the fully faithful embedding 
$\freeAS \hookrightarrow \noddinf\text{-}A^{\B}$
induces a fully faithful embedding 
$ \hperf(\freeAS) \hookrightarrow \noddinf\text{-}A^{\B}$
since the target category is pre-triangulated and homotopy Karoubi complete. 
On the homotopy level we have thus a fully faithful embedding
$ D_c(\freeAS) \hookrightarrow D(A) $
whose essential image is the Karoubi complete triangulated hull of 
$\freeAS$. By the proof of  
Prop.~\ref{prps-perfect-iff-lies-in-hperf-of-perfect-generator-frees}
$D_c(A)$ is the Karoubi complete triangulated hull
of $\freeAS$ in $D(A)$, so the embedding above restricts
to an equivalence $D_c(\freeAS) \simeq D_c(A)$.   
\end{proof}

\begin{cor}
\label{cor-quasi-isomorphism-of-Ainfty-algebra-implies-equiv-of-derived-cats}
Let $A$ and $B$ be a strongly homotopy unital $\Ainfty$-algebras in a monoidal 
DG category $\A$. If there exists a morphism $f_\bullet\colon A
\rightarrow B$ of $\Ainfty$-algebras with $f_1\colon A \rightarrow B$ 
a homotopy equivalence in $A$, then $D_c(A) \simeq D_c(B)$. 
\end{cor}
\begin{proof}
Any morphism $f_\bullet\colon A
\rightarrow B$ induces an $\Ainfty$-functor
$\mathfrak{f}_\bullet: \kleisliA \rightarrow \kleisliB$ which is
identity on objects and whose 
$\mathfrak{f_1}\colon \homm_\kleisliA(E,F) \rightarrow 
\homm_\kleisliB(E,F)$ is given by 
$(\id_F \otimes f_1) \circ (-)$. 
If $f_1$ is a homotopy equivalence, then 
$\mathfrak{f}_1$ is a quasi-isomorphism, and thus $\mathfrak{f}$
is a quasi-equivalence. A quasi-equivalence of classical $\Ainfty$-categories 
induces an equivalence of their derived categories. 
\end{proof}

Theorem \ref{theorem-compact-derived-category-of-A-is-that-of-frees-and-kleisli}
shows that $D_c(A)$ is 
independent of the choice of the ambient category $\B$ where 
we take infinite direct sums. Since this is not true for
$D(A)$, we can't hope to relate it to $D(\freeAS) \simeq D(\kleisliAS)$ for a general $\B$. However, we can when $\A$ is small
and $\B = \modA$:

\begin{prps}
\label{prps-when-B-is-modA-we-get-the-classical-derived-category}
Let $A$ be a strongly homotopy unital $\Ainfty$-algebra in a small monoidal 
DG category $\A$. Let $\B = \modA$ with the induced monoidal
structure. 
$$ D(A) \simeq D(\freeA) \simeq D(\kleisliA). $$
\end{prps}
\begin{proof}
As before, by Theorem
\ref{theorem-ainfty-quasi-equivalence-from-kleisli-to-free} 
we have $D(\freeA) \simeq D(\kleisliA)$. 
Thus it suffices to demonstrate that $D(A) \simeq D(\kleisliA)$.
We first construct a functor
$$ f_\bullet\colon \noddinf\text{-}A^{\modA} \rightarrow
\noddinf\text{-}\kleisliA. $$
The monoidal structure induced on $\modA$ from $\A$, cf. 
\cite[\S4.5]{GyengeKoppensteinerLogvinenko-TheHeisenbergCategoryOfACategory}, 
has the property that for any $E \in \modA$ and $a \in \A$ the product
$E \homm_\A(-,a)$ is the $\A$-module 
$E \otimes_\A \homm_\A(-,-a),$
where $ \homm_\A(-,-a)$ is viewed as $\A$-$\A$-bimodule. Thus 
any $(E,p_\bullet) \in \noddinf\text{-}A^{\modA}$ is the data of
$\A$-module $E$ and structure morphisms
$$ p_i\colon  E \otimes_\A \homm_\A(-,-A) \otimes_\A \dots \otimes_\A \homm_\A(-,-A) \rightarrow E. $$
Furthermore, taking product in $\modA$ with 
the operations $m_i\colon A^i \rightarrow A$ corresponds to tensoring
over $\A$ with the operations defined 
in \eqref{eqn-defn-of-ainfty-structure-on-kleisli-category}
$$ m_i\colon \homm_\A(-,-A) \otimes_\A \dots \otimes_\A \homm_\A(-,-A)
\rightarrow \homm_\A(-,-A). $$
These are the $\Ainfty$-operations of $\kleisliA$ under 
$\homm_\A(-,-A) \simeq \homm_{\kleisliA}(-,-)$.
In other words, $\noddinf\text{-}A^{\modA}$ is isomorphic to 
the category of $\Ainfty$-modules in $\modA$ over bimodule $\Ainfty$-algebra
$\homm_\A(-,-A)$ in $\AmodA$. 

We now observe that the forgetful functors 
$\forget\colon \AmodA \rightarrow k_\A\text{-}\modd\text{-}k_\A$
and
$\forget\colon \modA \rightarrow \modd\text{-}k_\A$
are lax monoidal with respect to $\otimes_\A$ and $\otimes_k$ and 
commute with infinite direct sums.  By definition, 
$\noddinf\text{-}\kleisliA$ is the category of $\Ainfty$-modules in 
$\modd\text{-}k_\A$ over the $\Ainfty$-algebra $\kleisliA$ in 
$k_\A\text{-}\modd\text{-}k_\A$. This algebra is the image under $\forget$ 
of the bimodule $\Ainfty$-algebra $\homm_\A(-,-A)$ in $\AmodA$. 
Hence the forgetful functors induce a functor
$$  \forget_\bullet\colon \noddinf\text{-}A^{\modA} \rightarrow
\noddinf\text{-}\kleisliA,$$ 
which commutes with infinite direct sums. 

We now claim that $\forget_\bullet$ is quasi-fully faithful on $\freeA$.
Indeed, the composition of $\forget_\bullet$ with the functor 
$\kleisliA \rightarrow \freeA$ constructed in Theorem
\ref{theorem-ainfty-quasi-equivalence-from-kleisli-to-free} is 
the Yoneda embedding 
$\kleisliA \hookrightarrow \noddinf\text{-}\kleisliA$
 of a (classical) $\Ainfty$-category 
is as in \cite[\S7]{Lefevre-SurLesAInftyCategories}. 
The functor constructed in Theorem \ref{theorem-ainfty-quasi-equivalence-from-kleisli-to-free} is a
quasi-equivalence since $A$ is strongly homotopy unital. So to show
that $\forget_\bullet$ is quasi-faithful on $\freeA$ it suffices to show
that the Yoneda embedding is quasi-fully faithful. This holds because
translated into our framework, $\Ainfty$-category $\kleisliA$ is the 
Kleisli category of the $\Ainfty$-algebra $\kleisliA$ in
$k_\A\text{-}\modd\text{-}k_\A$. The Yoneda embedding is then
the composition of the $\Ainfty$-functor of Theorem
\ref{theorem-ainfty-quasi-equivalence-from-kleisli-to-free} 
from Kleisli category to the free $\Ainfty$-modules with the inclusion of
these into all $\Ainfty$-modules. It is therefore quasi-fully faithful 
since $A$, and hence $\kleisliA$, are strongly homotopy unital. 

Thus $\forget_\bullet$ gives an exact functor
$H^0(\noddinf\text{-}A^{\modA}) \rightarrow
H^0(\noddinf\text{-}\kleisliA)$ which is fully faithful on $\freeA$,
sends $\freeA$ to $\kleisliA$, and commutes with infinite direct sums. 
It gives therefore an equivalence of cocomplete triangulated hulls of 
$\freeA$ and $\kleisliA$ which are $D(A)$ and $D(\kleisliA)$, respectively. 
\end{proof}

Another common case worth considering separately is when $\B = \A$:
\begin{prps}
\label{prps-when-B-is-A-we-get-the-???}
Let $A$ be a strongly homotopy unital $\Ainfty$-algebra in a monoidal 
DG category $\A$. Suppose that $\A$ is cocomplete, strongly
pre-triangulated, its monoidal structure is closed, 
and $H^0(\A)$ is compactly generated. In other words, it
satisfies our assumptions on $\B$, and we set $\B = \A$.
Then 
\begin{align*}
D(A) =  H^0(\nodhuA)
\quad \quad \text{ and } \quad \quad
D_c(A) \subseteq 
H^0(\noddinfhu\text{-}A^{\A^{pf}})
\end{align*}
\end{prps}
\begin{remark}
\label{remark-not-all-perfect-object-modules-are-perfect}
The inclusion in the second assertion can be strict. Set $k = \mathbb{C}$, 
$\A = \modk$, $A = k[x^2, xy, y^2]$, and consider the strict $A$-module 
$A/(x^2, xy, y^2)$. 
Its underlying object of $\modk$ is $k$ concentrated in degree $0$,
thus it lies in $H^0(\noddinfhu\text{-}A^{\A^{pf}})$. On the other hand, as explained
in \S\ref{section-examples-associative-algebras}, the derived category 
of $A$ as an $\Ainfty$-algebra in $\modk$ coincides with the usual
derived category of $A$ as an associative algebra. Thus 
$A/(x^2, xy, y^2)$ is not compact object in $D(A)$. 
\end{remark}

\begin{proof}
The first assertion follows from the second equality in
Prop.~\ref{prps-derived-category-of-A-is-triangulated-cocomplete-hull-of-freeA}.

For the compact derived category, since $A$ is strong homotopy unital
the modules in $\free\text{-}(\A^{pf})$ are perfect. 
By Prop.~\ref{prps-perfect-iff-lies-in-hperf-of-perfect-generator-frees} every 
perfect module in $\nodhuA$ is a homotopy direct summand of a bounded twisted
complex of $\free\text{-}(\A^{pf})$. Therefore its underlying object
of $\A$ is a homotopy direct summand of a bounded twisted complex of
perfect objects and thus perfect itself. 
\end{proof}

\subsection{Strict algebra case}
\label{section-the-derived-category-strict-algebra-case}

In this section we examine the case when $A$ is a strict algebra. 
We can then consider the category of strict modules over $A$:

\begin{defn}
\label{defn-strict-algebra-strict-module-categories}
Let $(A,\mu)$ be a strict algebra in a monoidal DG category $\A$. 
Define $\nodstrA$ to be the subcategory of $\nodA$ comprising 
strict $A$-modules with strict morphisms between them whose
differentials are also strict. Define $\nodstrhuA$ to be 
the full subcategory of $\nodstrA$ comprising the $H$-unital modules. 
If $A$ is strictly unital, denote by $\modd\text{-}A$ the full
subcategory of $\nodstrhuA$ comprising strictly unital modules. 
\end{defn}

By definition, the functor $\free\colon \A \rightarrow \nodhuA$
factors through $\nodstrhuA$. For any $(E,p) \in \nodstrhuA$ 
the adjunction counit $\free \circ \forget \rightarrow \id_{\nodhuA}$ is 
the strict morphism  $EA \rightarrow (E,p)$ given by $p$. We now
verify that the Free-Forgetful homotopy adjunction of 
\S\ref{section-free-forgetful-homotopy-adjunction} restricts 
from $\nodhuA$ to its non-full subcategory $\nodstrhuA$:

\begin{lemma}
\label{lemma-free-forgetful-homotopy-adjunction-restricts-to-strict-modules}
Let $A$ be a strict, strongly homotopy unital algebra 
in a monoidal DG category $\A$. The functors 
$\free, \forget \colon \A \leftrightarrows \nodstrhuA$
are homotopy adjoint with the unit is given by $E \xrightarrow{E\eta} EA$ 
for any $E \in \A$ and the counit by $EA \xrightarrow{p} (E,p)$. 
\end{lemma}
\begin{proof}
Let $E \in \A$. The composition 
$\free \xrightarrow{\free(\unit)} \free\forget\free
\xrightarrow{\counit} \free$
evaluated at $E$ is a $\nodstrA$ morphism
$E\mu \circ E \eta A \colon EA \rightarrow EA$. 
As $A$ is strongly homotopy unital, it is homotopy unital
and so $\mu \circ {\eta}A$ is homotopic to $\id_A$. Hence 
the composition above is homotopic to $\id_{EA}$. 

Let $(E,p) \in \nodstrhuA$. Since $A$ is strongly homotopy unital, 
$(E,p)$ is homotopy unital. The composition 
$\forget \xrightarrow{\unit} \forget\free\forget
\xrightarrow{\forget(\counit)} \forget$
evaluated at $(E,p)$ is an $\A$ morphism 
$p \circ E\eta \colon  E \rightarrow E$. 
It being homotopic to $\id_E$ is the definition of 
homotopy unitality of $(E, p)$. 
\end{proof}

Given any subcategory $\C$ of $\noddinfhu\text{-}A^{\B}$, 
we denote by $\acyc_\C$ the full subcategory of $\C$ comprising
acyclic objects, i.e. those modules whose underlying object is
null-homotopic. Where no confusion is possible, we write
$\acyc$ for $\acyc_\C$. 

\begin{defn}
Let $A$ be a strict algebra in a monoidal DG category $\A$.
A module $E \in \nodstrA$ is \em h-projective \rm 
if $\homm_{\nodstrA}^\bullet(E,F)$ is acyclic for any acyclic $F$.  
\end{defn}
\begin{lemma}
Let $A$ be a strict, strongly homotopy unital algebra in a monoidal 
DG category $\A$. Any free $A$-module is $h$-projective. 
\end{lemma}
\begin{proof}
Let $E \in \A$. 
Let $(F,q) \in \nodstrA$ be an acyclic module, thus $F$ is
null-homotopic in $\A$. By Free-Forgetful 
adjunction we have a quasi-isomorphism 
$$ \homm^\bullet_{\nodstrA}(EA, (F,q)) \rightarrow 
\homm^\bullet_{\A}(E, F). $$
The RHS is null-homotopic in $\modk$, and hence acyclic. 
\end{proof}

The bar-resolution $\infbarres$  allows us to describe 
the derived category $D(A)$ as a Verdier quotient:
\begin{theorem}
\label{theorem-derived-category-as-localisation-in-the-strict-case}
Let $A$ be a strict, strongly homotopy unital algebra in a monoidal 
DG category $\A$. Then we have an exact equivalence
$$ D(A) \simeq H^0(\nodd^{hu}\text{-}A^{\left<\A\right>_{tr,cc}}) / \acyc. $$
\end{theorem}
\begin{proof}
The same argument using the bar-resolution as in 
Prop.~\ref{prps-derived-category-of-A-is-triangulated-cocomplete-hull-of-freeA}
shows that
$$ \left< \freeA \right>_{tr,cc}^{strict}
=
H^0(\nodd^{hu}\text{-}A^{\left<\A\right>_{tr,cc}}) , $$
where on the LHS we take the triangulated cocomplete hull 
in $H^0(\nodd\text{-}A^\B)$. 

Consider the non-full inclusion 
$$ 
\nodd^{hu}\text{-}A^{\left<\A\right>_{tr,cc}}
\hookrightarrow 
\noddinfhu\text{-}A^{\left<\A\right>_{tr,cc}}. $$
By Prop.~\ref{prps-derived-category-of-A-is-triangulated-cocomplete-hull-of-freeA} the homotopy category of the RHS is $D(A)$, so we obtain an exact
non-full inclusion of triangulated categories
$$
H^0(\nodd^{hu}\text{-}A^{\left<\A\right>_{tr,cc}})
\hookrightarrow 
D(A).
$$
The source category is cocomplete and the inclusion respects direct
sums, thus its essential image is triangulated and cocomplete.
As the image contains $\freeA$, it must be the whole of $D(A)$. 
By Homotopy Lemma, the inclusion kills all acyclic modules. It
therefore induces an essentially surjective exact functor 
$$
H^0(\nodd^{hu}\text{-}A^{\left<\A\right>_{tr,cc}})/\acyc
\longrightarrow 
D(A).
$$
To show it to be fully faithful, it is enough to show it on a set 
of compact generators. By above, the free modules generate
$H^0(\nodd^{hu}\text{-}A^{\left<\A\right>_{tr,cc}})$ and thus its
Verdier quotient. Since $\A^{pf}$ generates $\A$ in $H^0(\B)$, compact 
free modules generate all free modules in $H^0(\nodd\text{-}A^\B)$. Since 
free modules are $h$-projective, $\homm$-spaces between them in the
Verdier quotient are the same as in 
$H^0(\nodd^{hu}\text{-}A^{\left<\A\right>_{tr,cc}})$ itself. 

It suffices now to show that for any $E,F \in \A$ the natural inclusion 
$$\homm_{\nodstrA}(EA,FA)
\hookrightarrow \homm_{\nodA}(EA,FA)$$
is a quasi-isomorphism.  By Lemma 
\ref{lemma-free-forgetful-homotopy-adjunction-restricts-to-strict-modules}
the Free-Forgetful homotopy adjunction restricts from
$\Ainfty$-modules to strict modules. Hence it gives quasi-isomorphisms 
from $\homm_{\nodstrA}(EA,FA)$ and 
$\homm_{\nodA}(EA,FA)$
to $\homm_{\A}(EA,F)$
which intertwines their natural inclusion. 
\end{proof}

\section{$\Ainfty$-coalgebras and comodules} 
\label{section-ainfty-coalgebras-and-comodules}

Most of the definitions and results for $\Ainfty$-algebras, modules, 
and their derived categories given in
\S\ref{section-ainfty-structures-in-monoidal-dg-categories}-\S\ref{section-the-derived-category}
can be duplicated for $\Ainfty$-coalgebras and comodules. 
In this section, we give the first few as an example, 
and leave the rest to the reader.  

A major subtlety is that with $\Ainfty$-algebras and their modules 
the bar-constructions (see Definitions
\ref{defn-algebra-bar-construction-in-a-monoidal-category},
\ref{defn-the-bar-construction-of-a-morphism-of-ainfty-algebras},
\ref{defn-right-module-bar-construction-in-a-monoidal-category},
\ref{defn-left-module-bar-construction-in-a-monoidal-category},
\ref{defn-right-module-bar-constructions-of-an-Ainfty-morphism}, and
\ref{defn-left-module-bar-constructions-of-an-Ainfty-morphism})  
have only a finite number of arrows coming out of any term of the
complex. Thus the definitions of
$\Ainfty$-algebras, modules, and their morphisms are
independent of the ambient category $\B$ used to define 
our unbounded twisted complexes. 

This is no longer the case for the cobar-construction. The condition
that the cobar-construction is a twisted complex -- over $\B$ -- now
involves a non-trivial condition that the sum of all its maps factors 
through the infinite direct sum -- in $\B$ -- 
of the terms of the complex. Thus what is and isn't an
$\Ainfty$-coalgebra or an $\Ainfty$-comodule depends on the choice of 
the ambient category $\B$. 
In this section, we always assume that a choice of $\B$ satisfying 
the assumptions in \S\ref{section-the-setting} and 
\S\ref{section-the-derived-category} is fixed and
"twisted complex" means a twisted complex in $\twbicx^{\pm}_\B(\A)$.

A consequence of the above is that the Homotopy Lemma 
(Lemma \ref{lemma-the-homotopy-lemma-for-nodA}) doesn't work for
$\Ainfty$-comodules. The reason is that the contracting homotopy
our method constructs doesn't apriori satisfy the condition  
that its cobar-construction is a valid map of twisted complexes in 
$\twbicx^{\pm}_\B(\A)$. We resolve some of the resulting 
complications by working only with strongly homotopy counital comodules. 
This makes it necessary to assume our
$\Ainfty$-coalgebra to be bicomodule homotopy counital. 

\subsection{$\Ainfty$-coalgebras}
\label{section-Ainfty-coalgebras-in-a-monoidal-category}
 
As in \S\ref{section-Ainfty-algebras-in-a-monoidal-category}, 
we first define the notion of a cobar-construction of a
collection of operations, and then say that these operations define
an $\Ainfty$-coalgebra if the cobar-construction is a twisted complex:

\begin{defn}
\label{defn-coalgebra-cobar-construction-in-a-monoidal-category}
Let $\A$ be a monoidal DG category, let $C \in \A$ and let 
$\left\{\Delta_i\right\}_{i \geq 2}$ be a collection of degree $2-i$
morphisms $C \rightarrow C^i$. 
Their \em (non-augmented) cobar-construction $\infbarnaug(C)$ \rm  
comprises objects $C^{i+1}$ for all $i \geq 0$ each placed
in degree $i$ and degree $k-1$ maps 
$d_{i(i+k)}\colon C^{i} \rightarrow C^{i+k}$ 
defined by
\begin{equation}
\label{eqn-differentials-in-non-aug-cobar-construction}
d_{i(i+k)} := (-1)^{(i-1)(k+1)} \sum_{j = 0}^{i-1} (-1)^{jk} 
\id^{i-j-1}\otimes \Delta_{k+1} \otimes \id^{j}. 
\end{equation} 
\begin{tiny}
\begin{equation}
\label{eqn-nonaugmented-cobar-construction-of-A-m_i}
\begin{tikzcd}[column sep = 2.5cm]
\underset{\degzero}{C}
\ar{r}[']{\Delta_2}
\ar[bend left=20]{rr}[description]{\Delta_3}
\ar[bend left=25]{rrr}[description]{\Delta_4}
\ar[bend left=30]{rrrr}[description]{\Delta_5}
&
C^2
\ar{r}[']{C\Delta_2-\Delta_2C}
\ar[bend left=20]{rr}[description]{-C\Delta_3 - \Delta_3C}
\ar[bend left=25]{rrr}[description]{C\Delta_4-\Delta_4C}
& 
C^3
\ar{r}[']{C^2\Delta_2-C\Delta_2C+\Delta_2C^2}
\ar[bend left=20]{rr}[description]{C^2\Delta_3+C\Delta_3C+\Delta_3C^2}
&
C^4
\ar{r}[']{\begin{smallmatrix}C^3 \Delta_2  - C^2\Delta_2C + \\ + C\Delta_2C^2 - \Delta_2 C^3 \end{smallmatrix}}
&
\dots
\end{tikzcd}
\end{equation}
\end{tiny}

\end{defn}
\begin{defn}
\label{defn-ainfty-coalgebra-in-a-monoidal-category}
Let $\A$ be a monoidal DG category. An \em $\Ainfty$-coalgebra 
$(C,\Delta_i)$ \rm in $\A$ is an object $C \in \A$ equipped 
with degree $2-i$ morphisms $\Delta_i \colon C \rightarrow C^i$ 
for all $i \geq 2$ whose non-augmented 
cobar-construction $\infbarnaug(C)$ is a twisted complex over $\A$. 
\end{defn}
\begin{defn} 
Let $(C, \Delta_k)$ and $(D, \Delta'_k)$ be $\Ainfty$-coalgebras in $\A$. 
Let $(f_i)_{i \geq 1}$ be a collection of degree $1-i$ morphisms 
$C \rightarrow D^i$. 
The \em cobar-construction \rm $\infbar(f_\bullet)$ is 
a collection of morphisms $C^{i} \rightarrow D^{i+k}$ defined by 
$$ \sum_{t_1 + \dots + t_i = i + k}
(-1)^{\sum_{l=2}^{i}(1-t_l)\sum_{n=1}^l t_n
} 
f_{t_1}\otimes\ldots \otimes f_{t_i} $$
\end{defn}

\subsection{$\Ainfty$-comodules}
\label{section-Ainfty-comodules-in-a-monoidal-category}

\begin{defn}
\label{defn-right-comodule-cobar-construction-in-a-monoidal-category}
Let $(C,\Delta_i)$ be an $\Ainfty$-coalgebra in a monoidal DG category $\A$. 
For $G \in \A$ and a collection $\left\{r_i\right\}_{i \geq 2}$ 
of degree $2-i$ morphisms $G \rightarrow G\otimes C^{i-1}$, 
their \em right comodule cobar-construction $\infbar(G)$ \rm 
comprises objects $G \otimes C^{i}$ for $i \geq 0$ 
placed in degree $i$ and degree $1-k$ maps 
$G \otimes C^{i-1} \rightarrow G \otimes C^{i+k-1}$ 
defined by
\begin{scriptsize}
\begin{equation}
\label{eqn-differentials-in-right-comodule-cobar-construction}
d_{i(i+k)} := (-1)^{(i-1)(k+1)} 
\left(\sum_{j = 0}^{i-2} \left( (-1)^{jk} \id^{i-j-1} \otimes \Delta_{k+1} \otimes
\id^{j}\right) + (-1)^{(i-1)k}r_{k+1}\otimes \id^{i-1} \right). 
\end{equation} 
\end{scriptsize}

For $G \in \A$ and a collection 
$\left\{r_i\right\}_{i \geq 2}$ of degree $2-i$
morphisms $ G\rightarrow C^{i-1}\otimes G$
their \em left comodule cobar-construction $\infbar(G)$ \rm
comprises objects $C^{i}\otimes G$ for all $i \geq 0$ placed
in degree $i$ and degree $1-k$ maps 
$C^{i-1}\otimes G \rightarrow C^{i+k-1}\otimes G$ 
defined by
\begin{equation}
\label{eqn-differentials-in-left-comodule-cobar-construction}
d_{i(i+k)} := (-1)^{(i-1)(k+1)} 
\left(\sum_{j = 1}^{i-1} \left( (-1)^{jk} \id^{i-j-1} \otimes \Delta_{k+1} \otimes
\id^{j}\right) + \id^{i-1} \otimes r_{k+1} \right). 
\end{equation} 
\end{defn}

\begin{defn}
Let $(C,\Delta_i)$ be an $\Ainfty$-coalgebra in a monoidal DG category
$\A$. A \em right (resp. left) $\Ainfty$-comodule 
$(G, r_i)$ over $C$ \rm is an object $G \in \A$ and a collection  
$\left\{r_i\right\}_{i \geq 2}$ of degree $2-i$
morphisms $G \rightarrow G \otimes C^{i-1}$ (resp. 
$G \rightarrow C^{i-1} \otimes G$) such that $\infbar(G)$ is a twisted complex. 
\end{defn} 

\begin{defn} 
Let $(C,\Delta_i)$ be an $\Ainfty$-coalgebra in a monoidal DG category
A \em degree $j$ morphism \rm 
$f_\bullet\colon (G, r_k) \rightarrow (H, s_k)$ of 
right $\Ainfty$-$C$-comodules is a collection $(f_i)_{i \geq 1}$ of 
degree $j - i + 1$ morphisms $G\rightarrow H \otimes C^{i-1}$. Its
\em cobar-construction \rm $\infbar(f_\bullet)$ is 
the morphism $\infbar(G) \rightarrow \infbar(H)$ in $\pretriagpls(\A)$
whose components are  
$$
G \otimes C^{i-1} \rightarrow H \otimes C^{i+k-1}\colon \;
(-1)^{j(i-1)}  f_{k+1}\otimes \id^{i-1}. $$
\end{defn}

\begin{defn} 
Let $(C,\Delta_i)$ be an $\Ainfty$-coalgebra in a monoidal DG category
A \em degree $j$ morphism $f_\bullet\colon (G, r_k) \rightarrow (H, s_k)$ of 
left $\Ainfty$-$C$-comodules \rm is a collection $(f_i)_{i \geq 1}$ of 
degree $j - i + 1$ morphisms $ G \rightarrow C^{i-1} \otimes H$. 
Its \em cobar-construction \rm $\infbar(f_\bullet)$ is 
the morphism $\infbar(G) \rightarrow \infbar(H)$ in $\pretriagpls(\A)$
whose components are 
$$ A^{i-1} \otimes G \rightarrow C^{i+k-1} \otimes H \colon \;
(-1)^{(j+k)(i-1)}  \id^{i-1} \otimes f_{k+1}. $$
\end{defn}

\begin{defn}
Let $(C,\Delta_i)$ be an $\Ainfty$-coalgebra in a monoidal DG category
Define the \em DG category $\conodC$ of right $\Ainfty$-$C$-comodules in $\A$ \rm by:
\begin{itemize}
\item Its objects are right $\Ainfty$-$C$-comodules in $\A$,
\item For any $G, H \in \obj \conodC$, $\homm^\bullet_{\conodC}(G, H)$
consists of $\Ainfty$-morphisms $f_\bullet\colon G \rightarrow H$
with their natural grading. 
The differential and the composition
are induced from those of their cobar-constructions.
\item For any $G \in \conodC$, its identity map $f_\bullet$ has 
$f_1 = \id_G$ and $f_{\geq 2} = 0$. 
\end{itemize}
The \em DG category $\Cconod$ of left $\Ainfty$-$C$-comodules 
in $\A$ \rm is defined analogously.  
\end{defn}

\subsection{Bicomodules}
\label{section-Ainfty-bicomodules-in-a-monoidal-category}

\begin{defn}
\label{defn-bicomodule-bar-construction-in-a-monoidal-category}
Let $(C,\Delta_i)$ and $(D,\Theta_i)$ be an $\Ainfty$-coalgebras 
in a monoidal DG category $\A$. 
Let $M \in \A$ and let $\left\{r_{ij}\right\}_{i + j \geq 1}$ be a collection 
of degree $1-i-j$ morphisms $M \rightarrow C^i \otimes M \otimes D^{j}$. 
The \em bicomodule cobar-construction $\infbar(M)$ \rm 
comprises objects $C^i \otimes M \otimes D^{j}$ with $i + j \geq 0$ 
placed in bidegree $i,j$ and degree $1+p+q-i-j$ maps 
$$ C^{i} \otimes M \otimes D^{j} \longrightarrow 
\bigoplus_{p + q = i + j + k} C^p \otimes M \otimes D^q $$
defined by
\begin{equation}
\label{eqn-differentials-in-the-bimodule-cobar-construction}
(-1)^{(i+j)(k+1)} 
\sum_{r = 0}^{i+j} (-1)^{rk} \id^{i + j - r} \otimes (\Delta r\Theta)_{k+1}
\otimes \id^{r},
\end{equation} 
where $(\Delta r \Theta)_{k+1}$ denotes the unique operation --- 
either $r_{s,t}$ with $s+t = k$ or $\Delta_{k+1}$ or $\Theta_{k+1}$ --- 
that can be applied to the corresponding factor
of $C^{i} \otimes M \otimes D^{j}$. 
\end{defn}

\begin{defn}
An \em $\Ainfty$-bicomodule \rm over $\Ainfty$-coalgebras $(C,\Delta_i)$ and $(D,\Theta_i)$ 
in a monoidal DG category $\A$ is
an object $M \in \A$ and a collection $\left\{r_{ij}\right\}_{i + j \geq 1}$ 
of degree $1-i-j$ morphisms $ M\rightarrow C^i \otimes M \otimes D^{j}$
such that $\infbar(M)$ is a twisted complex. 
\end{defn}

\begin{defn} 
Let $(C,\Delta_i)$ and $(D,\Theta_i)$ be an $\Ainfty$-coalgebras 
in a monoidal DG category $\A$.  Let $(M, r_{ij})$ and $(N, s_{ij})$ be
$\Ainfty$-$C$-$D$-bicomodules. 
A \em degree $k$ morphism \rm $f_{\bullet\bullet}\colon 
(M, r_{ij}) \rightarrow (N, s_{ij})$ of 
$\Ainfty$-$C$-$D$-bicomodules is a collection $(f_{lm})_{l + m \geq 0}$ of 
degree $k - l - m$ morphisms $M \rightarrow C^l \otimes  N \otimes D^{m}$. 
Its \em cobar-construction \rm $\infbar(f_{\bullet\bullet})$ is 
the morphism $\infbar(M) \rightarrow \infbar(N)$ in $\twbiospls(\A)$
whose components are  
$$
C^{i} \otimes M \otimes D^{j} \rightarrow C^{i+l} \otimes N \otimes D^{j+m}\colon \;
(-1)^{i(l+m) + k (i+j)}  \id^{i} \otimes f_{l,m} \otimes \id^{j}. $$
\end{defn}

\begin{defn}
Let $(C,\Delta_i)$ and $(D,\Theta_i)$ be $\Ainfty$-coalgebras in a monoidal DG category 
$\A$. The \em DG category $\CconodD$ of
$\Ainfty$-$C$-$D$-bicomodules in $\A$ \rm is:
\begin{itemize}
\item Its objects are $\Ainfty$-$C$-$D$-bicomodules in $\A$,
\item For any $M,N \in \CconodD$,  
$\homm^\bullet_{\CconodD}(M,N)$
consists of $\Ainfty$-morphisms $f_{\bullet\bullet}\colon M \rightarrow N$
with their natural grading. The differential and
the composition are induced from those of their cobar constructions. 
\item For any $M \in \CconodD$ its identity map $f_{\bullet\bullet}$
has $f_{00} = \id_M$, $f_{ij, i+j > 0} = 0$. 
\end{itemize}
\end{defn}

\subsection{Notions of homotopy counitality}
\label{section-notions-of-homotopy-counitality}

We define the notions of
\begin{itemize}
\item $H$-counitality, strong homotopy counitality, and bicomodule homotopy
counitality for $\Ainfty$-coalgebras,
\item $H$-counitality, homotopy counitality, and 
strong homotopy counitality for $\Ainfty$-comodules,
\end{itemize}
analogously to the way they are defined for algebras and modules
in \S\ref{section-unitality-conditions-for-algebras}-\ref{section-unitality-conditions-for-A-modules}. 

Since the Homotopy Lemma (Lemma
\ref{lemma-the-homotopy-lemma-for-nodA}) doesn't hold for
$\Ainfty$-comodules, we do not get an analogue of Theorem
\ref{theorem-tfae-unitality-conditions-for-A-modules}. Its proof
relies on the Homotopy Lemma to show that $H$-unitality implies strong homotopy
unitality. Indeed,  the Chi-Rho Lemma (Lemma \ref{lemma-chi-rho-lemma})
shows that an $\Ainfty$-$A$-module $(E,p_\bullet)$ is strongly homotopy unital 
if and only if the canonical map 
$\rho\colon \infbarres(E,p_\bullet) \rightarrow (E,p_\bullet) $
is a homotopy equivalence in $\nodA$. On the other hand, $H$-unitality
means that $\forget(\rho)$ is a homotopy equivalence in $\A$. 
In absence of the Homotopy Lemma, the two are not equivalent. 
Thus for $\Ainfty$-comodules, we only get:
\begin{prps}
Let $(C,\Delta_\bullet)$ be a strongly homotopy counital coalgebra 
in a monoidal DG category $\A$ and let $(G, r_\bullet)$ 
be an $\Ainfty$-$C$-comodule. Then $(G, r_\bullet)$ is homotopy counital iff 
it is $H$-counital. 
\end{prps}
For strongly homotopy counital $\Ainfty$-coalgebras and their homotopy
counital comodules we have the \em Forgetful-Free homotopy adjunction\rm, 
analogous to the Free-Forgetful homotopy adjunction for algebras and modules
described in \S\ref{section-free-forgetful-homotopy-adjunction}. Note 
that the direction of the adjunction has changed. We also have:
\begin{defn}
Let $(C,\Delta_\bullet)$ be a strongly homotopy counital $\Ainfty$-coalgebra 
in a monoidal DG category $\A$. Define its \em co-Kleisli category 
$\cokleisliC$ \em to be the $\Ainfty$-category
(in the classical sense of \cite{Lefevre-SurLesAInftyCategories})
defined by the following data:
\begin{itemize}
\item Its objects are the objects of $\A$. 
\item For any $G, H \in \A$ the $\homm$-complex between them is
\begin{equation}
\homm_{\cokleisliC}(G, H) := \homm_{\A}(GC, H).
\end{equation}
\item For any $G_1, G_2, \dots, G_{n+1} \in \A$ and any $\alpha_i \in 
\homm_{\cokleisliC}(G_i, G_{i+1})$ define 
\begin{equation}
\label{eqn-defn-of-ainfty-structure-on-cokleisli-category}
m_n^{\cokleisliC}(\alpha_1, \dots, \alpha_n) := 
G_1C \xrightarrow{G_{1}\Delta_n} G_1C^n \xrightarrow{\alpha_1C^{n-1}} G_2C^{n-1} \xrightarrow{\alpha_2 C^{n-2}} \dots \xrightarrow{\alpha_n} G_{n+1}.
\end{equation}
\end{itemize}
\end{defn}
We have the natural $\Ainfty$-functor 
$f_\bullet\colon \cokleisli(C) \rightarrow \free (C)$
defined analogously to Definition
\ref{defn-kleisli-to-free-Ainfty-functor}. When $C$ is strong
homotopy counital, it is a quasi-equivalence.

However, when constructing the derived category of an
$\Ainfty$-coalgebra $C$, the notion of $H$-counitality of its
$\Ainfty$-comodules is not enough. We want
the derived category to be the cocomplete triangulated hull of 
the free comodules. For $\Ainfty$-algebras, any $H$-unital module was
homotopy equivalent to its bar-resolution, and hence had a resolution by
free modules. 

For $\Ainfty$-coalgebras, $H$-counitality of a 
comodule means that its cobar-construction is null-homotopic as 
a twisted complex over $\A$. Without the Homotopy Lemma, this is no longer 
equivalent to its cobar-construction being null-homotopic as 
a twisted complex over $\conodC$. The latter is what we need to have 
the cobar-resolution by free modules. We therefore introduce a new 
notion:
\begin{defn}
Let $(C,\Delta_\bullet)$ be a coalgebra 
in a monoidal DG category $\A$. We say that $(G, r_\bullet) \in \conodC$ 
is {\it strongly $H$-counital} if its cobar-construction is null-homotopic 
as a twisted complex over $\conodC$. We denote the category of strongly 
$H$-counital comodules over $C$ by $\conodshuC$.
\end{defn}
This fixes the above mentioned place in the proof of Theorem
\ref{theorem-tfae-unitality-conditions-for-A-modules}:
\begin{lemma}
\label{lemma-tfae-strong-counitality-conditions-for-A-comodules}
Let $(C,\Delta_\bullet)$ be a bicomodule homotopy counital coalgebra 
in a monoidal DG category $\A$. Then any $(G, r_\bullet) \in \conodC$
is strongly homotopy counital iff it is strongly $H$-counital. 
\end{lemma}
The analogue of Lemma 
\ref{lemma-free-modules-are-strong-homotopy-unital}
then holds:
\begin{lemma}
$(C,\Delta_\bullet)$ be a bicomodule homotopy counital coalgebra 
in a monoidal DG category $\A$. Then free $C$-comodules are 
strongly homotopy counital.
\end{lemma}
Thus a comodule is 
strongly homotopy counital if and only if it has a free resolution. 
Finally, for the strongly homotopy counital modules we do have
an analogue of the Homotopy Lemma: 
\begin{lemma}
\label{lemma-coalgebra-analogue-of-the-homotopy-lemma}
Let $(C,\Delta_\bullet)$ be a bicomodule homotopy counital coalgebra 
in a monoidal DG category $\A$.
\begin{enumerate}
\item A comodule in $\conodshuC$ is acyclic iff it is null-homotopic.
\item A twisted complex in $\twcxub(\conodshuC)$ is acyclic iff it is null-homotopic.
\item A morphism of comodules $f_\bullet$ is a homotopy equivalence in $\conodshuC$
iff its component $f_1$ is a homotopy equivalence in $\A$.
\end{enumerate}
\end{lemma}
\begin{proof}
Suppose that a comodule $(G, r_\bullet)\in \conodshuC$ is acyclic. By forgetful-free adjunction, 
for every free comodule $HC$ we have
\begin{equation*}
\homm_C((G, r_\bullet), HC) \simeq \homm_\A(G, H) \simeq 0
\end{equation*}
since $G$ is null-homotopic in $\A$. Since free comodules are generators for $D(C)$,
this means that $(G, r_\bullet)\simeq 0$ in $D(C)=H^0(\conodshuC)$, meaning $(G, r_\bullet)$ 
is null-homotopic in $\conodshuC$. The rest is proved similarly to 
Lemma \ref{lemma-the-homotopy-lemma-for-nodA}.
\end{proof}

Thus the bottom line is: provided that one works with bicomodule
homotopy counital $\Ainfty$-coalgebras and their strongly $H$-counital/strongly
homotopy counital comodules, we have 
the analogues of all the results and definitions 
in \S\ref{section-ainfty-structures-in-monoidal-dg-categories} and 
\S\ref{section-strong-homotopy-unitality}. 

\subsection{The derived category}
\label{section-the-derived-category-of-ainfty-coalgebra}

When defining the derived category of an $\Ainfty$-coalgebra
$(C,\Delta_\bullet)$, we need to work with strongly $H$-counital comodules, 
cf.~\S\ref{section-notions-of-homotopy-counitality}. 
The definitions and results of
\S\ref{section-the-derived-category-the-general-case}-\ref{section-the-derived-category-the-h-unital-case} translate straightforwardly. 
The unbounded derived category $D(C)$ is then 
the cocomplete triangulated hull of $H^0(\conodshuC)$ 
in $H^0(\conodC^\B)$, and the derived category $D_c(C)$ is its compact
part. The derived category $D(C)$ is then generated by free
$C$-comodules, see Lemma
\ref{lemma-tfae-strong-counitality-conditions-for-A-comodules}. 

When trying to translate  
\S\ref{section-the-derived-category-the-strong-homotopy-unital-case}
and describe $D_c(C)$ in terms of the free comodules 
generated by the compact generators of $\A$, and hence in terms of the
co-Kleisli category of $C$, we run into a difficulty. For modules,  
the free module generated by a perfect object of $\A$ was also perfect.
For comodules, this is not apriori so, since we have 
the Forgetful-Free homotopy adjunction, and not the Free-Forgetful
one. 

Hence, apriori, there is no good description of $D_c(C)$
along the lines of 
\S\ref{section-the-derived-category-the-strong-homotopy-unital-case}
even if $C$ is bicomodule homotopy counital. Indeed, even if it is  
strictly counital. 

\subsection{Coalgebras with perfect free comodules}
\label{section-coalgebras-with-perfect-free-comodules}

We can impose an additional condition on 
$\Ainfty$-coalgebra $(C,\Delta_\bullet)$: the free comodules 
generated by perfect objects of $\A$ are perfect. 
Then all the arguments in 
\S\ref{section-the-derived-category-the-strong-homotopy-unital-case}
translate straightforwardly to coalgebras and comodules. 
In particular, we have:
\begin{theorem}
\label{theorem-compact-derived-category-of-bimodule-homotopy-unital-C-is-that-of-frees-and-cokleisli}
Let $C$ be a bicomodule homotopy counital $\Ainfty$-coalgebra in a
monoidal DG category $\A$ and $S$ be a set of compact generators 
of $H^0(\A)$ in $H^0(\B)$. 

If for every perfect $G \in \A$, the free $\Ainfty$-comodule $GC$ is
also perfect, then 
$$ D_c(C) \simeq D_c(\free_S\text{-}C) \simeq D_c(\cokleisliCS). $$
Here $\free_S\text{-}C$ denotes the full subcategory of the DG
category of free $C$-comodules comprising the free comodules
generated by the objects of $S$. Moreover, $\cokleisliCS$ 
denotes the full subcategory of $\cokleisliC$ comprising all objects of $S$.
\end{theorem}

This condition holds if there is a homotopy right adjoint
to $C$ in $\A$: 
\begin{lemma}
\label{lemma-has-right-adjoint-implies-perfect-objects-generate-perfect-comodules}
Let $C$ be a bicomodule homotopy counital $\Ainfty$-coalgebra in a
monoidal DG category $\A$. If there exists a homotopy right adjoint 
$A \in \A$ to $C$ as an object of $\A$, 
then for any perfect $G\in \A$ the free $\Ainfty$-comodule 
$GC \in \conoddinfshu\text{-}C$ is perfect.  
\end{lemma}
\begin{proof}
Let $\left\{ H_i \right\}_{i \in I}$ be a collection 
in $\conoddinfshu$-$C^\B$. Then $\bigoplus_{i \in I} H_i$ is also
strong $H$-counital, and thus homotopy equivalent to its cobar-resolution. 
We therefore have a natural homotopy equivalence in $\modk$
\begin{small}
\begin{equation*}
\begin{tikzcd}[column sep=0.5cm]
\homm_{C^\B}(GC,
\ar{d}{\sim}
&
\bigoplus_{i \in I} H_i)
&
&
&
\\
\homm_{C^\B}\Bigl(GC, 
&
(\bigoplus_{i \in I} H_i)C
\ar{r}
\ar[bend left=10]{rr}
\ar[bend left=12.5]{rrr}
&
(\bigoplus_{i \in I} H_i)C^2
\ar{r}
\ar[bend left=10]{rr}
&
(\bigoplus_{i \in I} H_i)C^3
\ar{r}
&
\dots
\Bigr)
\end{tikzcd}
\end{equation*}
\end{small}
For brevity, we write $\homm_{C^\B}$ for $\homm_{\conoddinf\text{-}C^\B}$. 
The bottom expression is isomorphic to the following twisted
complex over $\modk$
\begin{scriptsize}
\begin{equation*}
\begin{tikzcd}[column sep=0.5cm]
\homm_{C^\B}\Bigl(GC, (\bigoplus_{i \in I} H_i)C\Bigr)
\ar{r}
\ar[bend left=10]{rr}
\ar[bend left=12.5]{rrr}
&
\homm_{C^\B}\Bigl(GC, (\bigoplus_{i \in I} H_i)C^2\Bigr)
\ar{r}
\ar[bend left=10]{rr}
&
\homm_{C^\B}\Bigl(GC, (\bigoplus_{i \in I} H_i)C^3\Bigr)
\ar{r}
&
\dots
\end{tikzcd}
\end{equation*}
\end{scriptsize}
Similarly, $\bigoplus_{i \in I} \homm_{\C^\B}(GC, H_i)$ is homotopy
equivalent to 
\begin{scriptsize}
\begin{equation*}
\begin{tikzcd}[column sep=0.5cm]
\bigoplus_{i \in I} \homm_{C^\B}\Bigl(GC, H_iC\Bigr)
\ar{r}
\ar[bend left=10]{rr}
\ar[bend left=12.5]{rrr}
&
\bigoplus_{i \in I} \homm_{C^\B}\Bigl(GC,  H_iC^2\Bigr)
\ar{r}
\ar[bend left=10]{rr}
&
\bigoplus_{i \in I} \homm_{C^\B}\Bigl(GC, H_iC^3\Bigr)
\ar{r}
&
\dots
\end{tikzcd}
\end{equation*}
\end{scriptsize}
Under these identifications, the natural inclusion map 
\begin{equation}
\label{eqn-bigoplus-homm-to-homm-bigoplus-for-EC}
\homm_{C^\B}(GC, \bigoplus_{i \in I} H_i)
\hookrightarrow 
\bigoplus_{i \in I} \homm_{C^\B}(GC, H_i)
\end{equation}
corresponds to the closed degree zero map of twisted complexes
\begin{scriptsize}
\begin{equation*}
\begin{tikzcd}[column sep=0.5cm, row sep = 1cm]
\homm_{C^\B}\Bigl(GC, (\bigoplus_{i \in I} H_i)C\Bigr)
\ar{r}
\ar[bend left=10]{rr}
\ar[bend left=12.5]{rrr}
&
\homm_{C^\B}\Bigl(GC, (\bigoplus_{i \in I} H_i)C^2\Bigr)
\ar{r}
\ar[bend left=10]{rr}
&
\homm_{C^\B}\Bigl(GC, (\bigoplus_{i \in I} H_i)C^3\Bigr)
\ar{r}
&
\dots
\\
\bigoplus_{i \in I} \homm_{C^\B}\Bigl(GC, H_iC\Bigr)
\ar{r}
\ar[bend left=10]{rr}
\ar[bend left=12.5]{rrr}
\ar{u}
&
\bigoplus_{i \in I} \homm_{C^\B}\Bigl(GC, H_iC^2\Bigr)
\ar{r}
\ar[bend left=10]{rr}
\ar{u}
&
\bigoplus_{i \in I}\homm_{C^\B}\Bigl(GC, H_iC^3\Bigr)
\ar{r}
\ar{u}
&
\dots
\end{tikzcd}
\end{equation*}
\end{scriptsize}
whose components are postcompositions with 
inclusions $H_i C^j \hookrightarrow \bigoplus_{i \in I} H_i C^j$. 

Thus, to show that 
\eqref{eqn-bigoplus-homm-to-homm-bigoplus-for-EC} is a homotopy 
equivalence and hence $GC$ is perfect, it suffices to establish 
that for all $j \geq 1$ the folowing map is a homotopy equivalence:
$$
\bigoplus_{i \in I}\homm_{C^\B}\Bigl(GC,H_iC^j\Bigr)
\hookrightarrow 
\homm_{C^\B}\Bigl(GC,\bigoplus_{i \in I}H_iC^j\Bigr)
$$
By Forgetful-Free homotopy adjunction, this is equivalent to the map 
$$
\bigoplus_{i \in I}\homm_{\B}\Bigl(GC,H_iC^{j-1}\Bigr)
\hookrightarrow 
\homm_{\B}\Bigl(GC,\bigoplus_{i \in I}H_iC^{j-1}\Bigr)
$$
being one. By the homotopy adjunction of $C$ and
$A$ this is further equivalent to 
$$
\bigoplus_{i \in I}\homm_{\B}\Bigl(G, H_iC^{j-1}A\Bigr)
\hookrightarrow 
\homm_{\B}\Bigl(G,\bigoplus_{i \in I}H_iC^{j-1}A\Bigr)
$$
being a homotopy equivalence. This holds since $G$ was assumed 
to be perfect.  
\end{proof}

\subsection{Comparison to known constructions}
\label{section-comparison-to-known-constructions}

Let $C$ be a counital DG coalgebra in the classical sense,
i.e.~$\A=\B=\modk$.
Positselski showed
\cite{Positselski-TwoKindsOfDerivedCategoriesKoszulDualityAndComoduleContramoduleCorrespondence}
that the quotient of the homotopy category of
the category $\comodd\text{-}C$ of
all DG $C$-comodules by the acyclics is
equivalent to the homotopy category of the category $\comodd_{hinj}\text{-}C$
of $h$-injective DG $C$-comodules. He called this the derived category 
of $C$ and then argued that it is not invariant under quasi-isomorphisms 
of coalgebras, citing a counterexample by Kaledin.

A DG $C$-comodule is a strictly counital strict $\Ainfty$
$C$-comodule. As it is strictly counital, its cobar construction 
is acyclic. If it is $h$-injective, its cobar construction is 
$h$-injective being a bounded below complex of
$h$-injectives. Being both acyclic and $h$-injective, it is 
null-homotopic. We conclude that $h$-injective DG $C$-comodules
are strong homotopy counital. Similarly,  
$h$-injective homotopy counital $\Ainfty$ $C$-comodules are strong
homotopy counital. Conversely, any strong homotopy counital $\Ainfty$
$C$-comodule is homotopy counital and $h$-injective as it is homotopy
equivalent to a bounded below twisted complex of free comodules.  
Thus for DG comodules, $h$-injectivity is equivalent to strong
homotopy counitality. 

In $\comodd\text{-}C$ we have genuine Forgetful-Free adjunction, 
and in $\conodhuC$ we have a homotopy one. Thus for any DG
$C$-comodules $E$ and $F$ with $F$ free the natural
inclusion $\homm_{\comodd\text{-}C}(E,F) \hookrightarrow \homm_{\conodC}(E,F)$
is a quasi-isomorphism. The same is true for any $h$-injective $F$, 
since such $F$ is homotopy equivalent 
to a bounded below twisted complex of free comodules. 
Thus the natural inclusion $\comodd_{hinj}\text{-}C
\hookrightarrow \conodshuC$ is quasi-fully faithful. Every strong homotopy
counital $\Ainfty$ comodule is homotopy equivalent to its cobar
resolution, and hence to an $h$-injective DG $C$-comodule. 
We conclude that $\comodd_{hinj}\text{-}C \hookrightarrow \conodshuC$  
is a quasi-equivalence. In particular, our $D(C)$ as defined in
\S\ref{section-the-derived-category-of-ainfty-coalgebra} is
equivalent to the derived category of $C$ defined by Positselski
\cite[\S2.4]{Positselski-TwoKindsOfDerivedCategoriesKoszulDualityAndComoduleContramoduleCorrespondence}. In his terminology, it is the derived category of first kind. 

The analogue of a quasi-isomorphism of DG coalgebras in our generality is
an $\Ainfty$-morphism $f_\bullet\colon C \rightarrow D$ of $\Ainfty$-coalgebras 
whose first component $f_1$ is a homotopy equivalence in the monoidal
category $\A$ where $C$ and $D$ are defined. This notion is, indeed, 
too weak to ensure equivalence $D(C) \simeq D(D)$. Furthermore, just 
like Homotopy Lemma fails for $\Ainfty$-comodules, we expect that 
the analogue of \cite[Corollaire 1.3.1.3]{Lefevre-SurLesAInftyCategories}
fails for $\Ainfty$-coalgebras: $f_1$ being a homotopy equivalence
in $\A$ doesn't imply that $f_\bullet$ is a homotopy equivalence of 
$\Ainfty$-coalgebras.

However, as explained in \S\ref{section-coalgebras-with-perfect-free-comodules}, if $C$ and $D$ satisfy the condition that free comodules generated 
by perfect objects of $\A$ are perfect then the analogues of all
results in \S\ref{section-the-derived-category-the-strong-homotopy-unital-case}
hold for them. In particular, $D_c(C) \simeq D_c(D)$ similar to Corollary
\ref{cor-quasi-isomorphism-of-Ainfty-algebra-implies-equiv-of-derived-cats}.

\section{Module-comodule correspondence}
\label{section-module-comodule-correspondence}
 
In this section we show that if a bicomodule homotopy counital 
$\Ainfty$-coalgebra $C$ and a strongly homotopy unital $\Ainfty$-algebra $A$ 
are homotopy adjoint, their derived categories of modules and comodules,
respectively, are equivalent.

\subsection{Homotopy adjoint $\Ainfty$-algebras and coalgebras}
\label{section-homotopy-adjoint-Ainfty-algebras-and-coalgebras}

We propose the following definition of 
an$\Ainfty$-algebra and an $\Ainfty$-coalgebra being homotopy adjoint:

\begin{defn}
\label{defn-homotopy-adjoint-ainfty-coalgebra-and-algebra}
Let $(C,\Delta_\bullet)$ and $(A,m_\bullet)$ be an $\Ainfty$-coalgebra
and $\Ainfty$-algebra in a monoidal DG category $\A$. We say that 
$(C,\Delta_\bullet)$ and $(A,m_\bullet)$ are \em homotopy adjoint
as $\Ainfty$-coalgebra and $\Ainfty$-algebra \rm if there exist 
\begin{itemize}
\item A collection $\left\{ \eta_i\colon \id \rightarrow C^i A
\right\}_{i \geq 1}$ of degree $1 - i$ morphisms in $\A$, 
\item A collection $\left\{ \epsilon_i\colon A^iC \rightarrow \id \right\}_{i \geq 1}$ of degree $1 - i$ morphisms in $\A$,  
\end{itemize}
such that 
\begin{enumerate}
\item 
\label{item-homotopy-adjoint-ainfty-coalgebra-and-algebra-eta1-epsilon1}
$\eta_1\colon \id \rightarrow CA$ and $\epsilon_1 \colon AC
\rightarrow \id$ are a counit and a unit of a homotopy adjunction of 
$C$ and $A$ as objects of $\A$, 
\item 
\label{item-homotopy-adjoint-ainfty-coalgebra-and-algebra-eta-twisted-complex}
The following is a twisted complex over $\A$: 
\begin{small} 
\begin{equation}
\label{eqn-the-twisted-complex-for-eta-i}
\begin{tikzcd}[column sep=3cm]
\id
\ar{r}{\eta_1}
\ar[bend left=16]{rr}[description]{\eta_2}
\ar[bend left=17]{rrr}[description]{\eta_3}
&
CA
\ar{r}{C^2m_2\circ C\eta_1A -}[']{-\Delta_2A}
\ar[bend right=16,']{rr}{- C^3m_2\circ C\eta_2A - C^3m_3\circ C^2\eta_1A^2\circ C\eta_1A - \Delta_3A }
&
C^2A
\ar{r}{C^3m_2\circ C\eta_1A -}[']{-C\Delta_2A + \Delta_2CA}
&
\ldots
\end{tikzcd}
\end{equation}
\end{small} 
The differentials $\id \to C^iA$ are $\eta_i$
and those $C^{i-1}A \rightarrow C^{k+i-1}A$ are: 
$$
(-1)^{(i-1)(k+1)} 
\left(\sum_{j = 1}^{i-1} \left( (-1)^{jk} C^{i-j-1} \Delta_{k+1}C^{j-1}A\right) + C^{i-1} r_{k+1} \right)
$$
where 
$$
r_{k+1} = \sum\limits_{k_1+\ldots + k_n=k} 
C^k m_{n+1} \circ C^{k_1+\ldots + k_{n-1}}\eta_{k_n} A^{n} \circ \ldots \circ C^{k_1}\eta_{k_2}A \circ \eta_{k_1}A \  \colon A \to C^kA.
$$
\item
\label{item-homotopy-adjoint-ainfty-coalgebra-and-algebra-epsilon-twisted-complex}
 Morphisms $\eta_i$ fit into a similar twisted complex over $\A$
whose differentials $A^i C \rightarrow \id$ are $\epsilon_i$
and those $A^{k+i-1}C \rightarrow A^{i-1}C$ are:
$$
(-1)^{(i-1)(k+1)} 
\left(\sum_{j = 1}^{i-1} \left( (-1)^{jk} A^{i-j-1}
m_{k+1}A^{j-1}C\right) + A^{i-1} p_{k+1} \right)
$$
where 
$$
p_{k+1} =\sum\limits_{k_1+\ldots+k_n=k}
\epsilon_{k_n}C \circ A^{k_n}\epsilon_{k_{n-1}}C^2 \circ \ldots \circ
A^{i-k_1}\epsilon_{k_1}C^{n-1} \circ A^i \Delta_{n}
\colon A^k C \to C.
$$
\end{enumerate}
\end{defn}

To motivate this definition, consider the situation where $C$
and $A$ are a strict colagebra and a strict algebra. In other words
$ m_{\geq 3} = \Delta_{\geq 3} = 0$ and we write $\Delta$ 
for the comultiplication $\Delta_2$ of $C$ and $\mu$ for the
multiplication $m_2$ of $A$. 

Suppose $C$ and $A$ are genuinely adjoint as objects of $\A$ with 
unit $\eta\colon \id \rightarrow CA$ and counit $\epsilon\colon AC
\rightarrow \id$. Then the composition 
$q\colon A \xrightarrow{\eta{A}} CA^2 \xrightarrow{C \mu} CA$
gives $A$ the structure of (strict) left $C$-comodule if and only
if the following diagram commutes:
\begin{equation}
\label{eqn-homotopy-adjunction-algebra-coalgebra-strict-case-condition-1}
\begin{tikzcd}
A 
\ar{r}{q} 
&
CA
\ar{d}[']{C\eta{A}}
\ar{dr}{\Delta{A}}
&
\\
&  
C^2A^2
\ar{r}[']{C^2\mu}
&
C^2 A.
\end{tikzcd}
\end{equation}
If it does, associativity of $\mu$ implies that $A$ is a 
$C$-comodule-$A$-module: the coaction commutes with the action. 
Similarly, 
$p\colon AC \xrightarrow{A{\Delta}} AC^2 \xrightarrow{\epsilon{C}} C$
gives $C$ the structure of (strict) left $A$-module if and 
only if the following diagram commutes
\begin{equation}
\label{eqn-homotopy-adjunction-algebra-coalgebra-strict-case-condition-2}
\begin{tikzcd}
A^2 C
\ar{r}{A^2\Delta} 
\ar{dr}{\mu{C}}
&
A^2C^2
\ar{d}{A\epsilon{C}} 
&
&
\\
&
AC
\ar{r}{p}
&
C.
\end{tikzcd}
\end{equation}
If it does, then $C$ is an $A$-module-$C$-comodule. 

If $\A$ admits kernels and cokernels, 
we can define the functor of tensor product over
$A$ and the functor of cotensor product over $C$. Then
cotensoring with the $C$-comodule-$A$-module $A$ and 
tensoring with the $A$-module-$C$-comodule $C$ gives mutually inverse
equivalences $\comodd\text{-}C \leftrightarrows \modd\text{-}A$ 
of categories of strict strictly counital
$C$-comodules and strict strictly unital $A$-modules. 

This suggests that in the strict case to be adjoint as an algebra and 
a coalgebra, $A$ and $C$ have to be adjoint as objects of $\A$ and 
the unit and counit of the adjunction have to be compatible with the
multiplication of $A$ and the comultiplication of $C$ in the way
specified by the conditions
\eqref{eqn-homotopy-adjunction-algebra-coalgebra-strict-case-condition-1}
and \eqref{eqn-homotopy-adjunction-algebra-coalgebra-strict-case-condition-2}. 

Our Definition \ref{defn-homotopy-adjoint-ainfty-coalgebra-and-algebra} 
generalises this to non-strict case, where
$(C,\Delta_\bullet)$ and $(A,\mu_\bullet)$ are arbitrary 
$\Ainfty$-coalgebra and $\Ainfty$-algebra. It was more 
natural to work with extensions of the desired structures on 
$A$ and $C$ by $\id_\A$ viewed as the zero comodule-module or 
the zero module-comodule. Indeed, an $\Ainfty$-$C$-comodule-$A$-module 
structure on $\id \oplus A$ such that
\begin{itemize}
\item The projection $\id \oplus A \rightarrow \id$ is 
a morphism of $\Ainfty$-$C$-comodule-$A$-modules, where we take  
$\id$ with the zero comodule-module structure, 
\item the $\Ainfty$-$A$-module structure is given by
$A \oplus A^2 \xrightarrow{\left(\begin{smallmatrix} 0 & 0 \\ \id &
\mu_\bullet \end{smallmatrix}\right)} \id \oplus A$, 
\end{itemize}
is the same as a collection of degree $1-i$ maps 
$\eta_i \colon \id \rightarrow C^iA$ which fit into the
twisted complex in 
Defn.~\ref{defn-homotopy-adjoint-ainfty-coalgebra-and-algebra}\eqref{item-homotopy-adjoint-ainfty-coalgebra-and-algebra-eta-twisted-complex}. 
Then $\eta_2$ is the homotopy up to which the diagram 
\begin{equation}
\label{eqn-strict-case-condition-for-C-coaction}
\begin{tikzcd}
\id 
\ar{r}{\eta_1} 
&
CA
\ar{d}[']{C\eta_1{A}}
\ar{dr}{\Delta{A}}
&
\\
&  
C^2A^2
\ar{r}[']{C^2\mu}
&
C^2 A
\end{tikzcd} 
\end{equation}
commutes. It is a stronger version of 
the condition 
\eqref{eqn-homotopy-adjunction-algebra-coalgebra-strict-case-condition-1}.
Similarly, the analogous $\Ainfty$-$A$-module-$C$-comodule 
structure on $C$ is  
the same as a collection of degree $1-i$ maps $\epsilon_i \colon  A^iC
\rightarrow \id$ which fit into the twisted complex 
in Defn.~\ref{defn-homotopy-adjoint-ainfty-coalgebra-and-algebra}\eqref{item-homotopy-adjoint-ainfty-coalgebra-and-algebra-epsilon-twisted-complex}.

Even in the strict case we might still need higher $\eta_i$ and
$\epsilon_i$. For example, when $\eta_1$ and $\epsilon_1$, 
the unit and counit of a homotopy adjunction of $C$ and $A$
as the objects of $\A$, only make the diagrams
\eqref{eqn-homotopy-adjunction-algebra-coalgebra-strict-case-condition-1}
and 
\eqref{eqn-homotopy-adjunction-algebra-coalgebra-strict-case-condition-2}
commute up to homotopy. An example is given by any homotopy adjoint triple 
of objects of $\A$:
\begin{prps}
\label{prps-LF-and-RF-are-homotopy-adjoint-in-a-monoidal-category}
Let $\A$ be a monoidal DG category. Let $(L,F)$ and $(F,R)$ be
homotopy adjoint pairs in $\A$ with unit and counit 
$(\eta_l,\epsilon_l)$ and $(\eta_r, \epsilon_r)$. 
Define
$$ \Delta := LF \xrightarrow{L\eta_lF} LFLF, $$
$$ \mu := RFRF \xrightarrow{R\epsilon_rF} RF. $$

Then $(LF, \Delta)$ and $(RF,\mu)$ are a strict coalgebra and a strict 
algebra which are strongly homotopy (co)unital and 
homotopy adjoint as per 
Definition \ref{defn-homotopy-adjoint-ainfty-coalgebra-and-algebra}.
\end{prps}
\begin{proof}
The associativity of $\mu$ and the coassociativity of $\Delta$ follows from $\A$
being a monoidal category. 

Let us denote the homotopies in homotopy adjunctions of $(L, F, R)$ as follows:
\begin{align*}
F \to FRF \to F \quad &= \quad \id + dh^r_F
\quad \quad \quad 
R \to RFR \to R \quad &= \quad \id + dh^r_R\;\\
F \to FLF \to F \quad &= \quad \id + dh^l_F
\quad \quad \quad 
L \;\to LFL \;\to L \quad &= \quad \id + dh^l_L.
\end{align*}
Homotopy unitality of $(RF, \mu)$ 
means that the maps $\mu\circ RF \eta_r$  and $\mu\circ \eta_r RF$ are homotopic to identity.
Indeed, they have the form $\id + Rh^r_F$ and $\id+dh^r_R F$.
Moreover, the homotopy $Rh^r_F$ commutes with the left action of $RF$ 
and the homotopy $h^r_R F$ commutes with the right action of $RF$, 
thus they are morphisms of left and right modules respectively, 
which makes $(RF, \mu)$ strongly homotopy unital.
Strong homotopy counitality of $(LF, \Delta)$ is shown analogously.

Finally, the morphisms $\eta_i$ and $\epsilon_i$ from Definition \ref{defn-homotopy-adjoint-ainfty-coalgebra-and-algebra}
have the form:
\begin{align*}
\eta_i=RFLh_FLh_F\ldots Lh_FLF \circ R\eta_l\eta_l\ldots \eta_l F \circ \eta_r \colon\quad &
\id \to RFLFLF\ldots LF\\
\epsilon_i=\epsilon_l \circ L \epsilon_r\epsilon_r\ldots \epsilon_r F \circ LFRh_F^lRh^l_F....Rh_F^lRF \colon\quad &
LFRFRF\ldots RF \to \id.
\end{align*}

\end{proof}

\begin{remark}
Note that $(RF, \mu)$ and $(LF, \Delta)$ are not generally bimodule homotopy unital
(resp. bicomodule homotopy counital) with the above structure. For that, we would require e.g.
$Rh^r_F$ and $h^r_RF$ to be homotopic, which we don't have for an arbitrary
homotopy adjunction.
\end{remark}

\subsection{Derived module-comodule correspondence}
\label{section-derived-module-comodule-correspondence-for-homotopy-adjoint-algebras-and-coalgebras}

If one took time to set up the formalism of
$\Ainfty$-tensor/cotensor products 
and $\Ainfty$-modules-comodules,
the above should construct derived 
equivalence 
$D_c(C) \simeq D_c(A)$ as mutually inverse derived functors of 
cotensoring with $A \oplus \id$ and tensoring with $\id \oplus C$.
However, there is a simpler approach via the Kleisli and the co-Kleisli
categories:

\begin{theorem}[Module-Comodule Correspondence]
\label{theorem-module-comodule-correspondence-in-a-monoidal-dg-category}
Let $(A,m_\bullet, \eta, h^r_\bullet, h^l_\bullet)$ 
and $(C,\Delta_\bullet, \barepsilon_{\bullet\bullet})$
be a strong homotopy unital $\Ainfty$-algebra and a bicomodule
homotopy counital $\Ainfty$-coalgebra in a DG monoidal category $\A$. 

If $C$ is homotopy left adjoint to $A$ in the sense of 
Definition \ref{defn-homotopy-adjoint-ainfty-coalgebra-and-algebra}, 
then 
$$ D_c(C) \simeq D_c(A). $$ 
\end{theorem}
\begin{proof}
Let $S$ be a set of compact generators of $H^0(\A)$ in $H^0(\B)$. 
Since $A$ is strongly homotopy unital, by Theorem
\ref{theorem-compact-derived-category-of-A-is-that-of-frees-and-kleisli}
we have $ D_c(A) \simeq D_c(\kleisli_S(A))$. 
Since $A$ is a homotopy right adjoint of $C$ as an object of 
$\A$, by Lemma 
\ref{lemma-has-right-adjoint-implies-perfect-objects-generate-perfect-comodules}
free $\Ainfty$-$C$-comodules generated by perfect objects of $\A$ are
perfect. Thus by 
Theorem.~\ref{theorem-compact-derived-category-of-bimodule-homotopy-unital-C-is-that-of-frees-and-cokleisli} we have 
$D_c(C) \simeq D_c(\cokleisliCS)$. 

It therefore suffices to show that 
$D_c(\kleisli_S(A)) \simeq  D_c(\cokleisliCS)$. 
For this, we show that $\kleisli_S(A)$ and $\cokleisliCS$
are $\Ainfty$-quasi-equivalent. Indeed, 
the set of objects of both $\kleisli_S(A)$ 
and $\cokleisliCS$ is $\A$. Define an $\Ainfty$-functor
$$ E_\bullet\colon \cokleisliCS \rightarrow \kleisli_S(A) $$
to be the identity map on objects and on morphisms set
\begin{align*}
E_n\colon\; 
 \homm_{\cokleisliCS}(F_n, F_{n+1}) \otimes_k \dots \otimes_k
\homm_{\cokleisliCS}(F_1 , F_2)  
\; \longrightarrow \; \homm_{\kleisli_S(A)}(F_1, F_{n+1}) 
\end{align*}
to be the map which sends any 
$\alpha_n \otimes \dots \otimes \alpha_1$
with $\alpha_i \in \homm_{\A}(F_iC, F_{i+1})$
to the map 
$$ F_1 \xrightarrow{F_1 \eta_n} F_1C^nA \xrightarrow{\alpha_1 C^{n-1}A}
F_2C^{n-1}A \xrightarrow{\alpha_2 C^{n-2}A} \dots
\xrightarrow{\alpha_{n-1}A^2}
F_nCA \xrightarrow{\alpha_nA} F_{n+1}A. $$
The defining conditions on $\eta_i$ in
Definition \ref{defn-homotopy-adjoint-ainfty-coalgebra-and-algebra}
imply that $E_\bullet$ is an $\Ainfty$-functor. 

Now observe that 
$E_1$ is the map 
$\homm_{A}(F_1C, F_2) \xrightarrow{(-)A \circ \eta_1}  \homm_{A}(F_1,
F_2A)$. 
It is a quasi-isomorphism for all $F_1, F_2 \in \A$ since 
$\eta_1$ is the unit of a homotopy adjunction of $C$
and $A$. We conclude that $E_\bullet$ is an $\Ainfty$-quasi-equivalence. 
\end{proof}

\section{Examples and applications}
\label{section-examples-and-applications}

\subsection{Associative algebras}
\label{section-examples-associative-algebras}

Let $k$ be a field. 
Let $\A$ be the category $\vectk$ of vector spaces over $k$ with 
the monoidal operation $\otimes_k$. 
It is a $k$-linear monoidal category which we consider as 
a monoidal DG category concentrated in degree $0$. 

Since all morphisms in $\A$ are of degree $0$, an 
$\Ainfty$-algebra $(A,m_i)$ in $\A$ 
in the sense of our Defn.~\ref{defn-ainfty-algebra-in-a-monoidal-category} 
is a vector space $A$ and a single linear map $\mu\colon A \otimes_k A
\rightarrow A$. Its bar-construction is the usual bar-construction of $A$:
$$ \infbar(A) = \quad \quad 
\dots \rightarrow A^4 \xrightarrow{A^2 \mu - A \mu A +
\mu A^2 } A^3 \xrightarrow{A \mu - \mu A} A^2 \xrightarrow{\mu} A. $$
The twisted complex condition in
Defn.~\ref{defn-ainfty-algebra-in-a-monoidal-category} asks that this 
is a complex, i.e. $d^2 = 0$. 
By Prop.~\ref{prps-comparing-new-and-old-defns-of-ainfty-algebra} 
we can only check it for the first two differentialls:
$\mu \circ (A\mu - \mu{A}) = 0$.
Thus our $\Ainfty$-algebras in $\vectk$ are non-unital associative
$k$-algebras $(A,\mu)$. All the homotopy unitality conditions given in 
\S\ref{section-strong-homotopy-unitality} are then equivalent to the strict
unitality of $A$ in the usual sense. 

Similarly, an $\Ainfty$-module over $A$ in $\vectk$ in our sense is 
a non-unital $A$-module $(E,\pi)$ in the usual sense and homotopy
unital modules are the strictly unital ones. Our bar-construction 
again reduces to the usual one:
\begin{equation}
\infbar(E,\pi) =
\quad \quad 
\dots 
\rightarrow 
EA^3 
\xrightarrow{EA\mu - E\mu{A} + \pi A^2}
EA^2
\xrightarrow{E\mu - \pi{A}}
EA
\xrightarrow{\pi}
E, 
\end{equation}
Morphisms in $\nodA$ are $\Ainfty$-morphisms in the usual sense. 
A degree $j$ morphism $(E,\pi) \rightarrow (F,\kappa)$ in $\nodA$ 
is a linear map $f: E \otimes A^{j} \rightarrow F$. 
Its bar-construction $\infbar(f)$ has components $(-1)^{j(i-1)} fA^{i-1}$. 
We compose and differentiate morphisms in $\nodA$ by 
composing and differentiating their bar-constructions. 

Thus $\nodhuA$ is the DG category of unital $A$-modules in the usual sense 
with $\Ainfty$-morphisms. 
Let $\modd\text{-}A$ be the category of unital $A$-modules with 
ordinary morphisms. We have an isomorphism of complexes of vector spaces
$$ \homm_{\nodhuA}(E,F) \simeq 
\homm_{\cx(\modd\text{-}A)}(\bar{B}^A_\infty(E), F) $$
where $\bar{B}^A_\infty(E)$ is the bar-resolution of $E$ in $\modd\text{-}A$:
$$ \bar{B}^A_\infty(E,\pi)\colon 
\rightarrow 
EA^3 
\xrightarrow{- EA\mu + E\mu{A} - \pi A^2}
EA^2
\xrightarrow{- E\mu + \pi{A}}
EA. 
$$
The isomorphism sends the morphism given by a linear
map $E \otimes_k A^{j} \rightarrow F$ to the $A$-module morphism 
$E \otimes_k A^{j+1} \rightarrow F$ which corresponds to it 
under the Free-Forgetful adjunction. The adjunction ensures 
bijectivity, but we need to check that this commutes with the
differentials. Here the contribution to the differential of 
$\homm_{\nodhuA}(E,F)$ made by the composition through $FA$ 
accounts for the extra term present 
in the differentials of $\bar{B}^A_\infty(E)$ as compared 
to $\infbar(E)$. Since the cohomologies of 
$\homm_{\cx(\modd\text{-}A)}(\bar{B}^A_\infty(E), F)$ are 
$\ext^i_{\modd\text{-}A}(E,F)$, it follows that 
$H^\bullet(\nodhuA)$ is the graded $\ext$-category 
of $\modd\text{-}A$. 

The $A$-modules free in the sense of our
\S\ref{section-free-modules-and-bimodules-over-Ainfty-algebra}
are those free in the usual sense. 
Thus $\freeA$ is the DG category of free $A$-modules with
$\Ainfty$-morphisms. The Kleisli category $\kleisliA$ is the category of
vector spaces with morphism spaces given by 
$\homm_{\vectk}(E, F \otimes_k A)$. By classical Free-Forgetful 
adjunction, it isomorphic to the category of free 
$A$-modules with ordinary $A$-module morphisms. 
Under this identification, the functor 
$\kleisliA \rightarrow \freeA$ of Theorem 
\ref{theorem-ainfty-quasi-equivalence-from-kleisli-to-free}
is the DG functor which sends any ordinary $A$-module morphism to
itself considered as a strict $\Ainfty$-morphism. It can be directly 
seen to be a quasi-equivalence: $H^\bullet(\nodhuA)$ is 
the graded $\ext$-category of $\modd\text{-}A$, 
and higher $\ext$s between free modules vanish. 

Category $\vectk$ is not small, so we can't set $\B = \modA$. 
Even if we used the category of finite-dimensional vector spaces $\vectkfg$ 
instead, $\modd\text{-}\vectkfg$ is needlessly large. 
We set $\B = \modk$, the category of DG $k$-modules. 
It contains $\vectk$ as the full subcategory of DG modules concentrated 
in degree $0$. It is also cocomplete, closed monoidal, and admits convolutions 
of twisted complexes. Thus it satisfies all the assumptions
in \S\ref{section-the-setting} and \S\ref{section-the-derived-category}. 

The DG category $\noddinf\text{-}A^{\modk}$ in the sense of our 
\S\ref{section-ainfty-structures-in-monoidal-dg-categories}
is the DG category of non-unital $\Ainfty$-$A$-modules in 
\cite[\S7]{Lefevre-SurLesAInftyCategories}. The derived
category $D(A)$ of $A$ in the sense of our
\S\ref{section-the-derived-category} is the cocomplete
triangulated hull of $\nodhuA$ in 
$H^0(\noddinf\text{-}A^{\modk})$. Since the co-complete triangulated
hull of $\vectk$ in $\modk$ is the whole of $\modk$, by
Prop.~\ref{prps-derived-category-of-A-is-triangulated-cocomplete-hull-of-freeA}
the derived category $D(A)$ is $H^0(\noddinfhu\text{-}A^{\modk})$. 

The whole of $H^0(\modk)$ is generated by a single compact object $k$. Thus 
we take the set $S$ of compact generators of $\vectk$ to be 
$\left\{ k \right\}$. Then $\freeAS$ is the full subcategory of $\nodA$
consisting of the free module $A$. The corresponding
Kleisli category $\kleisliAS$ has a single
object with the endomorphism space $\homm_{\vectk}(k,A) = A$, 
i.e. the algebra $A$ considered as a single-object category. 
By
Prop.~\ref{prps-perfect-iff-lies-in-hperf-of-perfect-generator-frees},
$D(A)$ in the sense of our \S\ref{section-the-derived-category} 
is the cocomplete triangulated hull of $\freeAS = \left\{ A \right\}$ in 
$H^0(\noddinf\text{-}A^{\modk})$. Thus our definition of $D(A)$
coincides with the definition of the derived category of an $\Ainfty$-algebra 
in \cite[\S4.1.2]{Lefevre-SurLesAInftyCategories}. 

The category $\nodstrA$
in the sense of \S\ref{section-the-derived-category-strict-algebra-case}
is the usual category of non-unital $A$-modules with $A$-module morphisms. 
Its full subcategory $\modd\text{-}A$ comprises unital $A$-modules.  
The category $\modd\text{-}A^{\modk}$ is the
DG category of complexes of unital $A$-modules and 
$H^0(\modd\text{-}A^{\modk})$ is the usual homotopy category of 
complexes of unital $A$-modules. The proof of Theorem 
\ref{theorem-derived-category-as-localisation-in-the-strict-case}
shows that $D(A)$ in the sense of  our \S\ref{section-the-derived-category} 
is equivalent to the Verdier quotient
of $H^0(\modd\text{-}A^{\modk})$ by the acyclics. 
Thus $D(A)$ is the localisation of $K(A)$, 
the homotopy category of complexes of $A$-modules, 
by quasi-isomorphisms, the classical definition 
of the derived category. 

Finally, we can take $k$ to be a commutative
ring and $\A$ the category of unital $k$-modules. This
is not covered by \cite{Lefevre-SurLesAInftyCategories}
which uses minimal models, a technique which works
only over a field, to prove its key results e.g. Lemme 4.1.3.7. 
Our theory works happily with commutative rings. 
Its output, however, is different. 

Firstly, our free $A$-modules $E \otimes_k A$ are not free in 
the usual sense, unless $E$ is a free $k$-module. 
Thus $\bar{B}^A_\infty(E)$ is no longer a projective resolution of $E$.
While $H^0(\nodhuA)$ is still $\modd\text{-}A$, it is no longer
true that $H^\bullet(\nodhuA)$ is the graded $\ext$-category of
$\modd\text{-}A$. 
More generally, $D(A)$ in the sense 
our \S\ref{section-the-derived-category} is no longer the classical
derived category of $A$. Instead, it is the
localisation of $K(A)$ by those 
maps which are homotopy equivalences as maps of complexes of $k$-modules.

This is not an oversight, but the intention. Homotopy equivalences
exist in any DG category, while quasi-isomorphisms only exist in 
$\A = \modk$ and its like. One of the morals of this paper is 
that working with $\Ainfty$-morphisms reduces homotopy equivalences of 
$A$-modules to the homotopy equivalences in the base category $\A$. 
Thus when $k$ is a field and $\A = \modk$, it reduces them to quasi-isomorphisms. 

To get the right derived category, $\A$ must have
the right class of morphisms as homotopy equivalences. If $A$ is 
an associative algebra over a commutative ring $k$, we can set $\A$ to be 
the monoidal category of projective $k$-modules and $\B$ to be 
the monoidal DG category of $h$-projective complexes of $k$-modules. 
Then our $D(A)$ is the usual localisation of $K(A)$ by quasi-isomorphisms. 

\subsection{DG and $\Ainfty$-algebras}
\label{section-examples-dg-and-ainfty-algebras}

Let $k$ be a commutative ring. Let $\A = \modk$, the category
of DG modules over $k$ with the monoidal operation $\otimes_k$ and
unit $k$.

The $\Ainfty$-algebras in $\A$ in the sense
of our \S\ref{section-ainfty-structures-in-monoidal-dg-categories}
are the usual ones in \cite{Keller-IntroductionToAInfinityAlgebrasAndModules}, 
\cite{Lefevre-SurLesAInftyCategories}. 
If $A$ is an $\Ainfty$-algebra in $\modk$, then our $\nodA$ is 
the DG category of non-unital $\Ainfty$-modules in 
\cite[\S5.1]{Lefevre-SurLesAInftyCategories}. 
Strict and strictly unital $\Ainfty$-algebras
are the usual DG algebras. If $A$ is a strict algebra in $\modk$ then our 
$\modd\text{-}A$ is the usual DG category of DG-modules over $A$, see 
\cite[\S2]{AnnoLogvinenko-SphericalDGFunctors},
\cite{Keller-DerivingDGCategories}, or
\cite{Toen-LecturesOnDGCategories}. Our $\nodstrA$ is the larger DG
category of non-unital DG $A$-modules.  

Let $A$ be a strongly homotopy unital $\Ainfty$-algebra in $\modk$. 
The free $A$-modules in the sense of our
\S\ref{section-free-modules-and-bimodules-over-Ainfty-algebra}
are the modules $E \otimes_k A$ with $E \in \modk$. When $k$ is a
field, such $A$-modules are semifree \cite[Lemma 2.16]{AnnoLogvinenko-BarCategoryOfModulesAndHomotopyAdjunctionForTensorFunctors},
and thus $h$-projective. Otherwise, 
they only have the $h$-projectivity property 
with respect to the $A$-modules null-homotopic as $k$-modules, as
opposed to all the ones with zero cohomology.  
$\freeA$ is the DG category of free modules $E \otimes_k A$
with $\Ainfty$-morphism complexes. The Kleisli category
$\kleisliA$ is the $\Ainfty$-category with the same objects as 
$\modk$ but $\homm$-complexes $\homm^\bullet_{\modk}(E, F \otimes_k A)$
between $E,F \in \modk$. 

The category $\modk$ satisfies our assumptions 
in \S\ref{section-the-setting} and
\S\ref{section-the-derived-category}. 
Let $\B = \modk$ and choose $S:= \left\{ k \right\}$
as the set of compact generators of $H^0(\modk)$. 
Then the category $\kleisliAS$ is the $\Ainfty$-algebra $A$ 
viewed as a single-object $\Ainfty$-category, while
$\freeAS$ is a single-object DG category consisting of 
the free module $A$ and its
$\Ainfty$-endomorphisms. The $\Ainfty$-functor 
$f_\bullet\colon \kleisliAS \rightarrow \freeAS$ of Theorem 
\ref{theorem-ainfty-quasi-equivalence-from-kleisli-to-free}
sends the single object $k$ 
of $\kleisliA$ to the free module $A$ and its component 
$$ f_n\colon \homm_{\kleisliA}(k, k)^{\otimes_k n} = 
A^{\otimes_k n} \rightarrow \homm_{\nodA}(A,A) $$
sends any $a_1 \otimes \dots \otimes a_n \in A^{\otimes_k n}$
to 
$\alpha_\bullet$ whose $\alpha_m\colon A^m \rightarrow A$ is the map
$$ b_1 \otimes \dots \otimes b_m \rightarrow 
m_{n+m}(a_1 \otimes \dots \otimes a_n \otimes b_1 \otimes \dots
\otimes b_m). $$
It is a special case of the $\Ainfty$-Yoneda embedding 
constructed in \cite[\S7.1]{Lefevre-SurLesAInftyCategories}. 

Since $\A = \B = \modk$, 
by Prop.~\ref{prps-when-B-is-modA-we-get-the-classical-derived-category}
our derived category $D(A)$ is $H^0(\nodhuA)$. It thus
agrees with the usual definition of $D(A)$ 
in \cite[\S4.1.2]{Lefevre-SurLesAInftyCategories}. 
By Theorem 
\ref{theorem-compact-derived-category-of-A-is-that-of-frees-and-kleisli}
we always have $D_c(A) \simeq D_c(\kleisliAS)$, but
here we also have $D(A) \simeq D(\kleisliAS)$
since $\kleisliAS$ is $A$ viewed as the usual $\Ainfty$-category. 
This is because $\B = \modk$, so we 
take infinite direct sums in $\modk$ which coincides with the
classical theory.  

A strict and strongly homotopy unital $A$ is a strongly homotopy 
unital DG algebra. $\nodstrhuA$ is the category 
of homotopy unital DG-modules over $A$. 
By Theorem \ref{theorem-derived-category-as-localisation-in-the-strict-case}, 
$D(A)$ is the localisation of
$H^0(\nodstrhuA)$ by the $\A$-homotopy equivalences, i.e.
$\Ainfty$-morphisms whose first component is a homotopy equivalence 
in $\modk$. When $k$ is a field these are the same as quasi-isomorphisms. 

A strict and strictly unital $A$ is the usual (strictly unital) DG algebra. 
By the proof of Theorem
\ref{theorem-derived-category-as-localisation-in-the-strict-case},
$D(A)$ is the localisation of 
$H^0(\modd\text{-}A)$ by $\A$-homotopy equivalences. When $k$ is a field, 
these are quasi-isomorphisms and the localisation is the classical 
derived category of $A$ as a DG algebra. We get another 
proof that for DG algebras the $\Ainfty$ derived category
coincides with the DG derived category. 

When $k$ is not a field, we see that $\Ainfty$ and DG derived
categories of $A$ are different. The former is localisation of 
$H^0(\modd\text{-}A)$ by $\A$-homotopy equivalences, the latter by
quasi-isomorphisms. We can fix this by changing our monoidal DG category 
$\A$ from $\modk$ to its monoidal subcategory comprising $h$-projective 
modules. Then $\A$-homotopy equivalences would be the same as
quasi-isomorphisms. 

\subsection{DG and $\Ainfty$-categories}
\label{section-examples-dg-and-ainfty-categories}

Usual DG categories \cite[\S2]{AnnoLogvinenko-SphericalDGFunctors}
\cite{Keller-DerivingDGCategories}
\cite{Toen-LecturesOnDGCategories} and $\Ainfty$-categories \cite{Keller-IntroductionToAInfinityAlgebrasAndModules} 
\cite{Lefevre-SurLesAInftyCategories} fit similarly into our framework,
however one needs to use a bicategorical setup. 
See \cite{Benabou-IntroductionToBicategories} for 
an introduction to bicategories. 

Let $k$ be a commutative ring. For any set $U$ define the category $k_U$
by $\obj k_U = U$ 
and $\homm_{k_U}(s,t) = k$ if $s = t$ and 
$0$ if $s \neq t$. Let $\A$ be the DG bicategory $\bikmodk$:
\begin{itemize}
\item The objects are categories $k_U$ for each small set $U$. 
\item The $1$-morphism DG categories are defined by
$$ \homm_{\bikmodk}(k_U, k_T) = k_U\text{-}\modd\text{-}k_T \quad
\quad \forall\; \text{ small sets } U,T.$$ 
\item $1$-composition is given by $\otimes_{k}$. 
\item For any small set $U$ the identity $1$-morphism is the diagonal
bimodule $k_U$.
\item Associator and unitor isomorphisms are the maps induced by the universal 
property of tensor product.  
\end{itemize}
When $U$ is the singleton set, we write $k$ for $k_U$. 

 Our $\Ainfty$-algebra $A$ in some 
$k_U\text{-}\modd\text{-}k_U$ is an $\Ainfty$-category   
\cite[\S5.1]{Lefevre-SurLesAInftyCategories} whose object set is $U$. 
Let $\C$ be this $\Ainfty$-category.  
Its DG category of right $\Ainfty$-$\C$-modules
\cite[\S5.2]{Lefevre-SurLesAInftyCategories} 
is our category $\nodd^{k,k_U}_\infty\text{-}A$ of right 
$\Ainfty$-$A$-modules in $k\text{-}\modd\text{-}k_U$. 
Note that we can consider the category 
$\nodd^{k_T,k_U}_\infty\text{-}A$ of right $\Ainfty$-$A$-modules in any 
$k_T\text{-}\modd\text{-}k_U$, but the result is just a Cartesian 
product of $|T|$ copies of $\nodd^{k,k_U}_\infty\text{-}A$. 
We thus write $\nodA$ for $\nodd^{k,k_U}_\infty\text{-}A$,
and similarly for left modules and bimodules. 
Finally, strict and strictly unital algebras, modules, and bimodules 
in $\bikmodk$ correspond to the usual DG categories, DG modules and bimodules. 

The $1$-morphism DG categories of $\bikmodk$ satisfy our assumptions 
in \S\ref{section-the-setting} and \S\ref{section-the-derived-category}. 
Each $H^0(k_U\text{-}\modd\text{-}k_T)$ is compactly 
generated by representable bimodules $k^{u,t} =
\homm_{k_U\text{-}\modd\text{-}k_T}(-, (u,t))$ whose only fiber 
is $k$ over $(u,t)$. We can thus set $\A = \B$ and take 
the set $S$ of the compact generators of $\modd\text{-}k_U$ 
to be $\left\{ k^u \;\middle|\; u \in U \right\}$. 

The Kleisli category $\kleisliAS$ is  
the $\Ainfty$-category $\C$ which corresponds to $A$. 
The category $\freeAS$ is the full subcategory of $\nodA$
comprising the modules corresponding to representable 
$\Ainfty$-$\C$-modules. The functor 
$f_\bullet\colon \kleisliAS \rightarrow \freeAS$ of Theorem 
\ref{theorem-ainfty-quasi-equivalence-from-kleisli-to-free}
is the $\Ainfty$-Yoneda embedding of $\C$ constructed in 
\cite[\S7.1]{Lefevre-SurLesAInftyCategories}.

The rest of our theory applies as in 
\S\ref{section-examples-dg-and-ainfty-algebras}. By 
Prop.~\ref{prps-derived-category-of-A-is-triangulated-cocomplete-hull-of-freeA}
for $H$-unital $A$ our derived category $D(A)$ coincides 
with the $\Ainfty$-derived category of $\C$  
\cite[\S4.1.2]{Lefevre-SurLesAInftyCategories}. If $k$ is a field
and $A$ is strict and strictly unital, $D(A)$ also 
coincides with the DG derived category of the DG category $\C$. 
When $k$ is not a field,  $D_{dg}(\C)$ is the localisation of 
$\modC$ by quasi-isomorphisms, while $D(A)$ is its localisation by 
$\modk$-homotopy equivalences.  
This can be fixed by switching from the DG bicategory $\bikmodk$ to its
$2$-full subcategory comprising $h$-projective bimodules. 

\subsection{Weak Eilenberg-Moore category and $\Ainfty$-monads}
\label{section-classical-and-ainfty-monads}

In
\S\ref{section-examples-associative-algebras}-\ref{section-examples-dg-and-ainfty-categories}
we gave examples of how known constructions 
fit into our theory in 
\S\ref{section-ainfty-structures-in-monoidal-dg-categories}-\S\ref{section-the-derived-category}.
From now on, we apply our theory to 
obtain the constructions that are, to our knowledge, new. 

Let $\A$ be the strict DG $2$-category $\DGFuntwocat$ of small DG categories, 
DG functors, and DG natural transformations. Let $\C$ be a small DG category. 
A strictly unital algebra $(T,\mu,\eta)$ in 
$\DGFun(\C,\C)$ is a DG monad $T: \C \rightarrow \C$ with multiplication 
$\mu\colon T^2 \rightarrow T$ and unit $\eta\colon \id_\C \rightarrow
T$, see \cite[\S{VI.1}]{MacLane-CategoriesfortheWorkingMathematician}. 
The category $T\text{-}\modd$ in the sense of our
Defn.~\ref{defn-strict-algebra-strict-module-categories}
is the \em Eilenberg-Moore \rm category $\eilmoor_T$ of $T$ 
\cite[\S{VI.2}]{MacLane-CategoriesfortheWorkingMathematician}. 
Its objects are $T$-modules: pairs $(a,p)$ where $a \in \C$
and $p$ is the structure morphism  $Ta \rightarrow a$ which satisfy 
strict associativity and unitality conditions.  Its morphisms $(a,p)
\rightarrow (b,q)$ are maps $f\colon a \rightarrow b$ in $\C$ which
intertwine $p$ and $q$. 

It has long been desired to define the weak Eilenberg-Moore
category $\eilmoorwk_T$ of a DG monad $T$. Naively, one could try 
to define the objects of $\eilmoorwk_T$ to be pairs $(a,p)$ for which 
associativity and unitality only hold up to homotopy and its
morphisms $(a,p) \rightarrow (b,q)$ to be maps 
$f\colon a \rightarrow b$ which intertwine $p$ and $q$ up to homotopy. 
This fails for the usual reasons – we can't define the composition of 
morphisms. Given three modules $(a,p), (b,q),$ and $(c,r)$ and 
two maps $f\colon a \rightarrow b$ and $g\colon b \rightarrow c$ 
which intertwine $p,q,$ and $r$ up to homotopy, their composition $gf$
doesn't, apriori, intertwine $p,r$. Indeed, if $fp - q(Tf) = dh$ and
$gq - r(Tg) = dj$, then $gfp - r(Tg)(Tf) = g(dh) - (dj)(Tf)$, and this
is not, apriori, a boundary. 

Conceptually, the fix is clear. One shouldn't
consider just the first homotopy $h$ up to which $f$ intertwines
$p,q$ but the whole $\Ainfty$-structure of homotopies. 
However, the usual definition of $\Ainfty$-structures 
\cite[\S4.1.2]{Lefevre-SurLesAInftyCategories} requires the underlying
DG category to be $\modk$, while here it is 
$\C$ and $\DGFun(\C,\C)$. It was this that led the authors to develop
the theory presented in this paper. 
With it, we define as per
\S\ref{section-ainfty-structures-in-monoidal-dg-categories}
the DG category $T\text{-}\noddinf$ of left $\Ainfty$-$T$-modules in $\C$ and
their $\Ainfty$-morphisms. As $T$ is strictly
unital, the notions of $H$-unitality, homotopy unitality, 
and strong homotopy unitality for $\Ainfty$-$T$-modules 
are equivalent as per \S\ref{section-strong-homotopy-unitality}. 
We then define: 
\begin{defn}
Let $T$ be a monad in a DG category $\C$. The 
\em weak Eilenberg-Moore category $\eilmoorwk_T$ \rm is  
DG category $T\text{-}\noddinfhu$ of homotopy unital $\Ainfty$-$T$-modules 
in $\C$. 
\end{defn}

Explicitly, an $\Ainfty$-$T$-module $(a,p_\bullet)$ is an object
$a \in \C$ and a collection of degree $2-i$ 
morphisms $p_i\colon T^{i-1} a \rightarrow a$ such that the maps 
$d_{i(i-k)}\colon T^i a \rightarrow T^{i-k} a$
given by
$
d_{i(i-1)} = \left( \sum_{j =
1}^{i-1} (-1)^{i-j} T^{j-1}\mu T^{i-j-1} \right)
 + T^{i-1} p_2 
$
and $d_{i(i-k)} = (-1)^{(i-k)(k+1)} T^{i-k} p_{k+1}$
for $k > 1$ give a twisted complex over $\C$
which is the \em bar-construction \rm $\infbar(a,p_\bullet)$:
\begin{equation}
\label{eqn-bar-construction-of-Ainfty-T-module}
\begin{tikzcd}[column sep = 2.0cm]
\dots
\ar{r}{}
\ar[dashed, bend left=25]{rrrr}[description]{p_5}
\ar[dashed, bend left=20]{rrr}[description]{Tp_4}
\ar[dashed, bend left=15]{rr}[description]{T^2 p_3}
& 
T^3a
\ar{r}[']{\mu T - T\mu + T^2 p_2}
\ar[dashed, bend left=20]{rrr}[description]{p_4}
\ar[dashed, bend left=15]{rr}[description]{(- T p_3)}
& 
T^2 a
\ar[dashed, bend left=15]{rr}[description]{p_3}
\ar{r}[']{-\mu + T p_2}
&
Ta
\ar{r}[']{p_2}
&
\underset{\degzero}{a}. 
\end{tikzcd}
\end{equation}

A degree $n$ morphism $(a,p_\bullet) \rightarrow (b,q_\bullet)$ is 
a collection $f_\bullet$ of degree $1-i+n$ morphisms $f_i\colon
T^{i-1}a \rightarrow b$ in $\C$. These are composed and differentiated
by the means of their bar-constructions. The bar-construction 
$\infbar(f_\bullet)$ is the twisted complex morphism 
$\infbar(a,p_\bullet) \rightarrow \infbar(b,q_\bullet)$ whose
components are the maps
$$ T^{i+k} a \rightarrow T^{i} b \quad \quad \text{ given by } \quad 
(-1)^{i(k+\deg_{f_\bullet})} T^{i} f_{k+1} \quad \quad k \geq 0. $$
We illustrate the case when $f_\bullet$ is of odd degree:
\begin{equation}
\label{eqn-bar-construction-of-ainfty-morphism-of-modules-over-monad}
\begin{tikzcd}[row sep=1.5cm, column sep = 2.0cm]
\dots
\ar{r}{}
& 
T^3a
\ar{r}
\ar{d}[description, pos=0.6]{-T^3f_1}
\ar{dr}[description, near start]{T^2f_2}
\ar{drr}[description, near start]{-Tf_3}
\ar{drrr}[description, near start]{f_4}
& 
T^2 a
\ar{r}
\ar{d}[description, pos=0.6]{T^2f_1}
\ar{dr}[description, near start]{ Tf_2}
\ar{drr}[description, near start]{f_3}
&
Ta
\ar{r}
\ar{d}[description, pos=0.6]{-Tf_1}
\ar{dr}[description, near start]{f_2}
&
a 
\ar{d}[description, pos=0.6]{f_1}
\\
\ar{r}{}
& 
T^3b
\ar{r}
& 
T^2 b
\ar{r}
&
Tb
\ar{r}
&
\underset{\degzero}{b}.
\end{tikzcd}
\end{equation} 

The Kleisli category $\kleisli(T)$ in the sense of our
\S\ref{section-kleisli-category} is the usual Kleisli category of the 
monad $(T,\mu, \eta)$ \cite{Kleisli-EveryStandardConstructionIsInducedByAPairOfAdjointFunctors}\cite[\S{VI.5}]{MacLane-CategoriesfortheWorkingMathematician}. It is the DG category whose objects are those of 
$\C$ and whose morphism complexes are defined as
$$ \homm^\bullet_{\kleisli(T)}(a,b) = \homm^\bullet_\C (a,Tb). $$
The composition of $f\colon a \rightarrow Tb$ and 
$g\colon b \rightarrow Tc$ in $\kleisli(T)$ is 
$ a \xrightarrow{f} Tb \xrightarrow{Tg} T^2c \xrightarrow{{\mu}c} Tc$.
The identity $1_a$ in $\kleisli(T)$ is the morphism
corresponding to $ a \xrightarrow{{\eta}a} Ta$ in $\C$. 

There is a natural non-full inclusion 
$\eilmoor_T \hookrightarrow \eilmoorwk_T$ whose image
comprises strict modules and their strict morphisms with strict 
differentials. The classical Free-Forgetful adjunction 
$\C \leftrightarrows \eilmoor_T$ extends to the homotopy adjunction 
$\C \leftrightarrows \eilmoorwk_T$ of 
\S\ref{section-free-forgetful-homotopy-adjunction}. The classical
adjunction yields a fully faithful embedding
$\kleisli(T) \hookrightarrow \eilmoor_T$ whose image is 
the full subcategory comprising free modules. 
Its composition with $\eilmoor_T \hookrightarrow \eilmoorwk_T$
is the quasi-fully faithful functor 
$\kleisli(T) \rightarrow \eilmoorwk_T$ of Theorem 
\ref{theorem-ainfty-quasi-equivalence-from-kleisli-to-free}.

More generally, by considering $\Ainfty$-algebras in $\DGFuntwocat$ 
as per \S\ref{section-ainfty-structures-in-monoidal-dg-categories}
we get a new notion of $\Ainfty$-monads over DG categories:
\begin{defn}
\label{defn-ainfty-monad-over-dg-category}
Let $\C$ be a DG category. An \em $\Ainfty$-monad $(T,m_\bullet)$ \rm is 
\begin{itemize}
\item a DG functor $T: \C \rightarrow \C$, 
\item a collection of degree $2-i$ natural transformations 
$m_i\colon T^i \rightarrow T$ for $i \geq 2$,
\end{itemize}
whose bar-construction $\infbarnaug(T)$
(see Defn.~\ref{eqn-nonaugmented-bar-construction-of-A-m_i}) is
a twisted complex over $\DGFun(\C,\C)$. 
\end{defn}
By Prop.~\ref{prps-defining-equalities-of-Ainfty-morphism}
the twisted complex condition in 
Defn.~\ref{defn-ainfty-monad-over-dg-category}
is equivalent to:
\begin{equation}
d_{\DGFun(\C,\C)} m_i + \sum_{\begin{smallmatrix}j+k+l = i, \\ k \geq
2\end{smallmatrix}} (-1)^{jk+l} m_{j+1+l} \circ \left(\id^j \otimes
m_k \otimes \id^l\right) = 0. 
\end{equation}
As explained in \S\ref{section-Ainfty-algebras-in-a-monoidal-category}, 
these are $m_1$-free analogues of the defining equations for the usual 
$\Ainfty$-algebras \cite[Defn.~1.2.1.2]{Lefevre-SurLesAInftyCategories}. 

\begin{defn}
\label{defn-ainfty-module-over-ainfty-monad}
Let $\C$ be a DG category and let $(T,m_\bullet): \C \rightarrow \C$
be an $\Ainfty$-monad.
An \em $\Ainfty$-module $(a,p_\bullet)$ \rm over $(T,m_\bullet)$ is:
\begin{itemize}
\item an object $a \in \C$,
\item a collection of degree $2-i$ morphisms 
$m_{i} \colon T^{i-1} a \rightarrow a$ for $i \geq 2$,
\end{itemize}
whose bar-construction $\infbarnaug(a,p_\bullet)$
(see Defn.~\ref{defn-left-module-bar-construction-in-a-monoidal-category}) is
a twisted complex over $\C$. 

An \em $\Ainfty$-morphism 
$f_\bullet\colon (a,p_\bullet) \rightarrow (b,q_\bullet)$  
of $\Ainfty$-$T$-modules \rm is a collection of degree $1-i$ 
morphisms $T^{i-1}a \rightarrow b$ in $\C$. The composition and the
differential of such morphisms are defined by the means of their
bar-constructions, see
Defn.~\ref{defn-right-module-bar-constructions-of-an-Ainfty-morphism}-\ref{defn-left-module-bar-constructions-of-an-Ainfty-morphism}.
We denote the resulting DG category of $\Ainfty$-modules over $T$ by
$T\text{-}\noddinf$. 
\end{defn}

To consider $\Ainfty$-monads over $\C$ as enhancements of genuine
monads on $H^0(\C)$, one needs some notion of homotopy unitality. We
can use the \em strong homotopy unitality \rm of
Defn.~\ref{defn-strong-homotopy-unitality-for-ainfty-algebras}. 
By Theorem \ref{theorem-tfae-unitality-conditions-for-A-modules},
if an $\Ainfty$-monad $T$ is strongly homotopy unital then 
for $\Ainfty$-$T$-modules it suffices to use weak homotopy unitality: 
$(a,p_\bullet)$ is homotopy unital if $(a,p_2)$ is a strictly unital 
$H^0(T)$-module. 
We thus define:
\begin{defn}
Let $\C$ be a DG category and let $(T,m_\bullet): \C \rightarrow \C$
be an strongly homotopy unital $\Ainfty$-monad. The \em weak Eilenberg-Moore 
category $\eilmoorwk_T$ \rm is the DG category $T\text{-}\noddinfhu$
of homotopy unital $\Ainfty$-$T$-modules in $\C$. 
\end{defn}

One application of this new notion is
to fix the well-known problem of the (strong) Eilenberg-Moore category 
of an exact monad over a triangulated category not being 
necessary triangulated:
\begin{theorem}
Let $C$ be a triangulated category and let $t: C \rightarrow C$ be 
an exact monad. Let $\C$ be a pretriangulated DG category such that
$H^0(\C) \simeq C$ and let $T\colon \C \rightarrow \C$ be a strong
homotopy unital $\Ainfty$-monad such that $H^0(T) = t$. In other
words, let $\C$ and $T$ be DG/$\Ainfty$-enhancements of $C$ and $t$.  

The homotopy category $H^0(\eilmoorwk_T)$ of the weak
Eilenberg-Moore category of $T$ is a triangulated category
and we have a natural functor 
$$ F\colon H^0(\eilmoorwk_T) \rightarrow \eilmoor_t $$ 
given by $(a,p_\bullet) \mapsto (a,p_1)$ and $f_\bullet \mapsto f_1$
which is not necessarily either full or faithful. 
\end{theorem}
\begin{proof}
Since $\C$ is pretriangulated,  so is
$T\text{-}\noddinf$ by 
\cite{AnnoLogvinenko-UnboundedTwistedComplexes}, Cor. 5.13. 
Since acyclicity is stable under taking
convolutions of bounded twisted complexes, so is $H$-unitality. 
Since $T$ is strongly homotopy unital, $H$-unitality is equivalent 
to homotopy unitality by 
Theorem \ref{theorem-tfae-unitality-conditions-for-A-modules}. 
We conclude that $\eilmoorwk_T$ is a pretriangulated subcategory 
of $T\text{-}\noddinf$. Hence $H^0(\eilmoorwk_T)$ is
triangulated, as desired. 

The fact that the functor $F$ is well-defined follows readily from 
the definitions. To see that the resulting functor is a priori 
neither full nor faithful, let $k$ be a field and let $\C = \modk$, 
the DG category of DG $k$-modules.  Let $R$ be the polynomial algebra
$k[x,y]$ and let $T = (-) \otimes_k R$. We are thus in the setup of 
\S\ref{section-examples-dg-and-ainfty-algebras}. As explained there, 
$H^0(\eilmoorwk_T)$ is the usual derived category $D(R)$
of complexes of $R$-modules.

On the other hand, $H^0(\C)$ is 
$\modd^{gr}\text{-}k$ the category of graded $k$-modules. 
The monad $H^0(T)$ is tensoring with $R$ in every
degree, thus $\eilmoor_t$ is the category 
$\modd^{gr}\text{-}R$ of graded $R$-modules. 
The functor $F\colon D(R) \rightarrow \modd^{gr}\text{-}R$  
is taking cohomologies of the complex. It is well-known to
be neither full nor faithful: 
\begin{equation}
\begin{tikzcd}[row sep = 0.25cm, column sep = 1cm]
k[x,y]
\ar{r}{x}
\ar{d}{x=0}
&
\underset{\degzero}{k[x,y]}
\\
k[y],
&
\end{tikzcd}
\end{equation}
is a non-trivial morphism in $D(R)$ which is zero on the level of
cohomologies, while
\begin{align*}
k[x,y] \oplus k[x,y] \xrightarrow{x \oplus y}
\underset{\degzero}{k[x,y]} \quad\quad \text{ and }  \quad\quad
k[x,y] \xrightarrow{0} \underset{\degzero}{k}
\end{align*}
are complexes non-isomorphic in $D(R)$ which have the same cohomologies,
i.e. the same image under $F$. Any morphism of complexes
lifting the identity map of that image would have to be a
quasi-isomorphism, i.e. an isomorphism in $D(R)$.  
\end{proof}

The $1$-morphism categories of $\DGFuntwocat$ are not pretriangulated
and cocomplete. For a small DG category $\C$ the
DG category $\DGFun(k,\C) \simeq \C$ isn't cocomplete. To construct
the derived category of an $\Ainfty$-monad, we need to embed 
$\DGFuntwocat$ into a bicategory $\B$ satisfying the assumptions
in \S\ref{section-the-setting} and \S\ref{section-the-derived-category}. 
In the spirit of the standard Yoneda embedding, 
we want to embed each $1$-morphism category
$\DGFun(\C,\D)$ into the continuous DG functor category 
$\DGFun_{cts}(\modC, \modD)$. By
\cite[Theorem 7.2]{Toen-TheHomotopyTheoryOfDGCategoriesAndDerivedMoritaTheory} 
every such continuous functor is homotopy equivalent to tensoring with
some $\C$-$\D$ bimodule. Thus we work with the following bicategory:
\begin{defn}
Let $\DGModtwocat$ be the following DG bicategory:
\begin{itemize}
\item The objects are small DG categories.  
\item The $1$-morphism DG categories are defined by
$$ \homm_{\DGModtwocat}(\C,\D) = \C{-}\modd\text{-}\D \quad
\quad \forall\; \text{ small DG-categories } \C,\D.$$ 
\item $1$-composition is given by the tensor product of bimodules. 
\item For any $\C$ its identity $1$-morphism is the diagonal
bimodule $\C$.
\item Associator and unitor isomorphisms are the maps induced by the universal 
property of tensor product.  
\end{itemize}
\end{defn}

DG bicategory $\DGModtwocat$ satisfies all assumptions in 
\S\ref{section-the-setting} and
\S\ref{section-the-derived-category} and we set $\B =
\DGModtwocat$. We have the $2$-fully faithful $2$-functor 
$\DGFuntwocat \hookrightarrow \DGModtwocat$
which is $\id$ on objects and on $1$-morphisms it is the
functor 
$\leftidx{_{(-)}}{\D} \colon \DGFun(\C,\D) \rightarrow \CmodD$
which sends any $F\colon \C \rightarrow \D$ to 
the extension of scalars bimodule $\leftidx{_F}{\D} = \homm_\D(-,F-)$.
It sends an $\Ainfty$-monad $T$ in $\DGFun(\C,\C)$ 
to an $\Ainfty$-algebra $\leftidx{_T}{\C}$ in $\CmodC$, which can 
be viewed as an $\Ainfty$-monad structure on $T^*: \modC
\rightarrow \modC$. Similarly, $\Ainfty$-modules over 
$\leftidx{_T}{\C}$ in $\DGModtwocat$ 
can be viewed as $\Ainfty$-modules over $T^*$ in $\modC$. 

The rest of the theory in \S\ref{section-the-derived-category}
now applies. In particular, by 
Prop.~\ref{prps-when-B-is-modA-we-get-the-classical-derived-category}
the derived category $D(T)$ of $T$ in the sense of
\S\ref{section-the-derived-category} is equivalent to the usual
derived category $D(\kleisli(T))$ of the $\Ainfty$-category $\kleisli(T)$. 
The resulting fully-faithful embedding $H^0(\eilmoorwk_T)
\hookrightarrow D(\kleisli(T))$ generalises the embedding
of the Eilenberg-Moore category of a monad into the category of
modules over its Kleisli category \cite[\S5]{Street-TheFormalTheoryOfMonads}. 

Finally, there is a version of this whole theory of $\Ainfty$-monads
and their modules for big categories. There 
we set both $\A$ and $\B$ to be the strict $2$-category of strongly
pretriangulated, cocomplete DG categories and continuous functors
between them. 

\subsection{$\Ainfty$-modules over the identity and strong
homotopy idempotents}
\label{section-ainfty-modules-over-the-identity-functor}

The notion of an $\Ainfty$-module over a DG monad described in
\S\ref{section-classical-and-ainfty-monads}
gives non-trivial results even when applied to the trivial monad, 
i.e. the identity functor. 

For an additive category $\C$ the Eilenberg-Moore category 
over $\eilmoor_{\id_\C}$ is $\C$. The non-unital $\id_\C$-modules  
are the idempotents of $\C$, however their category 
is not the Karoubi-comletion of $\C$. Given $a,b \in \C$
with idempotents $p$ and $q$, the morphisms in the former are
$f\colon a \rightarrow b$ such that $fp = qf$ and in the latter ––
such that $fp = qf = f$. 

Let $\C$ be a small DG category. $\Ainfty$-modules 
over the identity monad $\id_C$ in the sense of our 
\S\ref{section-ainfty-structures-in-monoidal-dg-categories} are 
the \em $\Ainfty$-idempotents \rm which appear in 
\cite[\S4]{GorskyHogancampWedrich-DerivedTracesOfSoergelCategories}. 
These are collections $(a,p_\bullet)$ with $a \in \C$ and 
degree $2-i$ endomorphisms $p_i: a \rightarrow a$ for $i \geq 2$ 
such that the bar-construction below is a twisted complex:
\begin{equation}
\label{eqn-bar-construction-of-Ainfty-Id-module}
\begin{tikzcd}[column sep = 2.0cm]
\dots
\ar{r}{p_2 - \id_a}
\ar[dashed, bend left=25]{rrrr}[description]{p_5}
\ar[dashed, bend left=20]{rrr}[description]{p_4}
& 
a
\ar{r}[']{p_2}
\ar[dashed, bend left=20]{rrr}[description]{p_4}
\ar[dashed, bend left=15]{rr}[description]{(- p_3)}
& 
a
\ar[dashed, bend left=15]{rr}[description]{p_3}
\ar{r}[']{p_2 - \id_a}
&
a
\ar{r}[']{p_2}
&
\underset{\degzero}{a}. 
\end{tikzcd}
\end{equation}
Thus $p_2$ is a homotopy idempotent of $a$ in $\C$, 
$p_3$ is the homotopy up to which $p_2^2 = p_2$ in $\C$, 
and $p_{\geq 4}$ are the higher homotopies involved.

The Kleisli category of $\id_\C$ is $\C$ itself and the
$\Ainfty$-functor $\kleisli(\id_\C) \rightarrow \id_\C\text{-}\free$ 
of Defn.~\ref{defn-kleisli-to-free-Ainfty-functor}
coincides with the natural non-full inclusion 
$\C \rightarrow \id_\C\text{-}\free$ given by the $\free$ functor. 
By Theorem \ref{theorem-ainfty-quasi-equivalence-from-kleisli-to-free}
this inclusion is therefore a quasi-equivalence. 

A module $(a,p_\bullet)$ over $\id_\C$ is homotopy unital if $p_2$ is
homotopic to $\id_a$. By the Homotopy Lemma, the morphism 
$\pi_2\colon \free(a) = (a,\id) \rightarrow  (a,p_\bullet)$ is then
a homotopy equivalence. 
Thus $\eilmoorwk_{\id_\C}$ is quasi-equivalent to 
$\id_\C\text{-}\free$ and hence to $\C$. 

As above, we can also consider 
the category $\id_\C\text{-}\noddinf$ of all (not necessarily unital)
$\Ainfty$-modules over $\id_\C$. Its objects are
$\Ainfty$-idempotents of $\C$, but it is not the universal homotopy 
Karoubi completion of $\C$ constructed in 
\cite[\S4]{GorskyHogancampWedrich-DerivedTracesOfSoergelCategories}.
In the latter, the morphisms are all twisted complex morphisms
between the respective bar-constructions, while in $\id_\C\text{-}\noddinf$
it only those twisted complex morphisms which are bar-constructions
of $\Ainfty$-morphism data. The morphisms constructed in
\cite[\S4]{GorskyHogancampWedrich-DerivedTracesOfSoergelCategories}
which show that $(a,p_\bullet)$ is the homotopy image of the homotopy
idempotent $p_2\colon a \rightarrow a$ are not bar-constructions of
$\Ainfty$-morphisms. 

Indeed, suppose $(a,p_\bullet)$ is an image of an 
idempotent $p_2 \colon a \rightarrow a$ in
$H^0(\id_\C\text{-}\noddinf)$. 
Applying $\forget$, $a$ is an image of the idempotent 
$p_2\colon a \rightarrow a$ in $H^0(\C)$. Hence either $a$ is 
a non-trivial retract of itself in $H^0(\C)$, or 
$p_2$ is a homotopy equivalence. Thus we see that we do
not, in general, get $(a,p_\bullet)$ to be the homotopy image of $p_2$. 

On the other hand,
$\infbarres(a,p_\bullet)$ is always the homotopy image of $p_2$
in $\pretriagmns \C$ by
\cite[\S4]{GorskyHogancampWedrich-DerivedTracesOfSoergelCategories}. 
This illustrates our 
Theorem \ref{theorem-tfae-unitality-conditions-for-A-modules}:
$\infbarres(a,p_\bullet) \xrightarrow{\rho} (a,p_\bullet)$
is a homotopy equivalence if and only if $(a, p_\bullet)$ is homotopy 
unital. To summarise:
\begin{theorem}
Let $\C$ be a DG category. The full subcategory of $\pretriagmns \C$
supported on the image of $\id_\C\text{-}\noddinf$ under the
bar-resolution functor $\infbarres(-)$ is the universal homotopy
Karoubi completion of $\C$ constructed in 
\cite[\S4]{GorskyHogancampWedrich-DerivedTracesOfSoergelCategories}. 
\end{theorem}

\subsection{Enhancing monads over enhanced triangulated categories}
\label{section-enhanced-monads}

Another application of our theory in
\S\ref{section-ainfty-structures-in-monoidal-dg-categories}-\ref{section-the-derived-category} is to enhancing exact monads over triangulated categories. 
For the details on the notion of a DG enhanced triangulated 
category, see \cite{BondalKapranov-EnhancedTriangulatedCategories}
\cite[\S1]{LuntsOrlov-UniquenessOfEnhancementForTriangulatedCategories}
\cite[\S3]{AnnoLogvinenko-SphericalDGFunctors},
\cite[\S4.4]{GyengeKoppensteinerLogvinenko-TheHeisenbergCategoryOfACategory}. 

A \em DG enhancement \rm of a triangulated category $C$ is 
a pretriangulated DG category $\C$ with an exact equivalence 
$C \simeq H^0(\C)$. A \em Morita enhancement \rm of $C$ is a small DG
category $\C$ with an exact equivalence $C \simeq D_c(\C)$. Note that
for any $\C$ its compact derived category $D_c(\C)$ is Karoubi
complete, so Morita enhancements can only be used for Karoubi complete
triangulated categories. 

We only work with Morita enhancements, thus for us a 
\em Karoubi complete enhanced triangulated category $\C$ \rm is 
a small DG category. Its underlying triangulated category 
is derived category $D_c(\C)$. We refer the reader to
\cite[\S4.4]{GyengeKoppensteinerLogvinenko-TheHeisenbergCategoryOfACategory}
for an overview of how the fundamental results in 
\cite{Toen-TheHomotopyTheoryOfDGCategoriesAndDerivedMoritaTheory} 
and \cite{Tabuada-InvariantsAdditifsDeDGCategories} imply that 
Karoubi complete enhanced triangulated categories form naturally 
$1$-category $\mor(\DGCat^1)$ which can then be refined to 
a strict $2$-category $\enhcatkc$. However, 
as explained in \S\ref{section-classical-and-ainfty-monads}, to define
an enhanced monad it is not enough to fix an algebra structure on 
a $1$-endomorphism $T$ in $\enhcatkc$. This doesn't give enough data
to construct an Eilenberg-Moore category which would be enhanced 
triangulated. We need a DG enhancement of $\enhcatkc$ and a lift of $T$ 
to an $\Ainfty$-algebra structure in this enhancement. 

The following DG enhancement of $\enhcatkc$ was constructed 
in \cite[\S4.4]{GyengeKoppensteinerLogvinenko-TheHeisenbergCategoryOfACategory}
with the notions developed in
\cite{AnnoLogvinenko-BarCategoryOfModulesAndHomotopyAdjunctionForTensorFunctors}. 
The \em bar category $\CmodbarD$ of bimodules \em is
a DG category isomorphic the category of DG $\C$-$\D$-bimodules 
with $\Ainfty$-morphisms. However, it has an intrinsic
definition and is simpler to work with. We refer the reader to 
\cite[\S3]{AnnoLogvinenko-BarCategoryOfModulesAndHomotopyAdjunctionForTensorFunctors} for the definition and the technical details. 
\begin{defn}[Defn.~4.19,
\cite{GyengeKoppensteinerLogvinenko-TheHeisenbergCategoryOfACategory}]
\label{defn-enhcatkcdg}
The \em homotopy unital DG bicategory $\enhcatkcdg$ of 
Karoubi complete enhanced triangulated categories \rm comprises: 
\begin{enumerate}
\item The \em objects \rm are all small DG categories, 
\item For any objects $\C,\D$, 
the \em $1$-morphism category \rm $\enhcatkcdg(\C,\D)$ is
the full subcategory of $\CmodbarD$ comprising
$\D$-perfect bimodules,
\item For any objects $\C,\D,\E$, the \em
$1$-composition functor \em is 
\begin{align*}
\enhcatkcdg(\D,\E) \otimes \enhcatkcdg(\C,\D) &\rightarrow \enhcatkcdg(\C,\E) \\
(M,N) \mapsto N \bartimes M \quad &\text{and} \quad
(f,g) \mapsto (-1)^{|f||g|} g \bartimes f, 
\end{align*}
where $M,N$ are objects and $f,g$ are morphisms. 
\item For any object $\C$, its \em identity
$1$-morphism \rm is the diagonal bimodule $\C$,
\item The \em associator \rm isomorphism is the natural isomorphism
$$(L \bartimes_{\C} M) \bartimes_{\D} N \simeq L \bartimes_{\C} (M
\bartimes_{\D} N).$$
\item  The \em unitor \rm morphisms are the homotopy equivalences 
of
\cite[\S3.3]{AnnoLogvinenko-BarCategoryOfModulesAndHomotopyAdjunctionForTensorFunctors}:
$$ {\alpha_\C} \colon \C \bartimes_\C M \rightarrow M 
\quad \quad \text{and} \quad \quad 
{\alpha_\D} \colon M \bartimes_\D \D \rightarrow M. $$ 
\end{enumerate}
\end{defn}

We thus set $\A$ to be $\enhcatkcdg$. It is only homotopy unital: the unitor
morphisms $\alpha_\bullet$ are homotopy equivalences, and not
isomorphisms. But they have natural homotopy inverses $\beta_\bullet$ which are moreover their genuine inverses on the right, see 
\cite[\S3.3]{AnnoLogvinenko-BarCategoryOfModulesAndHomotopyAdjunctionForTensorFunctors}. This weakening doesn't affect 
the theory of $\Ainfty$-structures developed in \S\ref{section-ainfty-structures-in-monoidal-dg-categories}-\S\ref{section-the-derived-category}.  
on its $1$-morphisms. The only additional subtlety is that when working with
the unital structures in \S\ref{section-strong-homotopy-unitality} one
must never suppress the identity $1$-morphism which is the source of 
the unit $2$-morphism $\eta\colon \id_\A \rightarrow A$ of the
$\Ainfty$-algebra in question. For example, when defining strong
homotopy unitality in Defn.~\ref{defn-strong-homotopy-unitality}, 
$h^r_\bullet$ is a degree $-1$ morphism $\id_\A \circ_1 A \rightarrow A$
in $\nodA$ satisfying $\mu_2 \circ_2 (\eta \circ_1 A) = \alpha +
dh^r_\bullet$ where $\alpha$ is the unitor $\id_\A
\circ_1 A \rightarrow A$, and similarly for $h^l_\bullet$. 

\begin{defn}
Let $\C$ be an enhanced triangulated category. An \em enhanced exact
monad \rm over $\C$ is a strongly homotopy unital $\Ainfty$-algebra
in $\homm_{\enhcatkcdg}(\C,\C)$. 

Explicitly, it is a collection 
$(T,p_\bullet, \eta, h^l_\bullet, h^r_\bullet)$ where:
\begin{itemize}
\item $T$ is a right-perfect $\C$-$\C$-bimodule, i.e. 
an enhanced exact functor $\C \rightarrow \C$,
\item $p_\bullet$ is a collection of degree $2-i$ morphisms 
$p_i\colon T^i \rightarrow T$ in $\CmodbarC$ such that their 
bar-construction \eqref{eqn-nonaugmented-bar-construction-of-A-m_i} is 
a twisted complex.
\item $\eta$ is a morphism $\C \rightarrow T$ in $\CmodbarC$, 
\item $h^l_\bullet$ and $h^r_\bullet$ are the degree $-1$ 
morphisms $T \rightarrow T$ in $T$-$\noddinf$ and $\noddinf$-$T$
such that the strong homotopy unitality conditions 
\eqref{eqn-strong-homotopy-unitality-conditions} hold. 
\end{itemize}

The underlying exact monad of $(T,p_\bullet)$ is 
$t\colon D_c(\C) \rightarrow D_c(\C)$ given by $(-) \bartimes_\C T$
with the operation $t^2 \rightarrow t$ given by $p_2$ and 
the unit $\id \rightarrow t$ given by $\eta$.  
\end{defn}

To define the derived categories of $\Ainfty$-algebras in
$\enhcatkcdg$, we set $\B$ to be:
\begin{defn}
The \em homotopy unital DG bicategory
$\BarModtwocat$ of bar categories of bimodules \rm is 
defined identically to $\enhcatkc$ above, only with 
the $1$-morphism category $\homm_{\BarModtwocat}(\C,\D)$ being
the whole of $\CmodbarD$. 
\end{defn}
It is readily checked that $\BarModtwocat$ 
satisfies our assumptions in \S\ref{section-the-setting} and 
\S\ref{section-the-derived-category}. 

\begin{defn}
Let $T$ be an enhanced monad over an enhanced triangulated category
$\C$. The \em Eilenberg-Moore category $\eilmoor_T$ \rm is the enhanced 
triangulated category $\noddinfhu\text{-}T^{\barperfC}$ of homotopy 
unital $\Ainfty$-$T$-modules in 
$\barperfC = \homm_{\enhcatkcdg}(k,\C)$. 
\end{defn}

Considering $\Ainfty$-$T$-modules in $\barperfC$ 
makes sense since we work in Morita setting and the triangulated 
category underlying $\C$ is $D_c(\C) \simeq H^0(\barperfC)$. 
Indeed, $\C$ and $\barperfC$
are Morita equivalent, but $\barperfC$ enhances
$D_c(\C)$ both in the Morita sense and in the classical one. 
\begin{lemma}
Let $T$ be an enhanced monad over an enhanced triangulated category
$\C$.
DG category $\eilmoor_T$ is pretriangulated and homotopy Karoubi complete.  
\end{lemma}
It follows that $D_c(\eilmoor_T) \simeq H^0(\eilmoor_T)$. Thus 
when viewed as enhanced triangulated category, 
the underlying triangulated category of $\eilmoor_T$ is $H^0(\eilmoor_T)$.
\begin{proof}
Since $\barperfC$ is pretriangulated, so is 
$\noddinf\text{-}T^{\barperfC}$ by 
\cite{AnnoLogvinenko-UnboundedTwistedComplexes}, Cor.~5.13. 
Since $\noddinf\text{-}T^{\modbarC}$ is cocomplete, it is homotopy
Karoubi complete. Hence every homotopy idempotent of 
$\noddinf\text{-}T^{\barperfC}$ splits in
$\noddinf\text{-}T^{\modbarC}$. But if $(b,q_\bullet)$ is a homotopy
direct summand of some $(a,p_\bullet) \in
\noddinf\text{-}T^{\barperfC}$, then $b$ is homotopy direct summand 
of $a$ in $\modbarC$. Since $a$ is perfect, 
$(b, q_\bullet)$ lies in $\noddinf\text{-}T^{\barperfC}$. 
Thus $\noddinf\text{-}T^{\barperfC}$ is pretriangulated and
homotopy Karoubi complete. Hence so is 
$\eilmoor_T = \noddinfhu\text{-}T^{\barperfC}$, since $H$-unitality 
is stable under taking convolutions of twisted complexes 
and homotopy direct summands. 
\end{proof}

Let $t\colon D_c(\C) \rightarrow D_c(\C)$ be the exact monad 
$(-) \ldertimes T$ underlying $T$. Recall that  
$\eilmoor_{t}$ is not apriori triangulated or Karoubi complete.
We have a natural functor 
$H^0(\eilmoor_T) \rightarrow \eilmoor_{t}$
which sends  every $(a,p_\bullet) \in H^0(\eilmoor_T)$ to
$(a,p_2)$ and every $f_\bullet$ to $f_1$.

$\barperfC$ generates the whole of $H^0(\modbarC)$. 
By Prop.~\ref{prps-derived-category-of-A-is-triangulated-cocomplete-hull-of-freeA},
the derived category $D(T)$ is $H^0(\noddinfhu\text{-}T^{\modbarC})$.  
By Prop.~\ref{prps-perfect-iff-lies-in-hperf-of-perfect-generator-frees}, 
the objects of $D_c(T)$, i.e. perfect $T$-modules, lie in 
$\eilmoor_T = \noddinfhu\text{-}T^{\barperfC}$, but the converse is not
true, cf.~Remark.~\ref{remark-not-all-perfect-object-modules-are-perfect}.

\begin{defn}
Let $T$ be an enhanced monad over an enhanced triangulated category
$\C$. The \em perfect Eilenberg-Moore category
$\eilmoor^{pf}_T$ of $T$ \rm is the full subcategory of $\eilmoor_T$ 
consisting of perfect $\Ainfty$-$T$-modules. 
\end{defn}

\begin{theorem}
Let $\C$ be an enhanced triangulated category and $T$ be an enhanced
exact monad over $\C$. As a DG category, $\eilmoor^{pf}_T$ 
is pretriangulated, homotopy Karoubi complete, and
contains all free $T$-modules. The full inclusion 
$\free\text{-}T \hookrightarrow \eilmoor^{pf}_T$
is an equivalence of enhanced triangulated categories. The underlying 
triangulated category is further equivalent to $D_c(T)$ and $D_c(\kleisli(T))$. 
\end{theorem}
\begin{proof}
By definition, $\eilmoor^{pf}_T$ comprises all perfect
$\Ainfty$-$T$-modules in $\modbarC$. Hence it 
is pretriangulated and homotopy Karoubi complete. 
Any $a \in \barperfC$ is compact in $\modbarC$, and hence by 
Lemma \ref{lemma-perfect-in-nodhuaA-iff-perfect-in-A}
the free module $Ta$ lies in $\eilmoor^{pf}_T$. 
The remaining assertions now follow by Theorem 
\ref{theorem-compact-derived-category-of-A-is-that-of-frees-and-kleisli}. 
\end{proof}

\subsection{Adjunction monads and comonads}
\label{section-adjunction-monads-and-comonads}

In this section we continue working in the setting of enhanced Karoubi
complete triangulated categories of 
\S\ref{section-enhanced-monads}. We show that for any enhanced
functor its adjunction monad and comonad can be enhanced by 
strictly (co)associative but bimodule homotopy unital algebra and
coalgebra in $\enhcatkcdg$. We then show that therefore for any 
adjoint triple $(L,F,R)$ the monad $RF$ and the comonad $LF$ are
derived module-comodule equivalent. 

Let $\C$ and $\D$ be enhanced triangulated categories, i.e. objects in
$\enhcatkcdg$. Let $F\colon \C \rightarrow \D$ be an enhanced exact
functor, i.e. a $1$-morphism in $\enhcatkcdg$. 
By definition, $F$ is some right-perfect bimodule $M \in \CmodbarD$. 

Let $F$ have an enhanced right adjoint $R\colon \D \rightarrow \C$. 
By this we mean, that $F$ and $R$ are $2$-categorically adjoint in 
the homotopy bicategory $\enhcatkc$ of $\enhcatkcdg$. In \cite[Theorem
4.1,Prop.~4.6]{AnnoLogvinenko-BarCategoryOfModulesAndHomotopyAdjunctionForTensorFunctors}
we showed that such $R$ exists if and only if the bar dual 
$M^{\barD} : = \barhom_{\D}(M,\D)$
is also a right-perfect $\D$-$\C$-bimodule. Moreover, in such case, 
up to homotopy equivalence, we can take $R = M^{\barD}$, the
adjunction counit $FR \xrightarrow{\counit} \id_\D$ to be the evaluation map 
\begin{equation}
\label{eqn-counit-of-F,R-adjunction}
\barhom_{\D}(M,\D) \bartimes_\C M \xrightarrow{\eval} \D,
\end{equation}
and the adjunction unit $\id_\C \xrightarrow{\unit} RF$ to be the composition 
\begin{equation}
\label{eqn-unit-of-F,R-adjunction}
\C \xrightarrow{\action} \barhom_{\D}(M,M) \xrightarrow{\zeta}
M \bartimes_\D \barhom_{\D}(M,\D),
\end{equation}
of the action map and any homotopy inverse $\zeta$ of the natural map 
$$ M \bartimes_\D \barhom_{\D}(M,\D) \xrightarrow{\eta} \barhom_{\D}(M,M). $$
See \cite[\S4.2]{AnnoLogvinenko-BarCategoryOfModulesAndHomotopyAdjunctionForTensorFunctors} for the definitions and technical details. 

Since $F$ and $R$ are homotopy adjoint, the operations 
$$ RFRF \xrightarrow{R\counit{F}} RF 
\quad \quad \text{ and } \quad \quad 
   \id \xrightarrow{unit} RF, $$
give the composition $RF$ the structure of a stricly unital algebra 
in the homotopy bicategory $\enhcatkc$. Similarly, the operations
$$ FR \xrightarrow{F\unit{R}} FRFR
\quad \quad \text{ and } \quad \quad 
FR \xrightarrow{\counit} \id, $$
give the composition $FR$ the structure of strictly counital coalgebra
in $\enhcatkc$. On the level of underlying exact functors these
become the adjunction monad $rf$ and comonad $fr$ of the adjunction 
$(f,r)$.  

As explained in \S\ref{section-enhanced-monads}, the above is not
enough to give $RF$ and $FR$ the structure of an enhanced monad and
comonad. We need to equip $RF$ and $FR$ with full $\Ainfty$- and 
strong/bimodule homotopy unitality structures in $\enhcatkcdg$ which 
descend to the above algebra and coalgebra structures in 
$\enhcatkc$. 

To simplify the computations involved, we modify 
a little the above setup. In
\cite[\S3.3]{AnnoLogvinenko-BarCategoryOfModulesAndHomotopyAdjunctionForTensorFunctors}
we have constructed natural quasi-inverse homotopy equivalences 
\begin{equation*}
\begin{tikzcd}
M \bartimes_\D \D 
\ar[shift left = 0.75]{r}{\alpha}
&
M
\ar[shift left = 0.75]{l}{\beta}
\end{tikzcd}
\quad \text{ and } \quad 
\begin{tikzcd}
M 
\ar[shift left = 0.75]{r}{\gamma}
&
\barhom_\D(\D,M).
\ar[shift left = 0.75]{l}{\delta}
\end{tikzcd}
\end{equation*}
We replace the bimodule $M$ by a homotopy equivalent bimodule 
$\barhom_\D(\D,M)$. This doesn't change the isomorphism class of the
corresponding $1$-morphism in $\enhcatkc$, and
hence doesn't change the enhanced functor. Henceforth we 
use $F$ to denote the $1$-morphism of $\enhcatkcdg$ defined by
$\barhom_\D(\D,M)$. 

Next, we construct the multiplication of $RF$ and 
the comultiplication of $FR$ which descend to the same
$2$-morphisms in $\enhcatkc$, under the identification of $M$ 
and $\barhom_\D(\D,M)$ provided by $\gamma$, as 
the homotopy adjunction multiplication and comultiplication above. 
For this, we need to make the following choice. By
\cite[Lemma
3.36]{AnnoLogvinenko-BarCategoryOfModulesAndHomotopyAdjunctionForTensorFunctors},
the map 
$$ \composition \colon\quad \barhom_\D(\D,M) \bartimes_\D \barhom_\D(M,\D)
\rightarrow \barhom_\D(M,M), $$
is a homotopy equivalence. We choose, once and for all, 
a homotopy inverse 
$$ \zeta\colon\quad \barhom_\D(M,M) \rightarrow \barhom_\D(\D,M)
\bartimes_\D \barhom_\D(M,\D) $$
and degree $-1$ endomorphisms
$$ \omega\colon\quad \barhom_\D(M,M) \rightarrow \barhom_\D(M,M), $$ 
$$ \omega'\colon\quad 
\barhom_\D(\D,M) \bartimes_\D \barhom_\D(M,\D)
\rightarrow 
\barhom_\D(\D,M) \bartimes_\D \barhom_\D(M,\D)$$
such that $d\omega = \composition \circ \zeta - \id$ and 
$d\omega' = \id - \zeta \circ \composition$. 

Since $\enhcatkcdg$ is only homotopy unital, we do not 
suppress the unit $1$-morphisms when they act as sources or images of 
$2$-morphisms we use. For example, $F\id_\C{R}$ denotes the bimodule 
$\barhom_{\D}(M,\D) \bartimes_\C \C \bartimes_\C \barhom_{\D}(\D,M)$
and $F \action R\colon\; F \id_\C R \rightarrow F(\barhom_{\D}(M,M))\R$ 
denotes the map $\id \bartimes \action \bartimes \id$ 
$$ 
\barhom_{\D}(M,\D) \bartimes_\C \C \bartimes_\C \barhom_{\D}(\D,M) 
\rightarrow 
\barhom_{\D}(M,\D) \bartimes_\C \barhom_{\D}(M,M) \bartimes_\C \barhom_{\D}(\D,M). $$

In \cite[\S2.11]{AnnoLogvinenko-BarCategoryOfModulesAndHomotopyAdjunctionForTensorFunctors}
we showed that for any small DG category $\C$ 
the bar-complex $\barC$ is a strictly
coassociative coalgebra in $\CmodC$. In particular,  
for any bimodules $E_1$ and $E_2$ which are  
right and left $\C$-modules
there is a canonical map 
$\beta^n\colon E_1 \bartimes_\C E_2 \rightarrow E_1 \bartimes_\C 
\C^{\bartimes n} \bartimes_\C E_2$
which equals the composition of any $n$ applications of $\beta$ to
any of the tensor factors involved. In our setup, we write e.g.
$F\beta^nR$ for the induced $2$-morphism $FR \rightarrow F \id^n R$
in $\enhcatkcdg$. 

\begin{defn}
\label{defn-multiplication-and-comultiplication-for-RF-and-FR}
Let $\C$ and $\D$ be $DG$-categories and let $M \in \CmodbarD$ be
a right-perfect bimodule such that $M^{\barD}$ is also right-perfect. 
Let $F$ and $R$ be $\barhom_\D(\D,M)$ and $\barhom_{\D}(M,\D)$ considered 
as $1$-morphisms in $\enhcatkcdg$. 

Define $2$-morphism 
$\mu\colon RFRF \rightarrow RF$
to be $(\composition^2)F$
where $\composition$ is the map of taking the composition in $\modbarD$, see
\cite[Defn.~3.11]{AnnoLogvinenko-BarCategoryOfModulesAndHomotopyAdjunctionForTensorFunctors}.
% Explicitly, it is the map 
% \begin{scriptsize}
% \begin{equation*}
% \barhom_{\D}(\D,M) \bartimes_\D \barhom_{\D}(M,\D) \bartimes_\C \barhom_{\D}(\D,M)
% \bartimes_\D \barhom_{\D}(M,\D) 
% \xrightarrow{\id \bartimes \composition^2}
% \barhom_{\D}(\D,M) \bartimes_\D \barhom_{\D}(M,\D),
% \end{equation*} 
% \end{scriptsize}
Define $2$-morphism 
$\Delta\colon FR \rightarrow FRFR$
to be $F(\zeta \circ \action \circ \beta)R$. 
\end{defn}
\begin{prps}
\label{prps-RF-LF-associativity-coassociativity}
Let $\C$ and $\D$ be $DG$-categories and let $M \in \CmodbarD$ be
a right-perfect bimodule such that $M^{\barD}$ is also right-perfect. 
Let $F$ and $R$ be $\barhom_\D(\D,M)$ and $\barhom_{\D}(M,\D)$ considered 
as $1$-morphisms in $\enhcatkcdg$. 

The $2$-morphisms $\mu\colon RFRF \rightarrow RF$ and $\Delta\colon FR
\rightarrow FRFR$ in Defn.~\ref{defn-multiplication-and-comultiplication-for-RF-and-FR}
are a strictly associative multiplication and 
a strictly coassociative comultiplication. 
\end{prps}
\begin{proof}
The multiplication $\mu\colon RFRF \rightarrow RF$ is strictly
associative because the composition in $\modbarD$ is strictly
associative: both $\mu \circ RF\mu$ and
$\mu \circ \mu{RF}$ are the same $2$-morphism 
$(\composition^4)F\colon RFRFRF \rightarrow RF$. 

The comultiplication $\Delta\colon FR \rightarrow FRFR$ is strictly
coassociative because both $FR\Delta \circ \Delta$ and 
$\Delta{FR} \circ \Delta$ are the same $2$-morphism 
$FR \rightarrow FRFRFR$ given by 
$F\bigl(\zeta^2 \circ  \action^2 \circ \beta^2 \bigr) R$. 
\end{proof}

We next show that $RF$ is bimodule homotopy unital, 
see Defn.~\ref{def-bimodule-homotopy-unitality} and more generally
\S\ref{section-unitality-conditions-for-algebras}. We also 
show that $FR$ is bicomodule homotopy counital. Moreover, 
most of the higher components of the bimodule homotopy unit
$\bareta_{\bullet\bullet}$ of $RF$ and bicomodule homotopy counit
$\barepsilon_{\bullet\bullet}$ of $FR$ are zero. 

\begin{remark}
\label{remark-bimodule-homotopy-unitality-via-four-morphisms-for-RF}
By Theorem \ref{theorem-conditions-for-bimodule-homotopy-unitality}, 
bimodule homotopy unitality of $RF$ is equivalent to the existence
of following four morphisms: 
\begin{itemize}
\item $\eta\colon \id \rightarrow RF$ in $\CmodbarC$ with $d\eta = 0$, 
\item $h^r_\bullet\colon \id RF \rightarrow RF$ in $\noddinf\text{-}RF$
such that $d h^r_\bullet = \alpha{F} - \mu \circ \eta RF$, 
\item $h^l_\bullet\colon RF\id \rightarrow RF$ in $RF\text{-}\noddinf$
such that $d h^l_\bullet = R{\alpha} - \mu \circ RF \eta$, 
\item $\kappa_{\bullet\bullet}\colon RF{\id}RF \rightarrow RF$ in  
$RF\text{-}\noddinf\text{-}RF$ such that $d \kappa_{\bullet \bullet} = \mu\circ( h^l_\bullet RF - RF h^r_\bullet).$
\end{itemize}
The conditions on the four morphisms are simpler than in 
the original definition due to $(RF,\mu)$ being a strict algebra. 
\end{remark}
\begin{remark}
\label{remark-bicomodule-homotopy-counitality-via-four-morphisms-for-FR}
By the coalgebra analogue of 
Theorem \ref{theorem-conditions-for-bimodule-homotopy-unitality}, 
bicomodule homotopy counitality of $FR$ is equivalent to the existence
of the following four morphisms:
\begin{itemize}
\item $\epsilon\colon FR \rightarrow \id$ in $\DmodbarD$ with
$d\epsilon = 0$, 
\item $g^r_\bullet\colon FR \rightarrow \id FR$ in $\conoddinf\text{-}FR$
such that $d g^r_\bullet = \beta{R} - \epsilon FR \circ \Delta$, 
\item $g^l_\bullet\colon FR \rightarrow FR\id$ in $FR\text{-}\conoddinf$
such that $d g^l_\bullet = {F}\beta - FR \epsilon \circ \Delta$, 
\item $\lambda_{\bullet\bullet}\colon FR \rightarrow FR{\id}FR$ in  
$FR\text{-}\conoddinf\text{-}FR$ with 
$d \lambda_{\bullet \bullet} = ( g^l_\bullet FR - FR g^r_\bullet)\circ
\Delta.$
\end{itemize}
\end{remark}

To construct these morphisms, we need the following auxilliary definition:

\begin{defn}
\label{defn-eta-epsilon-chi-chi'-lambda}
Let $\C$ and $\D$ be $DG$-categories and let $M \in \CmodbarD$ be
a right-perfect bimodule such that $M^{\barD}$ is also right-perfect. 
Let $F$ and $R$ be $\barhom_{\D}(\D,M)$ and $\barhom_{\D}(M,\D)$ 
considered as $1$-morphisms in $\enhcatkcdg$. 

We define the following $2$-morphisms in $\enhcatkcdg$:
\begin{itemize}
\item $\eta\colon \id \rightarrow RF$ is the 
degree $0$ map $\zeta \circ \action$. 
\item $\epsilon\colon FR \rightarrow \id$ is the 
$\DmodbarD$ degree $0$ map $\delta \circ \composition$. 
\item $\chi\colon F\id \rightarrow F$ 
is the degree $-1$ map 
$\composition \circ (F\omega) \circ (F\action)$. 
\item $\chi'\colon \id{R} \rightarrow R$ i
is the degree $-1$ map
$\composition \circ (\omega{R}) \circ (\action{R})$. 
\item 
$\xi\colon F \rightarrow \id F$ is the degree 
$-1$ map
\begin{align*}
\theta \circ \delta{F} \circ \composition{F}
\circ F\zeta \circ F\action \circ \beta
\;+\;
\beta \circ \composition \circ \kappa{F} \circ 
\composition{F} \circ F\zeta \circ F\action
\circ \beta. 
\end{align*}
Here $\theta$ and $\kappa$ are the natural maps constructed in 
\cite[\S3.3]{AnnoLogvinenko-BarCategoryOfModulesAndHomotopyAdjunctionForTensorFunctors}
such that $d\theta = \id - \beta \circ \alpha$ and $d\kappa = \gamma \circ
\delta - \id$. 

\item 
$\xi'\colon R \rightarrow R\id$ is the $\DmodbarC$ degree $-1$ map
\begin{align*}
\theta \circ R\delta \circ R\composition
\circ \zeta{R} \circ \action{R} \circ \beta
\; + \;
\beta \circ \composition \circ R\kappa \circ 
R\composition \circ \zeta{R} \circ 
\action{R} \circ \beta. 
\end{align*}

\item $\lambda\colon FR \rightarrow FR{\id}FR$
is the degree $-2$ morphism in $\DmodbarD$ defined as follows. 
First, let $\nu\colon FR \rightarrow F(\barhom_\D(M,M)R$ 
be the degree $-2$ map
\begin{equation*}
F\bigl(\composition \circ \omega^2 \circ \action^2 \circ \beta^2 \bigl)R.
\end{equation*}
Its differential $d\nu$ is the composition 
of the closed degree $-1$ morphism  
\begin{equation}
\label{eqn-h^lFR-RFh^r-circ-Delta}
FR \xrightarrow{\Delta}
FRFR \xrightarrow{
FR(\beta \circ \chi \circ \beta + \xi)R - F(\beta \circ \chi' \circ
\beta + \xi')FR }
FR{\id}FR,
\end{equation}
with the homotopy equivalence 
$F\bigl(\composition \circ R\alpha\bigr)R$. 
In \cite[Lemma
4.8(1)]{AnnoLogvinenko-BarCategoryOfModulesAndHomotopyAdjunctionForTensorFunctors}
we showed that a homotopy lift of a composition of 
a closed morphism with a homotopy equivalence induces
a homotopy lift of the morphism itself.  
We define $\lambda$ to be the homotopy lift of
\eqref{eqn-h^lFR-RFh^r-circ-Delta} induced by $\nu$. 
Indeed,  
the homotopy equivalence $\composition \circ R\alpha$
has a homotopy inverse $R\beta \circ \zeta$ and 
$$\upsilon := \quad
R\beta \circ \omega' \circ R\alpha 
+ R\theta$$
is a homotopy between $\id$ and their composition. 
We thus define
\begin{align}
\label{eqn-lambda-definition}
\lambda := \quad F\upsilon{R} \circ \eqref{eqn-h^lFR-RFh^r-circ-Delta} 
+ 
F\Bigl(\composition \circ R\alpha\Bigr)R \circ \nu. 
\end{align}
\end{itemize}
\end{defn}
\begin{prps}
\label{prps-RF-LF-unitality-counitality}
Let $\C$ and $\D$ be $DG$-categories and $M \in \CmodbarD$ be
a right-perfect bimodule with $M^{\barD}$ also right-perfect. 
Let $F$ and $R$ be $\barhom_\D(\D,M)$ and $\barhom_{\D}(M,\D)$ considered 
as $1$-morphisms in $\enhcatkcdg$. 
Let $\mu$,$\Delta$ be as in 
Defn.~\ref{defn-multiplication-and-comultiplication-for-RF-and-FR}
and $\eta$, $\epsilon$, $\chi$, $\chi'$, and $\lambda$ be as in 
Defn.~\ref{defn-eta-epsilon-chi-chi'-lambda}.
 
Setting 
$(\eta,h^r_\bullet, h^l_\bullet, \kappa_{\bullet\bullet})$
in Remark \ref{remark-bimodule-homotopy-unitality-via-four-morphisms-for-RF}
to be $(\unit,-\chi'{F}, -R\chi, 0)$ defines 
a bimodule homotopy unital structure on the strict algebra $(RF,\mu)$. 
Setting $(\epsilon, g^r_\bullet, g^l_\bullet,
\lambda_{\bullet\bullet})$ 
in Remark \ref{remark-bicomodule-homotopy-counitality-via-four-morphisms-for-FR}
to be $(\counit, -(\beta \circ \chi \circ \beta + \xi){R}, -{F}(\beta \circ
\chi' \circ \beta +\xi'), \lambda)$ a bicomodule homotopy 
counital structure on the strict coalgebra $(FR,\Delta)$.  
\end{prps}
Here we view $\chi'F$, $R\chi$, $(\beta \circ \chi + \xi){R}$, 
and ${F}(\beta \circ \chi'+\xi')$ as strict 
$\Ainfty$-morphisms in $\noddinf\text{-}RF$, $RF\text{-}\noddinf$,
$\conoddinf\text{-}FR$, and $FR\text{-}\conoddinf$. 
\begin{proof}
The maps $\eta$ and $\epsilon$ are closed of degree $0$ 
by their definitions in Defn.~\ref{defn-eta-epsilon-chi-chi'-lambda}. 
Hence conditions $d\eta = 0$ and $d\epsilon = 0$ hold.  

Recall the definition of the comultiplication $\Delta$ in
Defn.~\ref{defn-multiplication-and-comultiplication-for-RF-and-FR}. 
Since $\beta \bartimes \id$ in it can be equally written as $\id \bartimes \beta$, we can write $\Delta$ both as a morphism 
$F(R \rightarrow RFR)$ and $(F \rightarrow FRF)R$. It follows
that $(\beta \circ {\chi} \circ \beta + \xi)R$ and $F(\beta \circ
\chi' \circ \beta + \xi')$ are strict morphisms of right and
left $FR$-comodules, respectively. We conclude that the differentials of 
$-(\beta \circ \chi \circ \beta + \xi)R$ and 
$-F(\beta \circ \chi' \circ \beta + \xi')$ as $\Ainfty$-morphisms of right 
and left $FR$-comodules are just their differentials as morphisms 
in $\DmodbarD$. 

Similarly, by the definition of the multiplication $\mu$  in
Defn.~\ref{defn-multiplication-and-comultiplication-for-RF-and-FR}
it is a morphism $(RFR \rightarrow R)F$, whence $R\chi$ 
is a strict morphisms of left $RF$-modules. Finally, 
we check by hand that $\chi'F$ 
is a strict morphism of right $RF$-modules. Indeed, 
both $\mu \circ \chi'FRF$ and $\chi'F \circ \mu$ equal the following
morphism ${\id}RFRF \rightarrow RF$ in $\CmodbarC$
\begin{equation*}
(\composition^3) F\circ (\omega \circ \action)RFRF.
\end{equation*}
We conclude that the differentials of $-\chi'F$ and $ - \chi R$ 
as $\Ainfty$-morphisms of right and left $FR$-modules
are the same as their differentials as morphisms in $\CmodbarC$. 

It can now be readily verified that we have in $\CmodbarC$ and $\DmodbarD$ 
\begin{align}
\id{RF} \xrightarrow{\eta{RF}} RFRF \xrightarrow{\mu} RF \quad &= 
\alpha{F} + d(\chi'{F}), \\
RF\id \xrightarrow{{RF}\eta} RFRF \xrightarrow{\mu} RF \quad &= 
R\alpha + d(R\chi)), \\
FR \xrightarrow{\Delta} FRFR \xrightarrow{\epsilon{FR}} \id{FR} \quad &=
\beta{R} + d((\beta \circ \chi \circ \beta){R} + \xi{R}), \\
\label{eqn-Delta-FR-epsilon}
FR \xrightarrow{\Delta} FRFR \xrightarrow{{FR}\epsilon} FR\id \quad &=
F\beta + d(F(\beta \circ \chi' \circ \beta) + F\xi). 
\end{align}
It follows that
\begin{itemize}
\item  
$d h^r_\bullet = \id - \mu \circ \eta RF$ in $\noddinf\text{-}RF$, 
\item 
$d h^l_\bullet = \id - \mu \circ RF \eta$ in $RF\text{-}\noddinf$,
\item
$d g^r_\bullet = \id - \epsilon FR \circ \Delta$ in
$\conoddinf\text{-}FR$, 
\item
$d g^l_\bullet = \id - FR \epsilon \circ \Delta$ in 
$FR\text{-}\conoddinf$. 
\end{itemize}

Finally, we verify the conditions 
\begin{itemize}
\item $\mu\circ( h^l_\bullet RF - RF h^r_\bullet) = 0$,
\item $(g^l_\bullet FR - FR g^r_\bullet)\circ \Delta = d \lambda_{\bullet \bullet}.$
\end{itemize}
For the former, by associativity of the composition 
in $\modbarD$ both $\mu \circ R\chi{RF}$ and 
$\mu \circ RF\chi'F$ equal the same $\CmodbarC$ given by
map ${RF}\id{RF} \rightarrow RF$
\begin{equation*}
(\composition^3)F \circ \bigl(RF(\omega \circ \action)RF\bigr).
\end{equation*}

For the latter,
by definition of $\lambda$ we have in $\DmodbarD$
$$ d\lambda = (-F(\beta \circ \chi' \circ \beta + \xi')FR + FR(\beta
\circ \chi \circ \beta + \xi)R) \circ \Delta. $$
Next, observe that both summands in the definition \eqref{eqn-lambda-definition}
of $\lambda$ are of the form 
$$ FR \xrightarrow{F\beta = \beta R} F{\id}R \xrightarrow{ 
F(...)R} FR{\id}FR $$
and hence are morphisms of $FR$-$FR$
bicomodules. Thus the differential of $\lambda$ as an
$\Ainfty$-morphism of $FR$-$FR$-bicomodules is the same as its
differential as a morphism in $\DmodbarD$. Thus
$d\lambda_{\bullet\bullet} = (g^l_\bullet FR - FR
g^r_\bullet)\circ\Delta$, as desired.
\end{proof}

We thus have the first main result of this section.  It tells us that
for any adjunction of enhanced exact functors, the adjunction monad
and comonad of their underlying exact functors can be enhanced by an
enhanced monad and comonad which are strictly (co)associative and
bi(co)module homotopy (co)unital:
\begin{theorem}
\label{theorem-bimodule-homotopy-unitality-for-enhanced-adjunction-monads-and-comonads}
Let $\C,\D$ be enhanced triangulated categories, 
$F\colon \C \rightarrow \D$ an enhanced exact functor, and 
$R$ its right adjoint. Replace $F$ and $R$ by equivalent 
functors $\barhom_\D(\D,M)$ and $\barhom_{\D}(M,\D)$ where
$M \in \CmodbarD$ is the bimodule defining $F$.  

Then $RF$ with the multiplication of  
Prps.~\ref{prps-RF-LF-associativity-coassociativity} and the unital 
structure of Prps.~\ref{prps-RF-LF-unitality-counitality} is a
strictly associative, bimodule homotopy unital enhanced monad. 
$FR$ with the comultiplication of  
Prps.~\ref{prps-RF-LF-associativity-coassociativity} and the counital 
structure of Prps.~\ref{prps-RF-LF-unitality-counitality} is a
strictly coassociative, bicomodule homotopy counital enhanced comonad. 
\end{theorem}

It follows that for an adjoint triple $(L,F,R)$ of enhanced functors
monad $RF$ and comonad $LF$ are derived module-comodule equivalent:

\begin{theorem}[Module-comodule correspondence for enhanced adjoint
triples] 
\label{theorem-module-comodule-correspondence-for-enhanced-adjoint-triples}

Let $\C,\D$ be enhanced triangulated categories, 
$F\colon \C \rightarrow \D$ an enhanced exact functor, and 
$L,R$ its left and right adjoints. Replace $L$, $F$ and $R$ by 
homotopy equivalent $1$-morphisms $\barhom_\C(M,\C)$, $\barhom_\C(\C,M)$ 
and $\barhom_{\D}(M,\D)$ where
$M \in \CmodbarD$ is the bimodule defining $F$. 

Then $RF$ with the structure of a strongly homotopy unital enhanced monad 
of Prop.~\ref{prps-LF-and-RF-are-homotopy-adjoint-in-a-monoidal-category}
and $LF$ with the structure of a bicomodule homotopy counital
enhanced comonad of Theorem
\ref{theorem-bimodule-homotopy-unitality-for-enhanced-adjunction-monads-and-comonads} are derived module-comodule equivalent:
$$ D_c(RF) \simeq D_c(LF). $$
\end{theorem}
\begin{proof}
By Prop.~\ref{prps-LF-and-RF-are-homotopy-adjoint-in-a-monoidal-category},
the monad $RF$ and the comonad $LF$ are homotopy adjoint. Since $RF$ 
is strongly homotopy unital and $LF$ is bicomodule
homotopy counital, they are derived module-comodule equivalent by Theorem
\ref{theorem-module-comodule-correspondence-in-a-monoidal-dg-category}. 
\end{proof}

\bibliography{references}
\bibliographystyle{amsalpha}
\end{document}